\documentclass[11pt]{article}
\usepackage{amsmath}
\usepackage{amsfonts}
\usepackage{amssymb}
\usepackage{amsthm}
\usepackage{graphicx}
\usepackage{empheq}
\usepackage{indentfirst}
\usepackage{cite} 
\usepackage{mathrsfs}
\usepackage{cases}
\usepackage{graphics}
\usepackage{xcolor}  
\usepackage{bm}
\usepackage{makeidx}
\usepackage[T1]{fontenc}
\newtheorem{theorem}{\textbf{Theorem}}[section]
\newtheorem{lemma}{\textbf{Lemma}}[section]
\newtheorem{proposition}{\textbf{Proposition}}[section]
\newtheorem{corollary}{\textbf{Corollary}}[section]
\newtheorem{remark}{\textbf{Remark}}[section]
\newtheorem{definition}{\textbf{Definition}}[section]
\allowdisplaybreaks[4] 
\def\be{\begin{equation}}
\def\ee{\end{equation}}
\def\bea{\begin{eqnarray}}
\def\eea{\end{eqnarray}}
\def\bt{\begin{theorem}}
\def\et{\end{theorem}}
\def\bl{\begin{lemma}}
\def\el{\end{lemma}}
\def\br{\begin{remark}}
\def\er{\end{remark}}
\def\bp{\begin{proposition}}
\def\ep{\end{proposition}}
\def\bc{\begin{corollary}}
\def\ec{\end{corollary}}
\def\bd{\begin{definition}}
\def\ed{\end{definition}}
\def\vp{\varphi}
\def\non{\nonumber}
\begin{document}

\title{Well-posedness and Long-time Behavior of \\
a Bulk-surface Coupled Cahn-Hilliard-diffusion System \\
with Singular Potential for Lipid Raft Formation}

\author{
Hao Wu
\footnote{
Corresponding author.
School of Mathematical Sciences,
Fudan University,
Handan Road 220, Shanghai 200433, P.R. China.
Email: \texttt{haowufd@fudan.edu.cn}
}
\ \
and
\ \
Shengqin Xu
\footnote{
School of Mathematical Sciences,
Fudan University,
Handan Road 220, Shanghai 200433, P.R. China.
Email: \texttt{20110840009@fudan.edu.cn}
}
}
\date{\today}
\maketitle


\begin{center}
\textit{Dedicated to Prof. Pierluigi Colli on the occasion of his 65th birthday \\
with friendship and admiration.} \medskip
\end{center}


\begin{abstract}
\noindent We study a bulk-surface coupled system that describes the processes of lipid-phase separation and lipid-cholesterol interaction on cell membranes, in which cholesterol exchange between cytosol and cell membrane is also incorporated. The PDE system consists of a surface Cahn-Hilliard equation for the relative concentration of saturated/unsaturated lipids and a surface diffusion-reaction equation for the cholesterol concentration on the membrane, together with a diffusion equation for the cytosolic cholesterol concentration in the bulk. The detailed coupling between bulk and surface evolutions is characterized by a mass exchange term $q$. For the system with a physically relevant singular potential, we first prove the existence, uniqueness and regularity of global weak solutions to the full bulk-surface coupled system under suitable assumptions on the initial data and the mass exchange term $q$. Next, we investigate the large cytosolic diffusion limit that gives a reduction of the full bulk-surface coupled system to a system of surface equations with non-local contributions. Afterwards, we study the long-time behavior of global solutions in two categories, i.e., the equilibrium and non-equilibrium models according to different choices of the mass exchange term $q$. For the full bulk-surface coupled system with a decreasing total free energy, we prove that every global weak solution converges to a single equilibrium as $t\to +\infty$. For the reduced surface system with a mass exchange term of reaction type, we establish the existence of a global attractor.
\end{abstract}

\section{Introduction}
\setcounter{equation}{0}
\noindent
Let $\Omega \subset\mathbb{R}^3$ be a bounded domain with smooth boundary $\Gamma:=\partial\Omega$ and $T>0$. We consider the following bulk-surface coupled system on the formation of lipid rafts \cite{GKRR16,AK20}:
\begin{subequations}
	\begin{alignat}{3}
	& \partial_t u=D\Delta u,
    & \qquad\qquad  \text{in}\ \Omega\times(0,T), \label{1.a}\\
    & D\partial_{\bm{n}} u=-q,
    &  \qquad\qquad\text{on}\ \Gamma\times(0,T), \label{1.b}\\
    &\partial_t\varphi =\Delta_\Gamma \mu,
    &  \qquad\qquad\text{on}\ \Gamma\times(0,T), \label{1.c}\\
    &\mu=-\epsilon \Delta_\Gamma \varphi +\frac{1}{\epsilon}W'(\varphi)-\frac{1}{2}\eta,
    &  \qquad\qquad\text{on}\ \Gamma\times(0,T),\label{1.d}\\
    &\partial_t v=\Delta_\Gamma \eta +q,
    &  \qquad\qquad\text{on}\ \Gamma\times(0,T), \label{1.e}\\
    &\eta=\frac{4}{\delta}\left(v-\frac{1+\varphi}{2}\right),
    &  \qquad\qquad\text{on}\ \Gamma\times(0,T).\label{1.f}
	\end{alignat}
\end{subequations}
This system is subject to the initial conditions
\begin{subequations}
\begin{alignat}{3}
&u|_{t=0}=u_{0}(x),&\qquad\qquad \text{in}\ \Omega,\label{ini1}\\
&\varphi|_{t=0}=\varphi_{0}(x), \quad  v|_{t=0}=v_{0}(x), &\qquad \qquad \text{on}\ \Gamma.
\label{ini2}
\end{alignat}
\end{subequations}
In \eqref{1.a}, the positive constant $D$  denotes the diffusivity coefficient. In \eqref{1.b}, $\bm{n}=\bm{n}(x)$ denotes the unit outward normal vector on $\Gamma$, and $\partial_{\bm{n}}$ denotes the outer normal derivative on $\Gamma$ such that $\partial_{\bm{n}}u=\nabla u \cdot\bm{n}$. In \eqref{1.c}--\eqref{1.e}, $\Delta_\Gamma$ denotes the Laplace-Beltrami operator on $\Gamma$.

Lipid rafts represent heterogeneous, nanometer-sized microdomains of specific lipid compositions on cell membranes \cite{Pi06}, which have been linked to a wide range of cellular functions like membrane trafficking, signal transduction and protein sorting \cite{BL00,LS10}. Several phenomenological models have been proposed to understand the dynamics and processes governing the formation and maintenance of lipid rafts \cite{EDF13,FSH10,Fo05,GSR08}. It was argued that the competition between lipid phase separation and recycling of raft components is of major importance for the dynamics and structure of lipid rafts. Moreover, different from other phase separation processes, one striking feature in the formation of lipid rafts is the emergence of microdomains with a length-scale below the system size \cite{FSH10}.

The system \eqref{1.a}--\eqref{1.f} under investigation was derived in \cite[Section 2]{GKRR16} based on thermodynamical conservation laws and free energy inequalities for coupled bulk and surface processes. From the modelling point of view, $\Omega$ represents the cell and $\Gamma$ represents its outer membrane, respectively.
The bulk function $u:\Omega\times(0,T)\to \mathbb{R}$ denotes the relative concentration of cytosolic cholesterol in the cell. Its dynamics is governed by the equation \eqref{1.a} that represents a simple diffusion process. The equation \eqref{1.b} characterizes the outflow of cholesterol on the boundary $\Gamma$ and it indicates  possible mass exchange between bulk and surface. Next, the surface function $\varphi:\Gamma\times(0,T)\to [-1,1]$ denotes a rescaled relative concentration of saturated/unsaturated lipid molecules on the cell membrane. The values $\varphi=1$ and $\varphi=-1$ represent the pure saturated-lipid and pure unsaturated-lipid phases, respectively. The lipid phase separation process on the cell membrane is modelled within the diffuse-interface framework, i.e., by the surface Cahn-Hilliard eqaution \eqref{1.c}--\eqref{1.d}. The Cahn-Hilliard equation is a classical phase-field model that gives a continuous description of phase separation process for binary mixtures \cite{CH}. In recent years, the surface Cahn-Hilliard equation and its variants have been used to study phase separation on lipid membranes \cite{RV06,ES10a,ES10b}.
It was shown that the surface Cahn-Hilliard equation can successfully simulate the formation of lipid microdomains \cite{WBV12,YQMO19}. See also \cite{ZWQ21} for the quantitative comparison of the surface Cahn-Hilliard equation with experimental data.
Finally, the surface function $v:\Gamma\times(0,T)\to \mathbb{R}$  denotes the relative concentration of membrane-bound cholesterol, where $v=1$ indicates the maximal saturation. Dynamics of cholesterol on the cell membrane is given by the equations \eqref{1.e}--\eqref{1.f}, which combines a mass-preserving relaxation of the interaction energy and a mass exchange with the bulk reservoir of cholesterol given by the flux from the cytosol. As a consequence, the equation \eqref{1.e} presents cross-diffusion contributions (for $\vp$ and $v$) as well as a mass source term  (cf. \eqref{1.b}).

The equations \eqref{1.a}--\eqref{1.c} and \eqref{1.e} represent the mass balance for the cytosolic cholesterol concentration $u$ in the bulk $\Omega$, and mass balance relations for the relative lipid concentration $\varphi$ as well as the cholesterol concentration $v$ on the surface $\Gamma$. In particular, the specific forms of \eqref{1.b} and \eqref{1.e} indicate that the outflow of cholesterol from the cytosol serves as a source-term in the membrane-cholesterol \cite{GKRR16}. By integration over $\Omega$ and $\Gamma$, we easily find the following properties \cite{AK20}:
\begin{align}
&\int_\Omega u(x,t)\,\mathrm{d}x+ \int_\Gamma v(x,t)\,\mathrm{d}S= \int_\Omega u_0(x)\,\mathrm{d}x+ \int_\Gamma v_0(x)\,\mathrm{d}S,\qquad \forall\, t\in[0,T],\label{mass1}\\
&\int_\Gamma \varphi(x,t)\,\mathrm{d}S=\int_\Gamma \varphi_0(x)\,\mathrm{d}S,\qquad \forall\, t\in[0,T].\label{mass2}
\end{align}
The identity \eqref{mass1} implies that the combined total mass of cytosolic and surface cholesterol is conserved, while \eqref{mass2} gives the mass conversion of lipids on the membrane.

From the energetic point of view, the surface free energy that describes the lipid phase separation and lipid-cholesterol interaction is given as follows \cite{GKRR16}:
\begin{align}
\mathcal{F}(\varphi,v) = \int_\Gamma \Big[\frac{\epsilon}{2}|\nabla_\Gamma\varphi|^2 +\frac{1}{\epsilon}W(\varphi) +\frac{2}{\delta}\Big(v-\frac{1+\varphi}{2}\Big)^2\Big]\,\mathrm{d}S,
\label{Surf:E}
\end{align}
where $\nabla_\Gamma$ denotes the surface gradient on $\Gamma$.
Then the two unknowns $\mu$, $\eta$ in \eqref{1.d} and \eqref{1.f} can be defined as chemical potentials associated with $\mathcal{F}$ such that
\begin{align*}
&\mu:=\frac{\delta\mathcal{F}}{\delta \varphi}
= -\epsilon \Delta_\Gamma \varphi +\frac{1}{\epsilon}W'(\varphi)-\frac{1}{\delta}(2v-1-\varphi),\\
&\eta:=\frac{\delta\mathcal{F}}{\delta v}
= \frac{2}{\delta}(2v-1-\varphi).
\end{align*}
The first two terms in \eqref{Surf:E} forms a classical Ginzburg-Landau type energy for the phase separation process on the membrane. The gradient term represents the surface energy of free interfaces (i.e., heterogeneity of the lipid mixtures) and the nonlinearity $W$ represents the homogeneous free energy. Competition between these two terms leads to
the spatial phase separation. The third term in \eqref{Surf:E} accounts for the affinity between saturated lipid molecules and the
membrane-bound cholesterol, which represents a preferential binding of cholesterol to the lipid-saturated phase. The constants $\epsilon>0$, $\delta>0$ characterize the width of transition layers
(diffuse interfaces) between the distinct lipid phases and the strength of the lipid-cholesterol affinity, respectively. The nonlinear function $W$ in \eqref{Surf:E} is a double-well potential that has two minima and a local unstable maximum in between. A physically significant example is given by
\be
W_{\text{log}}(r)=\frac{\theta}{2}\big[(1-r)\ln(1-r)+(1+r)\ln(1+r)\big] -\frac{\theta_{0}}{2} r^2,\quad \forall\, r\in[-1,1],
\label{logpot}
\ee
with $0<\theta<\theta_{0}$ (see, e.g., \cite{CH,CMZ}). It is also referred to as the Flory-Huggins potential in the context of polymer solutions \cite{Flo42,Hug41}.
The logarithmic part of $W_{\text{log}}$ is related to the mixing entropy for the biniary mixture, while the quadratic term accounts for the demixing effects. We say $W_{\text{log}}$ is a singular potential, because its derivative $W'$ diverges to $\pm \infty$ when its argument approaches $\pm 1$, respectively.
In the literature, $W_{\text{log}}$ is often approximated by a fourth-order polynomial (via Taylor's expansion at $r=0$ with suitably adjusted coefficients, see e.g., \cite{DD95}):
\be
W_{\text{reg}}(r)=\frac{1}{4}\big(1-r^2\big)^2,\qquad \forall\,r\in\mathbb{R},
\label{regular}
\ee
or some more general polynomials  \cite{Mi19}.
As in \cite[Section 2]{GKRR16}, after taking the bulk free energy into account, we can derive the following energy identity (for sufficient regular solutions):
\begin{align}
& \frac{\mathrm{d}}{\mathrm{d}t}\mathcal{E}(u,\varphi,v) + D\int_\Omega |\nabla u|^2\,\mathrm{d}x +\int_\Gamma \big(|\nabla_\Gamma \mu|^2+|\nabla_\Gamma \eta|^2\big)\,\mathrm{d}S    =\int_\Gamma q (\eta-u)\,\mathrm{d}S,\quad \label{BEL1}
\end{align}
where
\begin{align}
\mathcal{E}(u,\varphi,v)
= \frac12\int_\Omega u^2\,\mathrm{d}x+ \mathcal{F}(\varphi,v),
\label{ToE}
\end{align}
denotes the total free energy for the full bulk-surface coupled system \eqref{1.a}--\eqref{1.f}.

The mass exchange term $q$ in \eqref{1.b} and \eqref{1.e} plays an important role in the modelling and analysis. It follows from \eqref{BEL1} that the temporal change of $\mathcal{E}(u,\varphi,v)$ depends on specific choices of $q$. To this end, two different qualitative regimes have been investigated in \cite{GKRR16}. The first one is referred to as the \emph{equilibrium case}, in which the constitutive relation on $q$ guarantees $\int_\Gamma q (\eta-u)\,\mathrm{d}S\leq 0$ so that the total free energy $\mathcal{E}$ is decreasing in time. A typical example in this category takes the following form (see e.g., \cite[(2.17)]{GKRR16}):
\begin{align}
 & q(u,\varphi,v) = -A(\eta-u), \label{q1}
\end{align}
where $A\geq 0$ can be either a constant or a function possibly depending on temporal and spatial variables.
The second regime is called the \emph{non-equilibrium case}, in which the total free energy  $\mathcal{E}$ might increase during the evolution.
One typical constitutive relation in this category considers
the membrane attachment as an elementary ``reaction'' between free sites on the membrane and cholesterol, with possible detachment proportional to the membrane-cholesterol concentration (see e.g., \cite[(1.8)]{GKRR16}):
\begin{align}
 & q(u,v) = B_1u(1-v)-B_2 v,  \label{q2}
\end{align}
 for some constants $B_1,\,B_2 > 0$.
 In \cite{GKRR16}, the authors compared qualitative behaviors of the system \eqref{1.a}--\eqref{1.f} under different choices for the mass exchange term $q$. In the equilibrium case, due to the validity of a global free energy inequality, it was expected that the evolution will approach a steady state as time goes infinity (though this point was not rigorously proved). Besides, the numerical simulations therein indicated that equilibrium-type models could only support the formation of macrodomains and connected phases. On the other hand, the non-equilibrium processes (e.g., with the prototypical choice \eqref{q2}) could lead to the formation of microdomains such as the raft-like structures. This observation in \cite{GKRR16} suggested that the persisting exchange of cholesterol between bulk and cell membrane in the non-equilibrium process is responsible for the formation of complex patterns on a mesoscopic scale. Therefore, the non-equilibrium scenario turns out to be more appealing in the biological modelling. For comprehensive discussions, we refer to \cite{GKRR16}.

Since in the application the cytosolic diffusion inside cells is usually much faster than the lateral diffusion on the cell membranes, the large cytosolic diffusion limit $D\to+\infty$ leads to a natural reduction of the full system \eqref{1.a}--\eqref{1.f} to a nonlocal model defined solely on the cell membrane (see \cite[Section 3]{GKRR16}):
\begin{subequations}
	\begin{alignat}{3}
	& \partial_t u=-\frac{1}{|\Omega|}\int_\Gamma q\,\mathrm{d}S,
    & \qquad\qquad \text{for}\ t\in(0,T),
    \label{r1.a}\\
    &\partial_t\varphi =\Delta_\Gamma \mu,
    &  \qquad\qquad\text{on}\ \Gamma\times(0,T),
    \label{r1.c}\\
    &\mu=-\epsilon \Delta_\Gamma \varphi +\frac{1}{\epsilon}W'(\varphi)-\frac{1}{2}\eta,
    &  \qquad\qquad\text{on}\ \Gamma\times(0,T),
    \label{r1.d}\\
    &\partial_t v=\Delta_\Gamma \eta +q,
    &  \qquad\qquad\text{on}\ \Gamma\times(0,T),
    \label{r1.e}\\
    &\eta=\frac{4}{\delta}\Big(v-\frac{1+\varphi}{2}\Big),
    &  \qquad\qquad\text{on}\ \Gamma\times(0,T).
    \label{r1.f}
	\end{alignat}
\end{subequations}
In this reduced model, the concentration of cytosolic
cholesterol $u$ becomes a time-dependent function that is spatially constant in $\Omega$ and satisfies an ordinary differential equation \eqref{r1.a}. The system \eqref{r1.c}--\eqref{r1.f} is then subject to the initial conditions
\begin{subequations}
\begin{alignat}{3}
&\varphi|_{t=0}=\varphi_{0}(x), \quad  v|_{t=0}=v_{0}(x), &\qquad \qquad \text{on}\ \Gamma,
\label{rini1}\\
&u|_{t=0}= u_0 =\frac{1}{|\Omega|}\Big(M-\int_\Gamma v_0\,\mathrm{d}S\Big),
\label{rini2}
\end{alignat}
\end{subequations}
where the constant $M >0$ denotes the total mass of cholesterol in the bulk and surface (which is conserved in time, cf. \eqref{mass1}). Numerical simulations performed in \cite{GKRR16} indicated that if the effect of the lipid interaction with membrane-bound cholesterol is sufficiently large (i.e., when $0<\delta \ll 1$, cf. \eqref{Surf:E}), the long-time behavior (in terms of stationary states) of the reduced system \eqref{r1.a}--\eqref{r1.f} can be very close to that of the well-known Ohta-Kawasaki system for phase separation in diblock copolymers \cite{NO95,OK86}, which has been shown to generate intermediate-sized structures in the phase separation process.

When the double well potential $W$ is assumed to be regular with the following specific form
$$
W(r)=(r^2-1)^2,
$$
rigorous mathematical analysis for the full bulk-surface coupled system \eqref{1.a}--\eqref{1.f} as well as the reduced surface system \eqref{r1.a}--\eqref{r1.f} has been carried out in the recent contribution \cite{AK20}. Under suitable assumptions on the mass exchange term $q$, the authors of \cite{AK20} proved existence, uniqueness and regularity of global weak solutions to the initial boundary value problem \eqref{1.a}--\eqref{ini2}. Furthermore, they investigated the large cytosolic diffusion limit as $D\to +\infty$ in the full bulk-surface coupled system \eqref{1.a}--\eqref{1.f} and then proved existence of global weak solutions to the reduced problem \eqref{r1.a}--\eqref{rini2}. For the reduced system \eqref{r1.a}--\eqref{r1.f}, they established the existence of stationary solutions by Leray-Shauder's principle and showed the boundedness of global weak solutions under certain sublinear growth condition on $q$. Finally, they studied the limit of large affinity between membrane components as $\delta\to 0^+$ and recovered a variant of the Ohta-Kawasaki system. We also mention that existence of weak solutions to the corresponding sharp interface model (i.e., a Mullins-Sekerka equation coupled to diffusion equations in the bulk domain $\Omega$ and on the boundary $\Gamma$) has been established in \cite{AK21a}. Besides, the sharp-interface limit as $\epsilon\to 0^+$, that is, the convergence of weak solutions of the diffuse-interface model to weak (varifold) solutions of the corresponding sharp-interface model was rigorously proved in \cite{AK21b}.

In view of the physical interpretation of $\vp$, only values in the interval $[-1,1]$ are admissible. However, the Cahn-Hilliard equation with a regular potential does not admit a maximum principle that can guarantee the physical bound $\vp\in [-1,1]$ along evolution (see for instance, \cite[Remark 2.1]{CMZ} for a counterexample). This motivates us to extend the mathematical analysis done in \cite{AK20} to the more interesting case with a singular potential $W$, in particular, including the physically relevant logarithmic type \eqref{logpot}. Furthermore, our second aim is to rigorously prove the long-time behavior of global weak solutions to both systems under suitable assumptions on the mass exchange term $q$.

In summary, we establish the following results.
\begin{itemize}
\item[(1)] We prove existence and uniqueness of global weak solutions to the full bulk-surface coupled system \eqref{1.a}--\eqref{ini2} (see Theorem \ref{thm:weak}). For the mass exchange term $q$ with a general form, the existence of a global weak solution is obtained under a linear growth assumption on $q$, while uniqueness can be proven if $q$ is globally Lipschitz continuous. Besides, in the equilibrium case with $q$ given by \eqref{q1}, we show that the global weak solution is uniformly bounded for all $t\geq0$. The proof is based on a suitable regularization of the singular potential $W$ and a Faedo-Galerkin approximating scheme for the bulk-surface coupled system.
\item[(2)] We show the instantaneous regularity of global weak solutions to the full bulk-surface coupled system \eqref{1.a}--\eqref{ini2} for positive time  (see Theorem \ref{thm:reg}). In particular, we prove that if the initial datum $\vp_0$ is not a pure state then the weak solution $\vp$ will keep a strictly positive distance from the pure states $\pm 1$ for any $t>0$. Moreover, this positive distance can be uniform in the equilibrium case \eqref{q1} from a certain time on. The property of \emph{separation from pure states} is crucial in the study of the Cahn-Hilliard equation with singular potential, since it enables us to overcome possible singularities for derivatives of the potential $W$ near $\pm 1$ and then obtain higher-order spatial regularity of solutions (see e.g., \cite{AW07,CMZ,FG12,GGW23,GGM2017,GGW18,H1,Mi19} and the references cited therein).
\item[(3)]  Applying the well-known {\L}ojasiewicz-Simon approach \cite{LS83} (see e.g., \cite{AW07,FS,HJ2001,HT01,GGW18,JWZ,RH99} for its various applications), we prove that in the equilibrium case with $q$ given by \eqref{q1}, every global weak solution to the full bulk-surface coupled system \eqref{1.a}--\eqref{ini2}  converges to a single equilibrium as time goes to infinity (see Theorem \ref{thm:conv}), provided that the potential function $W$ is real analytic (cf. \eqref{logpot}) and the mass exchange between the bulk and the cell membrane decays sufficiently fast. This result rigorously justifies the expectation made in \cite{GKRR16} about the long-time behavior of global solutions in the equilibrium case. The stationary state $\vp_\infty$ coincides with a critical point of the Cahn-Hilliard energy subject to the mass constraint
    $\int_\Gamma \vp_\infty\,\mathrm{d}S=\int_\Gamma \vp_0\,\mathrm{d}S$, while other steady states can be determined by the bulk-surface mass constraint (cf. \eqref{mass1})  and the constitutive relation  \eqref{q1}.
\item[(4)] For the reduced problem \eqref{r1.a}--\eqref{rini2}, we are mainly interested in the non-equilibrium case with a reaction type source term $q$ given by \eqref{q2}. We establish the existence and uniqueness of a global weak solution that is uniformly bounded in time (see Theorem \ref{thm:rweak}). The proof for the existence result is based on the large cytosolic diffusion limit as $D\to +\infty$ for the full system \eqref{1.a}--\eqref{1.f} with a modified mass exchange term $\widetilde{q}$ given by \eqref{q2m} (see Proposition \ref{prop:LD}). This extends \cite[Proposition 2.6]{AK20} to the current case with a singular potential. Although the free energy $\mathcal{F}$ may increase in the non-equilibrium case, under the special structure of \eqref{q2}, we find that the system \eqref{r1.a}--\eqref{r1.f} admits a dissipative type estimate, which yields the uniform-in-time boundedness of it solution.
\item[(5)] For the reduced problem \eqref{r1.a}--\eqref{rini2} with $q$ given by \eqref{q2},  we further investigate its long-time behavior in terms of the global attractor (see Theorem \ref{thm:att}). First, we show that global weak solutions to the reduced problem \eqref{r1.a}--\eqref{rini2} generate a strongly continuous semigroup $\mathcal{S}(t)$ in a suitable phase space $\mathcal{V}_{M,m}$. The dissipative estimate obtained in (4) further gives the existence of an absorbing set in $\mathcal{V}_{M,m}$. Combining this with the regularity property of weak solutions (which yields the asymptotic compactness), we are able to prove the existence of a global attractor for the dissipative dynamical system $(\mathcal{S}(t),\mathcal{V}_{M,m})$ by adapting the classical theory for infinite dimensional dynamical systems \cite{T}.
\end{itemize}

Our study provides a first-step theoretical analysis on the bulk-surface coupled system \eqref{1.a}--\eqref{1.f} as well as its reduced system \eqref{r1.a}--\eqref{r1.f} with a singular potential. Some interesting future issues include, for instance, estimates on the fractal dimension of the global attractor and existence of exponential attractors, connection with the Ohta-Kawasaki equation in the large affinity limit, evolution problem on evolving domains/surfaces (cf. e.g., \cite{CE21,CEGP,ER15,YQO20,ZT19} for modelling and analysis of phase separation on dynamic surfaces). \medskip

\emph{Plan of the paper}. The remaining part of this paper is organized as follows. In Section \ref{sec:main}, we first introduce some notations, assumptions and analytic tools that will be used throughout the paper. Then we present the main results of this paper. In Section \ref{sec:wellBS}, we study well-posedness of the full bulk-surface coupled system \eqref{1.a}--\eqref{ini2}. We prove Theorem \ref{thm:weak} on the existence and uniqueness of global weak solutions and Theorem \ref{thm:reg} on their regularity properties. Section \ref{sec:LTBS} is devoted to the long-time behavior of the full system \eqref{1.a}--\eqref{ini2} in the equilibrium case with a mass exchange term $q$ satisfying \eqref{q1}. We characterize the $\omega$-limit set and prove the convergence of global weak solutions to a single equilibrium as $t\to+\infty$ by means of the {\L}ojasiewicz-Simon approach (Theorem \ref{thm:conv}). In the final Section \ref{sec:RdS}, we study the reduced problem \eqref{r1.a}--\eqref{rini2} in the non-equilibrium case with a mass exchange term $q$ satisfying \eqref{q2}. We prove the existence and uniqueness of a global weak solution (Theorem \ref{thm:rweak}) and establish the existence of a global attractor (Theorem \ref{thm:att}).

\section{Main Results}\label{sec:main}
\setcounter{equation}{0}
\subsection{Preliminaries}
Let $X$ be a real Banach space. The dual space of $X$ is denoted by $X'$, and the duality pairing between $X$, $X'$ is denoted by
$\langle \cdot,\cdot\rangle_{X',X}$. Given an interval $J\subset [0,+\infty)$, we introduce the function space $L^p(J;X)$ with $p\in [1,+\infty]$, which consists of Bochner measurable $p$-integrable
functions with values in the Banach space $X$. The boldface letter $\bm{X}$ denotes the vectorial space $X^d$ ($d\in \mathbb{Z}^+$) endowed with the product structure.

We assume that $\Omega \subset\mathbb{R}^3$ is a smooth bounded domain with its boundary denoted by $\Gamma$. Besides, $\Gamma$ is assumed to be a two-dimensional smooth, compact and connected hypersurface without boundary. For the standard Lebesgue and Sobolev spaces on $\Omega$ and $\Gamma$, we use the notations $L^{p}(\Omega)$, $W^{k,p}(\Omega)$, $L^{p}(\Gamma)$, $W^{k,p}(\Gamma)$ for any $p \in [1,+\infty]$ and $k > 0$,
equipped with the corresponding norms
$\|\cdot\|_{L^{p}(\Omega)}$, $\|\cdot\|_{W^{k,p}(\Omega)}$, $\|\cdot\|_{L^{p}(\Gamma)}$, $\|\cdot\|_{W^{k,p}(\Gamma)}$, respectively. When $p = 2$, these spaces are Hilbert spaces and we simply denote $H^{k}(\Omega) := W^{k,2}(\Omega)$, $H^{k}(\Gamma) := W^{k,2}(\Gamma)$.
For details of how to define surface Sobolev spaces, we refer to \cite{Aubin82}. The following interpolation inequality  will be used in the subsequent analysis (see e.g., \cite[Chapter 2, (2.27)]{La}):
\begin{align}
\|f\|_{L^2(\Gamma)}\leq C\|f\|_{H^1(\Omega)}^\frac12\|f\|_{L^2(\Omega)}^\frac12,\qquad \forall\, f\in H^1(\Omega),
\label{inter1}
\end{align}
where $C>0$ only depends on $\Omega$ and $\Gamma$.

Next, we recall some basic facts of calculus on surfaces, see e.g.,  \cite[Section 2]{DE13}.
Let $\bm{n}=\bm{n}(x)$ be the unit outward normal vector on $\Gamma$. For a scalar differentiable function $f: \Gamma \to \mathbb{R}$, its surface (tangential) gradient is defined as $\nabla_\Gamma f = \nabla \widetilde{f} - (\nabla \widetilde{f}\cdot \bm{n})\bm{n}$. Here, $\nabla$ denotes the gradient in $\mathbb{R}^3$, $\widetilde{f}$ is a smooth extension of $f$ to a three-dimensional neighbourhood $U$ of the surface $\Gamma$ such that $\widetilde{f}|_\Gamma =f$. The surface gradient $\nabla_\Gamma f$ depends only on
the values of $f$ on $\Gamma$ and its components are denoted by $\nabla_\Gamma f= (\underline{D}_if)_{i=1}^3$. For a differentiable vector $\bm{f}: \Gamma \to \mathbb{R}^3$, $\nabla_\Gamma \bm{f}$ is a matrix with components $(\nabla_\Gamma \bm{f})_{ij}=\underline{D}_j\bm{f}_i$, $1\leq i, j\leq 3$.
Then the surface divergence operator is defined by $\mathrm{div}_\Gamma \bm{f} = \mathrm{tr}(\nabla_\Gamma \bm{f})$ and the Laplace-Beltrami operator $\Delta_\Gamma$ is given by the surface divergence of the surface gradient, that is, for any twice differentiable
function $f: \Gamma \to \mathbb{R}$, $\Delta_\Gamma f = \mathrm{div}_\Gamma (\nabla_\Gamma f)$. Moreover, we have the formula for integration by parts as well as Green's
formula \cite[Theorems 2.10, 2.14]{DE13}:
\begin{align*}
& \int_\Gamma \nabla_\Gamma f \,\mathrm{d}S = \int_\Gamma f \overline{H} \bm{n} \,\mathrm{d}S,& & \forall\, f\in C^1(\Gamma),\\
& \int_\Gamma \nabla_\Gamma f\cdot \nabla_\Gamma g\,\mathrm{d}S=-\int_\Gamma f\Delta_\Gamma g \,\mathrm{d}S,& & \forall\, f,g\in C^1(\Gamma),
\end{align*}
where $\overline{H}=\mathrm{div}_\Gamma (\bm{n})$ denotes the mean curvature of $\Gamma$.

We denote by $|\Omega|$ the Lebesgue measure of the domain $\Omega$ and by $|\Gamma|$ the
two-dimensional surface measure of its boundary $\Gamma$.
For every $f\in (H^1(\Omega))'$, we denote by $\langle f\rangle_\Omega$ its generalized mean value over $\Omega$ such that
$\langle f\rangle_\Omega=|\Omega|^{-1}\langle f,1\rangle_{(H^1(\Omega))',\,H^1(\Omega)}$. If $f\in L^1(\Omega)$, then its spatial mean is simply given by $\langle f\rangle_\Omega=|\Omega|^{-1}\int_\Omega f \,\mathrm{d}x$. Concerning a function $f$ on $\Gamma$, its mean value (denoted by $\langle f\rangle_\Gamma$), can be defined in a similar manner.
Then we recall Poincar\'{e}'s inequalities in $\Omega$ and on $\Gamma$ (see e.g., \cite[Theorem 2.12]{DE13} for the case on $\Gamma$):
\begin{align*}
&\|f-\langle f\rangle_\Omega\|_{L^p(\Omega)} \leq C_\Omega\|\nabla f\|_{\bm{L}^p(\Omega)},\qquad\, \forall\,
f\in W^{1,p}(\Omega), \\
&\|f-\langle f\rangle_\Gamma\|_{L^p(\Gamma)} \leq C_\Gamma\|\nabla_\Gamma f\|_{\bm{L}^p(\Gamma)},\qquad \forall\,
f\in W^{1,p}(\Gamma),
\end{align*}
where $p\in [1,+\infty)$, $C_\Omega$ (resp. $C_\Gamma$) is a positive constant depending only on $\Omega$ (resp. $\Gamma$) and $p$.
Next, for any given $m\in\mathbb{R}$, we define
$$
H^1_{(m)}(\Gamma) = \big\{f\in H^1(\Gamma)\ \big|\ \langle f\rangle_\Gamma=m\big\}.
$$
Consider the Poisson equation
\begin{align}
-\Delta_\Gamma \vp = g,\quad \text{on}\ \ \Gamma. \label{ellipg}
\end{align}
By the Lax-Milgram theorem, it easily follows that for any $g\in (H^1(\Gamma))'$ satisfying $\langle g\rangle_\Gamma= 0$, the equation \eqref{ellipg} admits a unique weak solution $\vp \in H^1_{(0)}(\Gamma)$, see e.g., \cite[Theorem 3.1]{DE13}. Moreover, if $g\in L^2(\Gamma)$ with $\langle g\rangle_\Gamma= 0$, then it holds (see \cite[Theorem 3.3]{DE13}, also   \cite{Aubin82})
$$
\|\vp\|_{H^2(\Gamma)}\leq C\|g\|_{L^2(\Gamma)},
$$
where $C>0$ only depends on $\Gamma$. Hence, the restriction of $-\Delta_\Gamma$ from $H^1_{(0)}(\Gamma)$ onto $\big(H^1_{(0)}(\Gamma)\big)':=\{f\in (H^1(\Gamma))'\,|\,\langle f\rangle_\Gamma= 0\} $ is an isomorphism, $-\Delta_\Gamma$ is positively defined on $H^1_{(0)}(\Gamma)$ and self-adjoint. For any $g\in \big(H^1_{(0)}(\Gamma)\big)'$, we can set $\|g\|_{\big(H^1_{(0)}(\Gamma)\big)'} =\|\nabla_\Gamma(-\Delta_\Gamma)^{-1}g\|_{\bm{L}^2(\Gamma)}$. Besides, for any $g\in (H^1(\Gamma))'$, $g\to (\|g-\langle g\rangle_\Gamma\|_{(H^1_{(0)}(\Gamma))'}^2+|\langle g\rangle_\Gamma|^2)^{1/2}$ gives an equivalent norm on $(H^1(\Gamma))'$. By Poincar\'e's inequality, we find
$g\to \|\nabla_\Gamma g\|_{\bm{L}^2(\Gamma)}$ and
$g\to (\|\nabla_\Gamma g\|_{\bm{L}^2(\Gamma)}^2+|\langle g\rangle_\Gamma|^2)^{1/2} $ give equivalent norms on $H^1_{(0)}(\Gamma)$, $H^1(\Gamma)$, respectively.

Throughout this paper, the symbols $C$, $C_i$, $i\in \mathbb{N}$, denote generic positive constants that may depend on structural coefficients of the system, the initial data, $\Omega$, $\Gamma$ and the final time $T$. Their values may change from line to line and specific dependence will be pointed out if necessary.
\medskip

\textbf{Basic assumptions}. Let us now introduce some basic hypotheses for the initial boundary value problem \eqref{1.a}--\eqref{ini2}.
\begin{enumerate}
	\item[\textbf{(H1)}] The singular potential $W$ belongs to the class of functions $C\big([-1,1]\big)\cap C^{3}\big((-1,1)\big)$ and can be written into the following form
	\begin{equation}
	W(r)=F(r)-\frac{\theta_{0}}{2}r^2,\nonumber
	\end{equation}
	such that
	\begin{equation}
	\lim_{r\to \pm 1} F'(r)=\pm \infty \quad \text{and}\ \  F''(r)\ge \theta,\qquad \forall\, r\in (-1,1),\nonumber
	\end{equation}
	where $\theta$ is a strictly positive constant and $\theta_0\in \mathbb{R}$. We make the extension $F(r)=+\infty$ for any $r\notin[-1,1]$.
	In addition, there exists $r_0\in(0,1)$ such that $F''$ is nondecreasing in $[1-r_0,1)$ and nonincreasing in $(-1,-1+r_0]$.
	\item[\textbf{(H2)}] The function $F$ satisfies the following growth assumption on its derivatives
	\be F^{\prime \prime}(r) \leq C \mathrm{e}^{C\left|F^{\prime}(r)\right|},
\qquad \forall\, r \in(-1,1), \label{de2}
\ee
	 for some positive constant  $C$.
	\item[\textbf{(H3)}] The structural coefficients $D$, $\epsilon$,   $\delta$ are  prescribed constants such that
	\be
	D>0,\quad \epsilon=1,\quad  \delta>0. \nonumber
	\ee
\end{enumerate}
\begin{remark}
The assumptions $\textbf{(H1)}$ and $\textbf{(H2)}$ on the singular potential $W$ are similar to those in \cite{AW07,GGM2017,MZ04}. In particular, $\textbf{(H2)}$ will be used to derive the strict separation property of the phase function $\vp$.
The thermodynamically relevant logarithmic potential \eqref{logpot} satisfies $\textbf{(H1)}$ and $\textbf{(H2)}$ with
$$
F(r)=\frac{\theta}{2}\big[(1-r)\ln(1-r)+(1+r)\ln(1+r)\big].
$$
In this paper, we focus on the diffuse interface model \eqref{1.a}--\eqref{1.f} and do not study its sharp-interface limit as $\epsilon \to 0^+$, so we just set $\epsilon=1$ in \textbf{(H3)} because its value does not influence the subsequent analyses. A formal asymptotic expansion for the sharp interface reduction as $\epsilon \to 0^+$ of the system \eqref{1.a}--\eqref{1.f} was performed in \cite[Section 4]{GKRR16}
and a rigorous justification in the case of a regular potential like \eqref{regular} can be found in \cite{AK21b}. The resulting limit problem consists of a free boundary problem of Mullins-Sekerka type on $\Gamma$ coupled to a diffusion process in $\Omega$ that also includes an interaction with the cholesterol concentration.
\end{remark}

Next, we consider the mass exchange term $q$. According to \cite[Remark 2.7]{AK20}, a linear growth condition for a general choice of $q$ (see \eqref{ligro} below) is needed in order to establish the existence of global weak solutions (see \cite[Theorem 2.3]{AK20} for the case with a regular potential like \eqref{regular}). However, the biologically interesting case \eqref{q2} does not fulfil this requirement. A possible modification is to introduce a proper cut-off via a bounded smooth function $\widetilde{h}$ such that
\begin{align}
\widetilde{q}(u,v)= B_1u -B_1\widetilde{h}(u)v-B_2 v.\label{q2m}
\end{align}
Then we impose the following assumptions on $q$.
\begin{itemize}
\item[\textbf{(H4)}] In the non-equilibrium case \eqref{q2} and its variant \eqref{q2m}, $B_1,B_2  > 0$ are given constants. Besides, the function $\widetilde{h}$ in \eqref{q2m} satisfies
$$
\widetilde{h}\in C^1(\mathbb{R})\cap W^{1,\infty}(\mathbb{R})\qquad \text{and}\qquad  \widetilde{h}(r)=r,\quad \forall\,r\in [-h_0,\,h_0],
$$
    for some $h_0>0$.
\item[\textbf{(H5)}] In the equilibrium case \eqref{q1}, the coefficient $A$ is allowed to be a function depending on the temporal variable, which  satisfies
$$
A\in C^1([0,+\infty))\cap W^{1,\infty}([0,+\infty))\qquad
    \text{and}
    \qquad A(r)\geq 0,\quad \forall\,r\geq 0.
$$
\end{itemize}
\begin{remark}
The assumption \textbf{(H4)} implies that the modified mass exchange term $\widetilde{q}$ satisfies a linear growth condition and it coincides with the original one given in \eqref{q2} as long as $|u|\leq h_0$. This restriction can be removed for the reduced system \eqref{r1.a}--\eqref{r1.f} under suitable assumptions on the initial datum $u_0$. The assumption \textbf{(H5)} guarantees that the mass exchange process leads to a negative contribution to the change of the total free energy as time evolves. Moreover, the value of $A$ measures the strength of mass exchange between bulk and surface.
\end{remark}

\textbf{Analytic tools}. We present some preliminary results that will be crucial in the subsequent analyses.

First, to derive the property of strict separation from pure states $\pm1$ for the phase function $\varphi$, we consider the following elliptic equation with a singular nonlinearity $F'$:
\be
-\Delta_\Gamma \varphi +F^{\prime}(\varphi)=g,\qquad \text { on } \Gamma.
\label{ph1}
\ee
Then we have

\bl\label{lem:sep}
Suppose that $\Gamma$ is a two-dimensional smooth, compact and connected hypersurface without boundary and $F$ satisfies the assumption \textbf{(H1)}.
\begin{itemize}
\item[(1)] For any given function $g\in L^2(\Gamma)$, \eqref{ph1} admits a unique solution $\varphi\in H^2(\Gamma)$ that satisfies the equation a.e. on $\Gamma$ with $F^{\prime}(\varphi)\in L^2(\Gamma)$.
\item[(2)] If $g\in H^1(\Gamma)$, then for any $p\in [2,+\infty)$, there exists a positive constant $C$ depending on $\Gamma$ and $p$ such that
\begin{align*}
\|\varphi\|_{W^{2, p}(\Gamma)}+\|F^{\prime}(\varphi)\|_{L^{p}(\Gamma)}
&\leq C\left(1+\|g\|_{H^1(\Gamma)}\right).
\end{align*}
If in addition, $F$ satisfies \textbf{(H2)}, then we have
\begin{align*}
\|F^{\prime \prime}(\varphi)\|_{L^{p}(\Gamma)}
&\leq C\Big(1+{e}^{C\|g\|_{H^1(\Gamma)}^{2}}\Big).
\end{align*}
Moreover, it holds $F^{\prime}(\varphi)\in W^{1, p}(\Gamma)$ and there exists a constant $\sigma\in (0,1)$ such that
\be
\|\varphi\|_{L^{\infty}(\Gamma)} \leq 1-\sigma.\label{sep2}
\ee
\end{itemize}
\el
\begin{remark}
The proof of Lemma \ref{lem:sep} can be carried out by adapting the arguments for the same equation in a two-dimensional bounded smooth Cartesian domain in $\mathbb{R}^2$ subject to a homogeneous Neumann boundary condition for $\vp$, see e.g., \cite{GGW23,GGM2017,H1}. We also refer to the recent work \cite[Theorem 3.9]{CEGP} for extensions to the more involved case on evolving surfaces.
\end{remark}

Second, to investigate the long-time behavior of global weak solutions to problem \eqref{1.a}--\eqref{ini2} as $t\to +\infty$, we consider the following nonlocal elliptic equation:
\begin{align}
\begin{cases}
-\Delta_\Gamma \varphi +W^{\prime}(\varphi)=\displaystyle{\frac{1}{|\Gamma|}\int_\Gamma W^{\prime}(\vp)\,\mathrm{d}S,}\qquad \text { on } \Gamma,\\
\text{subject to the constraint}\ \ \displaystyle{\langle \vp\rangle_\Gamma =m,}
\end{cases}
\label{ph2}
\end{align}
where $m\in (-1,1)$ is a given constant.
Define
\begin{align*}
 &Z_m =\big\{\vp \in H^1_{(m)}(\Gamma)\ \big|\  \|\vp\|_{L^\infty(\Gamma)}\leq 1\big\},   \\
 &\varPhi_m  =\big\{\vp \in H^2(\Gamma)\cap H^1_{(m)}(\Gamma)\ \big|\ \vp\ \text{solves}\ \eqref{ph2}\big\}.
\end{align*}
Then we have
\begin{lemma}
\label{lem:LS}
Suppose that $\Gamma$ is a two-dimensional smooth, compact and connected hypersurface without boundary and the nonlinearity $W$ satisfies the assumption \textbf{(H1)}. For any given constant $m\in (-1,1)$, let us consider the energy functional
\begin{align}
E(\vp)=  \frac12 \|\nabla_\Gamma \vp\|_{L^2(\Gamma)}^2
+ \int_\Gamma  W(\vp) \,\mathrm{d}S,\qquad \forall\, \vp\in Z_m.
\label{EEE}
\end{align}
\begin{itemize}
\item[(1)]
The set $\varPhi_m \subset Z_m$ is nonempty.
Moreover, for every $\vp_*\in \varPhi_m$, there exists a constant $\sigma_* \in (0,1)$ such that $|\vp_*(x)|\leq 1-\sigma_*$ for all $x\in \Gamma$.
\item[(2)] Assume in addition, the potential function $W$ is real analytic on $(-1,1)$.
Then for every $\vp_* \in \varPhi_m$, there exist constants $\chi\in (0,1/2)$ and $C,\,\ell>0$ such that
\begin{align}
|E(\vp)-E(\vp_*)|^{1-\chi}\leq C\Big\|-\Delta_\Gamma \varphi +W^{\prime}(\varphi)-\displaystyle{\frac{1}{|\Gamma|}\int_\Gamma W^{\prime}(\vp)\,\mathrm{d}S}\Big\|_{(H^1(\Gamma))'},
\label{LSa}
\end{align}
whenever $\vp\in H^2(\Gamma)\cap H^1_{(m)}(\Gamma)$ satisfying $\|\vp-\vp_*\|_{H^2(\Gamma)}\leq \ell$. The positive constants $\chi,\,C,\,\ell$ may depend on $\vp_*$, $m$ and $\Gamma$, but not on $\vp$.
\end{itemize}
\end{lemma}
\begin{remark}
It is easy to verify that the physically relevant logarithmic potential \eqref{logpot} is real analytic on $(-1,1)$ and thus fulfils all requirements in Lemma \ref{lem:LS}-(2).
The proof of Lemma \ref{lem:LS} follows the arguments in \cite[Section 6]{AW07} for the classical Cahn-Hilliard equation with a singular potential, subject to homogeneous Neumann boundary conditions in a smooth bounded Cartesian domain $\Omega\subset \mathbb{R}^3$.

The inequality \eqref{LSa} is a specific type of the so-called Lojasiewicz-Simon gradient inequality, which turns out to be very useful in the study of long-time behavior for nonlinear evolution equations with a dissipative structure, see e.g., \cite{A2009,AW07,HJ2001,JWZ,Mi19,RH99} and the references cited therein. Note that in Lemma \ref{lem:LS}-(2) we require $\|\vp-\vp_*\|_{H^2(\Gamma)}$ to be sufficiently small. This condition combined with the Sobolev embedding $H^2(\Gamma)\hookrightarrow C(\Gamma)$ and the strict separation property of $\vp_*$ as stated in Lemma \ref{lem:LS}-(1) yields that the perturbation $\vp$ belongs to $Z_m$ and indeed it keeps a uniform positive distance from the pure states $\pm 1$. Hence, we can avoid the difficulty from the singularity of $W'$ at $\pm 1$ and adapt the argument for \cite[Proposition 6.3]{AW07}.
\end{remark}

\subsection{Statement of results}

\subsubsection{The full bulk-surface coupled system \eqref{1.a}--\eqref{1.f}}

We first introduce the notion of (global) weak solutions.

\begin{definition}[Weak solutions]
\label{def:weak}
Let $T>0$. We call the functions $(u,\varphi,\mu,v,\eta)$ a weak solution to problem \eqref{1.a}--\eqref{ini2} on the interval $[0,T]$, if the following properties are fulfilled:
\begin{itemize}
\item[(1)] $(u,\varphi,\mu,v,\eta)$ satisfy
\begin{align}
& u\in L^\infty(0,T;L^2(\Omega))\cap L^2(0,T;H^1(\Omega))\cap H^1(0,T;(H^1(\Omega))'),\notag\\
& \varphi \in L^{\infty}(0,T;H^1(\Gamma))\cap L^{4}(0,T;H^2(\Gamma))\cap L^2(0,T;W^{2,p}(\Gamma)) \cap H^{1}(0,T;(H^1(\Gamma))'),\notag \\
&\mu \in L^{2}(0,T;H^1(\Gamma)),\qquad F'(\varphi)\in L^2(0,T;L^p(\Gamma)),\notag \\
& v,\,\eta  \in L^{\infty}(0,T;L^2(\Gamma))\cap L^{2}(0,T;H^1(\Gamma)) \cap H^1(0,T;(H^1(\Gamma))'),\notag\\
&\varphi\in L^{\infty}(\Gamma\times (0,T))\ \textrm{and}\ \ |\varphi(x,t)|<1\ \ \textrm{a.e.\ on}\ \Gamma\times(0,T),\notag
\end{align}
where $p\in [2,+\infty)$ is arbitrary.
\item[(2)] For all $\xi\in L^2(0,T;H^1(\Omega))$ and $\zeta\in L^2(0,T;H^1(\Gamma))$, it holds
\begin{subequations}
			\begin{alignat}{3}
& \int_0^T  \langle \partial_t u,\xi \rangle_{(H^1(\Omega))',\,H^1(\Omega)}\,\mathrm{d}t
+ D\int_0^T \!\int_\Omega \nabla u\cdot \nabla \xi\,\mathrm{d}x\mathrm{d}t=
- \int_0^T\! \int_\Gamma q \xi \,\mathrm{d}S \mathrm{d}t,&\label{w1.a}\\
& \int_0^T \left \langle \partial_t \varphi,\zeta\right \rangle_{(H^1(\Gamma))',\,H^1(\Gamma)}\,\mathrm{d}t
 = - \int_0^T \!\int_\Gamma \nabla_\Gamma \mu \cdot\nabla_\Gamma \zeta\,\mathrm{d}S \mathrm{d}t, &\label{w1.c} \\
			&\ \, \mu=- \Delta_\Gamma \varphi + W'(\varphi)-\frac{1}{2}\eta,\qquad \text{a.e. on}\ \Gamma\times (0,T),&\label{w1.d}\\
			&\int_0^T\left \langle\partial_t v,\zeta\right \rangle_{(H^1(\Gamma))',\,H^1(\Gamma)}\,\mathrm{d}t +  \int_0^T \!\int_\Gamma \nabla_\Gamma \eta\cdot \nabla_\Gamma \zeta\,\mathrm{d}S \mathrm{d}t = \int_0^T\! \int_\Gamma q \zeta \,\mathrm{d}S \mathrm{d}t,& \label{w1.e}\\
            &\ \, \eta = \frac{2}{\delta}\left(2v-1-\varphi\right),\qquad \text{a.e. on}\ \Gamma\times (0,T).& \label{w1.f}
\end{alignat}
\end{subequations}
\item[(3)] The initial conditions \eqref{ini1}--\eqref{ini2} are satisfied.
\end{itemize}
\end{definition}
\begin{remark}
The initial data can be attained in the following sense: from
regularity properties of weak solutions, we have
$$
u\in C_w([0,T];L^2(\Omega)),\quad \varphi\in C_w([0,T];H^1(\Gamma))\quad \text{and}\quad
v\in C_w([0,T];L^2(\Gamma)),
$$
where the subscript ``w'' stands for the weak continuity in time. Indeed, we can derive the strong continuity in time, see Section \ref{mei}.
\end{remark}

The existence of a global weak solution to the full bulk-surface coupled system \eqref{1.a}--\eqref{ini2} is ensured by the following theorem.

\bt[Existence and uniqueness] \label{thm:weak}
Let $T>0$ and the assumptions \textbf{(H1)}, \textbf{(H3)} be satisfied. Suppose that the general mass exchange term $q=q(u,\vp,v):\,\mathbb{R}^3\to \mathbb{R}$ is continuous and fulfils the linear growth condition
\begin{align}
 |q(u,\,\varphi,\,v)|\leq C(1+|u|+|\varphi|+|v|),\qquad \forall\, u,\,\varphi,\,v\in \mathbb{R},
 \label{ligro}
\end{align}
  with some positive constant $C$ that is independent of $u$, $\varphi$ and $v$.
\begin{itemize}
 \item[(1)] For any initial data $(u_0,\,\varphi_0,\,v_0)$ satisfying
\begin{align}
 u_0\in L^2(\Omega),\ \  \varphi_{0}\in H^1(\Gamma),\ \  v_{0}\in L^2(\Gamma)\ \ \text{with}\ \ \|  \varphi_{0} \|_{L^{\infty}(\Gamma)} \le 1, \ \  |\langle\varphi_{0}\rangle_\Gamma|<1,
 \label{cini}
 \end{align}
 the initial boundary value problem \eqref{1.a}--\eqref{ini2} admits at least one global weak solution $(u,\varphi,\mu,v,\eta)$ on $[0,T]$ in the sense of Definition \ref{def:weak}. Moreover, we have
 \begin{align}
 &\sup_{t\in [0,T]}\|\varphi(t)\|_{L^\infty(\Gamma)}\leq 1.
  \label{ues1}
 \end{align}
\item[(2)] If the mass exchange term $q$ is globally Lipschitz continuous on $\mathbb{R}^3$, then the weak solution is unique.
\item[(3)] In the equilibrium case, i.e., $q$ takes the form of \eqref{q1} under the assumption \textbf{(H5)}, the global weak solution is unique and it satisfies the energy equality
 \begin{align}
 & \mathcal{E}(u(t),\varphi(t),v(t))
 + D\int_0^t\!\int_\Omega |\nabla u(\tau)|^2\,\mathrm{d}x \mathrm{d}\tau \notag\\
 &\qquad
 + \int_0^t\! \int_\Gamma \big(|\nabla_\Gamma \mu(\tau)|^2+|\nabla_\Gamma \eta(\tau)|^2\big)\,\mathrm{d}S\mathrm{d}\tau
 +\int_0^t\! A(\tau)\int_\Gamma |(\eta-u)(\tau)|^2 \,\mathrm{d}S\mathrm{d}\tau  \nonumber\\
 &\quad  =\mathcal{E}(u_0,\varphi_0,v_0),\qquad \forall\, t>0,
\label{BELweak}
 \end{align}
 where $\mathcal{E}$ is the total free energy given by \eqref{ToE}. Moreover, the solution  $(u,\varphi,\mu,v,\eta)$
 is uniformly bounded on the whole interval $[0,+\infty)$ such that
 \begin{align}
 & \|u\|_{L^\infty(0,+\infty;L^2(\Omega))}  + \|\varphi\|_{L^\infty(0,+\infty;H^1(\Gamma))}  + \|v\|_{L^\infty(0,+\infty;L^2(\Gamma))} \notag\\
 &\quad + \|\nabla u\|_{L^2(0,+\infty;\bm{L}^2(\Omega))}
 + \|\nabla_\Gamma \mu\|_{L^2(0,+\infty;\bm{L}^2(\Gamma))}
 + \|\nabla_\Gamma \eta\|_{L^2(0,+\infty;\bm{L}^2(\Gamma))} \leq C,
 \label{ues2}
 \end{align}
 where the constant $C>0$ depends on $\|u_0\|_{L^2(\Omega)}$, $\|\varphi_0\|_{H^1(\Gamma)}$, $\|v_0\|_{L^2(\Gamma)}$, $\int_\Gamma F(\varphi_0)\,\mathrm{d}S$, $\langle\varphi_0\rangle_\Gamma$, $\Omega$, $\Gamma$ and the coefficients $D$, $\delta$.
\end{itemize}
\et
\begin{remark}
The condition \eqref{cini} on the initial data implies that the initial total energy $\mathcal{E}(u_0,\varphi_0,v_0)$ is finite. Besides, the requirement $|\langle\varphi_{0}\rangle_\Gamma|<1$ says that the initial state of the phase function $\varphi$ cannot be a pure state, namely, $\varphi_{0}$ is not identically equal to $-1$ or $1$ on $\Gamma$, under which the phase separation of saturated/saturated lipids will never happen. Concerning typical forms of the mass exchange term $q$, in the equilibrium case with \eqref{q1} (under the assumption \textbf{(H5)}), $q$ satisfies both the linear growth condition \eqref{ligro} and the global Lipschitz continuity, thus the global weak solution exists and is unique. For the non-equilibrium case with \eqref{q2m} (under the assumption \textbf{(H4)}), $q$ satisfies the linear growth condition but is only locally Lipschitz continuous, then uniqueness of weak solutions is not available.
\end{remark}

Our second result concerns the instantaneous regularity of weak solutions for arbitrary positive time.

\bt[Regularity] \label{thm:reg}
Suppose that the assumptions in Theorem \ref{thm:weak} are satisfied, $q$ is globally Lipschitz continuous on $\mathbb{R}^3$ and the nonlinearity $F$ fulfils \textbf{(H2)}.
\begin{itemize}
\item[(1)] For any $t_0\in (0,T)$, the global weak solution obtained in Theorem \ref{thm:weak}-(1) fulfils
\begin{align}
& \varphi \in L^{\infty}(t_0,T;H^3(\Gamma))\cap L^{2}(t_0,T;H^4(\Gamma))\cap H^{1}(t_0,T; H^1(\Gamma)),\notag \\
&\mu \in L^\infty(t_0,T;H^1(\Gamma))\cap  L^2(t_0,T;H^3(\Gamma))\cap H^{1}(t_0,T; (H^1(\Gamma))'),\notag \\
& v \in L^{\infty}(t_0,T;H^1(\Gamma))\cap L^{2}(t_0,T;H^2(\Gamma)) \cap H^1(t_0,T;L^2(\Gamma)).\notag
\end{align}
Moreover, there exists a constant $\sigma_{t_0,T}\in (0, 1)$ such that the phase function $\vp$ satisfies
\begin{align}
\|\vp(t)\|_{C(\Gamma)}\leq 1-\sigma_{t_0,T},\qquad \forall\,t\in [t_0,T].
\label{vp-sep0}
\end{align}
where $\sigma_{t_0,T}$ may depend on $\|u_0\|_{L^2(\Omega)}$, $\|\varphi_0\|_{H^1(\Gamma)}$, $\|v_0\|_{L^2(\Gamma)}$,  $\int_\Gamma F(\varphi_0)\,\mathrm{d}S$, $\langle\varphi_0\rangle_\Gamma$, $\Omega$, $\Gamma$, coefficients of the system, $t_0$ and $T$.
\item[(2)] In the equilibrium case, i.e., $q$ takes the form of \eqref{q1} under the assumption \textbf{(H5)}, for any $t_0>0$, we have
\begin{align}
& \|u\|_{L^\infty(t_0,+\infty; H^1(\Omega))}
+ \|\varphi\|_{ L^{\infty}(t_0,+\infty;H^3(\Gamma))}
+ \|\mu\|_{ L^\infty(t_0,+\infty;H^1(\Gamma))}\non\\
&\quad +\|v\|_{L^{\infty}(t_0,+\infty;H^1(\Gamma))}\leq C(t_0),\qquad \forall\, t\geq t_0,
\label{unit-es1}
\end{align}
for some $C(t_0)>0$ and the strict separation property \eqref{vp-sep0} holds on $[t_0,+\infty)$ with some $\sigma_{t_0}\in (0,1)$. Both constants $C(t_0)$ and $\sigma_{t_0}$ may depend on $\|u_0\|_{L^2(\Omega)}$, $\|\varphi_0\|_{H^1(\Gamma)}$, $\|v_0\|_{L^2(\Gamma)}$,  $\int_\Gamma F(\varphi_0)\,\mathrm{d}S$, $\langle\varphi_0\rangle_\Gamma$, $\Omega$, $\Gamma$, coefficients of the system and $t_0$, but are independent of the time $t$.
\end{itemize}
\et

Next, in the equilibrium case \eqref{q1}, under suitable decay property of the coefficient $A(t)$, we can prove that every global weak solution converges to a single equilibrium as $t\to+\infty$.

\begin{theorem}[Convergence to equilibrium]
\label{thm:conv}
Let the assumptions \textbf{(H1)}--\textbf{(H3)} be satisfied.
In addition, we suppose that
\begin{itemize}
\item[(1)] the nonlinearity $W$ is real analytic on $(-1,1)$;
\item[(2)] the mass exchange term $q$ takes the specific form of \eqref{q1} with its coefficient $A$ fulfilling the assumption \textbf{(H5)} and the following decay property
\begin{align}
A(t)\leq C(1+t)^{-\alpha} \qquad \text{for some}\ \alpha>1. \label{decayA}
\end{align}
\end{itemize}
Then for any initial datum $(u_0,\varphi_0,v_0)$ such that
 $$
 u_0\in L^2(\Omega),\quad \varphi_{0}\in H^1(\Gamma),\quad  v_{0}\in L^2(\Gamma)\quad \text{with}\quad \|  \varphi_{0} \|_{L^{\infty}(\Gamma)} \le 1\quad \text{and}\quad |\langle\varphi_{0}\rangle_\Gamma|<1,
 $$
 the unique global weak solution $(u,\varphi,\mu,v,\eta)$ to problem \eqref{1.a}--\eqref{ini2} satisfies
 \begin{align}
 &\lim_{t\to+\infty} \big(\|u(t)-u_\infty\|_{L^2(\Omega)} + \|\vp(t)-\vp_\infty\|_{H^2(\Gamma)}+ \|v(t)-v_\infty\|_{L^2(\Gamma)}\big)=0,\label{conv1}\\ &\lim_{t\to+\infty} \big(\|\mu(t)-\mu_\infty\|_{L^2(\Gamma)} +
 \|\eta(t)-\eta_\infty\|_{L^2(\Gamma)}\big)=0.
 \label{conv2}
 \end{align}
 Here, $\vp_\infty\in H^3(\Gamma)$ is a classical solution to the stationary problem
 \begin{align}
 \begin{cases}
 -\Delta_\Gamma \vp_\infty+W'(\vp_\infty)=\displaystyle{\frac{1}{|\Gamma|}\int_\Gamma W'(\vp_\infty)\,\mathrm{d}S,}\qquad \text{on}\ \Gamma,\\
  \text{subject to the constraint}\ \langle \vp_\infty\rangle_\Gamma = \langle \vp_0\rangle_\Gamma,
 \end{cases}
 \label{sta1}
 \end{align}
 $(u_\infty, \mu_\infty, \eta_\infty)$ are three constants, which  together with $v_\infty\in H^2(\Gamma)$  fulfil the following relations
 \begin{align}
 & u_\infty |\Omega|+ \int_\Gamma v_\infty\,\mathrm{d}S= \int_\Omega u_0\,\mathrm{d}x+ \int_\Gamma v_0\,\mathrm{d}S,\label{sta2}\\
 & \frac12\eta_\infty+\mu_\infty = \frac{1}{|\Gamma|}\int_\Gamma W'(\vp_\infty)\,\mathrm{d}S,\label{sta3}\\
 & 2\Delta_\Gamma v_\infty= \Delta_\Gamma\vp_\infty,\label{sta4}  \\
 & \eta_\infty|\Gamma|=\frac{4}{\delta}\int_\Gamma v_\infty\,\mathrm{d}S- \frac{2}{\delta}\int_\Gamma(1+\vp_0)\,\mathrm{d}S.\label{sta5}
 \end{align}
\end{theorem}
\begin{remark}
The decay property \eqref{decayA} of $A(t)$ plays a crucial role in the proof of Theorem \ref{thm:conv}.  It serves as a sufficient condition to obtain the \emph{uniqueness of asymptotic limit} for every global weak solution $(u(t),\vp(t), v(t), \mu(t), \eta(t))$  as $t\to +\infty$. If $A(t)$ satisfies the assumption \textbf{(H5)} and it is asymptotically
autonomous  but  does not decay fast enough (for instance,  $\lim_{t\to+\infty} A(t)=A_\infty$ for some constant $A_\infty \geq 0$, without a ``good'' convergence rate), then we can only prove a weaker result, that is, the \emph{sequential convergence} for $(u(t),\vp(t), v(t), \mu(t), \eta(t))$ as $t\to +\infty$ (see Proposition \ref{pro:ome}). Besides, the decay property \eqref{decayA} also enables us to derive some estimates on the decay rate of global weak solutions (see Corollary \ref{cor:rate}).
\end{remark}

\subsubsection{The reduced system \eqref{r1.a}--\eqref{r1.f}}

Concerning the reduced system \eqref{r1.a}--\eqref{r1.f}, we are mainly interested in the non-equilibrium case with a mass exchange term $q$ given by \eqref{q2}, i.e., a reaction type source term that is of biological interest \cite{GKRR16}.

\begin{definition}[Weak solutions]
\label{def:rweak}
Let $T>0$. We call the functions $(u,\varphi,\mu,v,\eta)$ a weak solution to problem \eqref{r1.a}--\eqref{rini2} on the interval $[0,T]$, if the following properties are fulfilled:
\begin{itemize}
\item[(1)] $(u,\varphi,\mu,v,\eta)$ satisfy
\begin{align}
& u\in L^\infty(0,T)\cap H^1(0,T),\notag\\
& \varphi \in L^{\infty}(0,T;H^1(\Gamma))\cap L^{4}(0,T;H^2(\Gamma))\cap L^2(0,T;W^{2,p}(\Gamma)) \cap H^{1}(0,T;(H^1(\Gamma))'),\notag \\
&\mu \in L^{2}(0,T;H^1(\Gamma)),\qquad F'(\varphi)\in L^2(0,T;L^p(\Gamma)),\notag \\
& v,\,\eta  \in L^{\infty}(0,T;L^2(\Gamma))\cap L^{2}(0,T;H^1(\Gamma)) \cap H^1(0,T;(H^1(\Gamma))'),\notag\\
&\varphi\in L^{\infty}(\Gamma\times (0,T))\ \textrm{and}\ \ |\varphi(x,t)|<1\ \ \textrm{a.e.\ on}\ \Gamma\times(0,T),\notag
\end{align}
where $p\in [2,+\infty)$ is arbitrary.
\item[(2)] For all $\zeta\in L^2(0,T;H^1(\Gamma))$, it holds
\begin{subequations}
			\begin{alignat}{3}
& |\Omega|u(t) + \int_\Gamma v(t)\,\mathrm{d}S=M,\quad \forall\,t\in [0,T], &\label{rew1.a}\\
& \int_0^T \left \langle \partial_t \varphi,\zeta\right \rangle_{(H^1(\Gamma))',\,H^1(\Gamma)}\,\mathrm{d}t
 = - \int_0^T \!\int_\Gamma \nabla_\Gamma \mu \cdot\nabla_\Gamma \zeta\,\mathrm{d}S \mathrm{d}t, &\label{rew1.c} \\
			&\ \, \mu=- \Delta_\Gamma \varphi + W'(\varphi)-\frac{1}{2}\eta,\qquad \text{a.e. on}\ \Gamma\times (0,T),&\label{rew1.d}\\
			&\int_0^T\left \langle\partial_t v,\zeta\right \rangle_{(H^1(\Gamma))',\,H^1(\Gamma)}\,\mathrm{d}t +  \int_0^T \!\int_\Gamma \nabla_\Gamma \eta\cdot \nabla_\Gamma \zeta\,\mathrm{d}S \mathrm{d}t = \int_0^T\! \int_\Gamma q \zeta \,\mathrm{d}S \mathrm{d}t,& \label{rew1.e}\\
            &\ \, \eta = \frac{2}{\delta}(2v-1-\varphi),\qquad \text{a.e. on}\ \Gamma\times (0,T),& \label{rew1.f}
\end{alignat}
\end{subequations}
where $M>0$ is the initial total cholesterol mass
$$
M=|\Omega|u_0+\int_\Gamma v_0\,\mathrm{d}S.
$$
\item[(3)] The continuity properties $u\in C([0,T])$, $\vp\in C([0,T];H^1(\Gamma))$, $v\in C([0,T];L^2(\Gamma))$ hold and the initial conditions \eqref{rini1}--\eqref{rini2} are satisfied.
\end{itemize}
\end{definition}

Then we have

\bt[Existence and uniqueness] \label{thm:rweak}
Suppose that the assumptions \textbf{(H1)}, \textbf{(H3)} are satisfied  and the mass exchange term $q$ takes the form of \eqref{q2} with $B_1, B_2> 0$.
 Let $T>0$ be arbitrary. For any initial datum $(u_0,\,\varphi_0,\,v_0)$ satisfying
\begin{align}
 & \varphi_{0}\in H^1(\Gamma),\ \  v_{0}\in L^2(\Gamma)\ \ \text{with}\ \ \|  \varphi_{0} \|_{L^{\infty}(\Gamma)} \le 1, \ \  |\langle\varphi_{0}\rangle_\Gamma|<1,
 \label{rcini1}\\
 & u_0= \frac{1}{|\Omega|}\Big(M- \int_\Gamma v_0\,\mathrm{d}S\Big)\in \Big[0,\ \frac{M}{|\Omega|}\Big],
 \label{rcini2}
 \end{align}
where $M>0$ is the initial total cholesterol mass,
 the initial boundary value problem \eqref{r1.a}--\eqref{rini2} admits a unique global weak solution $(u,\varphi,\mu,v,\eta)$ on $[0,T]$ in the sense of Definition \ref{def:rweak}. Moreover, we have
 \begin{align}
  u(t)\in \Big[0,\ \frac{M}{|\Omega|}\Big],\qquad \forall\, t\geq 0, \label{bdd-u}
 \end{align}
 and the following dissipative estimate
 \begin{align}
 &\|\vp(t)\|_{H^1(\Gamma)}^2+ \|v(t)\|_{L^2(\Gamma)}^2 +\int_{t}^{t+1} \big(\|\nabla_\Gamma \mu(\tau)\|^2_{\bm{L}^2(\Gamma)} + \|\nabla_\Gamma(2v(\tau)-\vp(\tau))\|_{\bm{L}^2(\Gamma)}^2\big)\,\mathrm{d}\tau\non\\
 &\quad \leq K_1e^{-K_2 t} + K_3,\qquad  \forall\, t\geq 0,
 \label{es-dissA}
\end{align}
where $K_1>0$ depends on $\|\vp_0\|_{H^1(\Gamma)}$, $\|v_0\|_{L^2(\Gamma)}$, $\int_\Gamma F(\vp_0)\,\mathrm{d}S$ and $\Gamma$, $K_2, K_3>0$ may depend on $M$, $B_1$, $B_2$, $\theta$, $\theta_0$, $\Omega$, $\Gamma$, $\delta$, $ \langle\varphi_0 \rangle_\Gamma$ and $\min_{r\in [-1,1]}F(r)$, but not on the norms $\|\vp_0\|_{H^1(\Gamma)}$ and $\|v_0\|_{L^2(\Gamma)}$.
\et

In view of the relation \eqref{rew1.a}, we can rewrite the reduced system \eqref{r1.a}--\eqref{r1.f} into a simpler form that only involves the unknowns $(\vp,v)$ on $\Gamma$:
\begin{subequations}
	\begin{alignat}{3}
	&\partial_t\varphi =\Delta_\Gamma \mu,
    &  \qquad\qquad\text{on}\ \Gamma\times(0,T),
    \label{nr1.c}\\
    &\mu=- \Delta_\Gamma \varphi + W'(\varphi)-\frac{1}{2}\eta,
    &  \qquad\qquad\text{on}\ \Gamma\times(0,T),
    \label{nr1.d}\\
    &\partial_t v=\Delta_\Gamma \eta +q(v),
    &  \qquad\qquad\text{on}\ \Gamma\times(0,T),
    \label{nr1.e}\\
    &\eta=\frac{4}{\delta}\Big(v-\frac{1+\varphi}{2}\Big),
    &  \qquad\qquad\text{on}\ \Gamma\times(0,T),
    \label{nr1.f}
	\end{alignat}
\end{subequations}
where under the specific choice of the mass exchange term \eqref{q2}, we find
\begin{align}
 & q(v) =\frac{B_1}{|\Omega|} \Big( M-\int_\Gamma v\,\mathrm{d}S  \Big)(1-v)-B_2 v.  \label{nq2}
\end{align}

Our next aim is to study the long-time behavior of global weak solutions to the reduced system \eqref{nr1.c}--\eqref{nr1.f} in terms of the global attractor.
To this end, for any $M>0$ and $m\in [0,1)$, we introduce the phase space
$$
\mathcal{V}_{M,m}=\left\{(\vp,v)\in H^1(\Gamma)\times L^2(\Gamma)\ \Big|\
\|\vp\|_{L^\infty(\Gamma)}\leq 1,\ |\langle \vp\rangle_\Gamma|\leq m,\ 0\leq \int_\Gamma v\,\mathrm{d}S\leq M\right\}
$$
endowed with the metric
$$
\mathrm{dist}_{\mathcal{V}_{M,m}}\big((\vp_1,v_1),(\vp_2,v_2)\big)= \|\vp_1-\vp_2\|_{H^1(\Gamma)} + \|v_1-v_2\|_{L^2(\Gamma)}.
$$
It is straightforward to verify that $\mathcal{V}_{M,m}$ is a complete metric space.

Theorem \ref{thm:rweak} enables us to show that
\begin{corollary}\label{cor:DS}
Suppose that the assumptions \textbf{(H1)}, \textbf{(H3)} are satisfied and $B_1,B_2>0$. The system \eqref{nr1.c}--\eqref{nq2} generates a strongly continuous semigroup
$$
\mathcal{S}(t):\ \mathcal{V}_{M,m}\to \mathcal{V}_{M,m} \quad \text{such that}\quad \mathcal{S}(t)(\vp_0,v_0)=(\vp,v),\quad \forall\, t\geq0,
$$
where $(\vp,v)$ is the unique global weak solution of  \eqref{nr1.c}--\eqref{nq2} corresponding to the initial datum $(\vp_0,v_0)$. Moreover, it holds
\begin{align}
\mathcal{S}(t)\in C(\mathcal{V}_{M,m},\,\mathcal{V}_{M,m}),\qquad \forall\, t\geq 0. \label{conti}
\end{align}
\end{corollary}

We then call $(\mathcal{S}(t), \mathcal{V}_{M,m})$ a dynamical system. Now let us introduce the notion of a global attractor. Roughly speaking, the global attractor is the smallest bounded set that attracts bounded sets of trajectories as time goes to infinity. It provides a suitable way to study  the long-time behavior of global solutions in the framework of infinite dimensional dynamical systems \cite{MZ08,T}.
\begin{definition}[Global attractor]
A nonempty compact set $\mathcal{A}_{M,m}\subset \mathcal{V}_{M,m}$ is said to be the global attractor for the dynamical system $(\mathcal{S}(t), \mathcal{V}_{M,m})$, if
\begin{itemize}
\item[(1)] $\mathcal{A}_{M,m}$ is invariant such that $\mathcal{S}(t)\mathcal{A}_{M,m} = \mathcal{A}_{M,m}$ for all $t\geq 0$;
\item[(2)] $\mathcal{A}_{M,m}$ is attracting, that is, for every bounded set $\mathcal{B}\subset \mathcal{V}_{M,m}$, it holds
    $$
    \lim_{t\to+\infty} \mathrm{dist}\big( \mathcal{S}(t)\mathcal{B}, \mathcal{A}_{M,m}\big)=0,
    $$
    where the Hausdorff semidistance between two subsets $\mathcal{A}, \mathcal{B}\subset \mathcal{V}_{M,m}$ is defined by
    $$
    \mathrm{dist}(\mathcal{A}, \mathcal{B})= \sup_{a\in \mathcal{A}}\inf_{b\in \mathcal{B}}\mathrm{dist}_{\mathcal{V}_{M,m}}(a,b).
    $$
\end{itemize}
\end{definition}

We are in a position to state the main result:
\begin{theorem}[Existence of the global attractor]\label{thm:att}
Suppose that the assumptions \textbf{(H1)}--\textbf{(H3)} are satisfied, $B_1,B_2>0$, $M>0$ and $m\in [0,1)$ are given constants. Then the dynamical system $(\mathcal{S}(t), \mathcal{V}_{M,m})$ possesses a unique connected global attractor $\mathcal{A}_{M,m}$, which is bounded in $H^3(\Gamma)\times H^1(\Gamma)$.
\end{theorem}
\begin{remark}
Due to the validity of the strict separation property and the regularity of weak solutions (see Lemma \ref{lem:comp}), we can apply the general approach in \cite{MZ08} to show that the dynamical system $(\mathcal{S}(t), \mathcal{V}_{M,m})$ indeed has an exponential attractor, which further implies that the global attractor $\mathcal{A}_{M,m}$ has finite fractal dimension in $\mathcal{V}_{M,m})$. This issue will be illustrated in a future study.
\end{remark}

\section{Well-posedness of the Full Bulk-Surface Coupled System}
\setcounter{equation}{0}
\label{sec:wellBS}
\subsection{Existence and uniqueness}
In what follows, we prove Theorem \ref{thm:weak} on the existence and uniqueness of a global weak solution to the full bulk-surface coupled system \eqref{1.a}--\eqref{ini2}. The proof consists of several steps.

\subsubsection{A regularized system}
Concerning the singular potential function $W$ satisfying \textbf{(H1)}, without loss of generality, we assume that $F(0)=0$. Then we approximate the singular derivative $F'$ as in \cite{FG12,GGM2017}:
\begin{equation}
F'_ {\kappa}(r)
=\left\{
\begin{aligned}
&F'(-1+\kappa) + F''(-1+\kappa)(r+1-\kappa),\quad r<-1+\kappa,\\
&F'(r),\qquad\qquad\qquad\qquad\qquad\qquad\qquad\ \ |r|\leq 1-\kappa,\\
&F'(1-\kappa) + F''(1-\kappa)(r-1+\kappa),\qquad\ \  r>1-\kappa,
\end{aligned}
\right.
\label{vF}
\end{equation}
for a sufficiently small constant $\kappa\in(0,\,r_0)$. Define
$$
F_{\kappa}(r)=\int_0^{r}F'_{\kappa}(s) \, \mathrm{d}s,\qquad W_\kappa(r)=F_{\kappa}(r)-\frac{\theta_0}{2}r^2,\qquad \forall\, r\in \mathbb{R}.
$$
We can verify that $F_{\kappa}''(r)\ge \theta>0$ and $F_ {\kappa}(r)\geq -L$ for all $r\in \mathbb{R}$, where $L>0$ is a constant independent of $\kappa$. Moreover, it holds $F_{\kappa}(r)\leq F(r)$ for $r\in [-1,1]$ (see, e.g., \cite{FG12}).

For every $\kappa\in(0,\,r_0)$, let us consider the regularized problem:
\begin{subequations}
	\begin{alignat}{3}
	& \partial_t u=D\Delta u,
    & \qquad\qquad  \text{in}\ \Omega\times(0,T), \label{a1.a}\\
    & D\partial_{\bm{n}} u=-q,
    &  \qquad\qquad\text{on}\ \Gamma\times(0,T), \label{a1.b}\\
    &\partial_t\varphi =\Delta_\Gamma \mu,
    &  \qquad\qquad\text{on}\ \Gamma\times(0,T), \label{a1.c}\\
    &\mu=-  \Delta_\Gamma \varphi + W_\kappa'(\varphi)-\frac{1}{2}\eta,
    &  \qquad\qquad\text{on}\ \Gamma\times(0,T),\label{a1.d}\\
    &\partial_t v=\Delta_\Gamma \eta +q,
    &  \qquad\qquad\text{on}\ \Gamma\times(0,T), \label{a1.e}\\
    &\eta=\frac{4}{\delta}\Big(v-\frac{1+\varphi}{2}\Big),
    &  \qquad\qquad\text{on}\ \Gamma\times(0,T),\label{a1.f}
	\end{alignat}
\end{subequations}
subject to the initial conditions
\begin{subequations}
\begin{alignat}{3}
&u|_{t=0}=u_{0}(x),&\qquad\qquad \text{in}\ \Omega,\label{aini1}\\
&\varphi|_{t=0}=\varphi_{0}(x), \quad  v|_{t=0}=v_{0}(x), &\qquad \qquad \text{on}\ \Gamma.
\label{aini2}
\end{alignat}
\end{subequations}
The system \eqref{a1.a}--\eqref{a1.f} is associated with the following modified free energy functional:
\begin{align}
\mathcal{E}_\kappa(u,\varphi,v)= \frac12\int_\Omega u^2\,\mathrm{d}x+ \mathcal{F}_\kappa(\varphi,v),
\label{mE1}
\end{align}
where
\begin{align}
\mathcal{F}_\kappa(\varphi,v) := \int_\Gamma \Big[\frac{1}{2}|\nabla_\Gamma\varphi|^2+W_\kappa(\varphi) +\frac{2}{\delta}\Big(v-\frac{1+\varphi}{2}\Big)^2\Big]\,\mathrm{d}S.
\label{mE2}
\end{align}

The existence of a global weak solution to the regularized problem \eqref{a1.a}--\eqref{aini2} can be proved by using a suitable Faedo-Galerkin approximation scheme as in \cite[Section 3]{AK20}. For the sake of completeness, we sketch the main steps below.

\subsubsection{The Galerkin approximation}
 As in \cite{AK20}, we denote by
$\{\xi_i\}_{i\in \mathbb{Z}^+}$ the family of eigenfunctions to the minus Laplace operator $-\Delta$ on $\Omega$ subject to the homogeneous Neumann boundary condition and by $\{\zeta_i\}_{i\in \mathbb{Z}^+}$ the family of eigenfunctions to the minus Laplace-Beltrami operator $-\Delta_\Gamma$ on $\Gamma$. Without loss of generality, we assume that $\|\xi_i\|_{L^2(\Omega)}=\|\zeta_i\|_{L^2(\Gamma)}=1$ for $i=1,2,\cdots$. Thus, $\{\xi_i\}_{i\in \mathbb{Z}^+}$ (resp. $\{\zeta_i\}_{i\in \mathbb{Z}^+}$) forms an orthonormal basis of $L^2(\Omega)$ (resp. $L^2(\Gamma)$)  and it is also an orthogonal basis of $H^1(\Omega)$ (resp. $H^1(\Gamma)$). For any given positive integer $n$, we look for functions $(u^n_\kappa,\varphi^n_\kappa,v^n_\kappa,\mu^n_\kappa,\eta^n_\kappa)$ of the form
$$
u^n_\kappa(x,t)=\sum_{i=1}^n a_i(t)\xi_i(x),\qquad \varphi^n_\kappa(x,t)=\sum_{i=1}^n b_i(t)\zeta_i(x),
$$
$$
v^n_\kappa(x,t)=\sum_{i=1}^n c_i(t)\zeta_i(x),\qquad \mu^n_\kappa(x,t)=\sum_{i=1}^n d_i(t)\zeta_i(x),
$$
and
\begin{align*}
 \eta^n_\kappa(x,t)
&= \frac{2}{\delta}\left(2v^n_\kappa-1-\varphi^n_\kappa \right)\\
&=\frac{2}{\delta}\left[2c_1(t)-\sqrt{|\Gamma|}-b_1(t)\right]\zeta_1(x)+
\frac{2}{\delta}\sum_{i=2}^n \big[2c_i(t)-b_i(t)\big]\zeta_i(x),
\end{align*}
which are elements of the finite dimensional spaces
$$
V_{n,\Omega}=\mathrm{span}\{\xi_1,\cdots,\xi_n\},\qquad V_{n,\Gamma}=\mathrm{span}\{\zeta_1,\cdots,\zeta_n\}.
$$
We denote by $\mathbb{P}_{V_{n,\Omega}}:\, L^2(\Omega)\to V_{n,\Omega}$ and $\mathbb{P}_{V_{n,\Gamma}}:\, L^2(\Gamma)\to V_{n,\Gamma}$ the $L^2$-orthogonal projections in $\Omega$ and on $\Gamma$, respectively.

The Galerkin scheme (written in a weak formulation) for the regularized problem \eqref{a1.a}--\eqref{aini2} is given by
\begin{subequations}
			\begin{alignat}{3}
&  \int_\Omega \partial_t u^n_\kappa \xi \,\mathrm{d}x
+ D\int_\Omega \nabla u^n_\kappa\cdot \nabla \xi\,\mathrm{d}x=
-  \int_\Gamma q^n_\kappa \xi \,\mathrm{d}S,&\label{rw1.a}\\
& \int_\Gamma \partial_t \varphi^n_\kappa \zeta\,\mathrm{d}S
 = - \int_\Gamma \nabla_\Gamma \mu^n_\kappa \cdot\nabla_\Gamma \zeta\,\mathrm{d}S, &\label{rw1.c} \\
			&\ \, \int_\Gamma \mu^n_\kappa\zeta \,\mathrm{d}S
=\int_\Gamma \Big[\nabla_\Gamma \varphi^n_\kappa\cdot\nabla_\Gamma\zeta + W_\kappa'(\varphi^n_\kappa) \zeta -\frac{1}{2}\eta^n_\kappa \zeta\Big]\,\mathrm{d}S,&\label{rw1.d}\\
			&\int_\Gamma \partial_t v^n_\kappa \zeta\,\mathrm{d}S +  \int_\Gamma \nabla_\Gamma \eta^n_\kappa\cdot \nabla_\Gamma \zeta\,\mathrm{d}S  =  \int_\Gamma q^n_\kappa \zeta \,\mathrm{d}S,& \label{rw1.e}
\end{alignat}
\end{subequations}
for any test functions $\xi\in V_{n,\Omega}$ and $\zeta\in V_{n,\Gamma}$, where
$$
q^n_\kappa=q(u^n_\kappa,\varphi^n_\kappa,v^n_\kappa).
$$
The discretized system \eqref{rw1.a}--\eqref{rw1.e} is supplemented with the initial conditions
\begin{subequations}
\begin{alignat}{3}
&u^n_\kappa|_{t=0}=\mathbb{P}_{V_{n,\Omega}}u_{0}(x),&\qquad\qquad \text{in}\ \Omega,\label{rwini1}\\
&\varphi^n_\kappa|_{t=0}=\mathbb{P}_{V_{n,\Gamma}}\varphi_{0}(x), \quad  v^n_\kappa|_{t=0}=\mathbb{P}_{V_{n,\Gamma}}v_{0}(x), &\qquad \qquad \text{on}\ \Gamma.
\label{rwini2}
\end{alignat}
\end{subequations}
Then it follows from the facts $u_0\in L^2(\Omega)$, $\varphi_0\in H^1(\Gamma)$ and $v_0\in L^2(\Gamma)$ that as $n\to +\infty$,
\begin{align}
\|u_\kappa^n(0)-u_0\|_{L^2(\Omega)}\to 0, \quad \|\varphi_\kappa^n(0)-\varphi_0\|_{H^1(\Gamma)}\to 0,\quad
\|v_\kappa^n(0)-v_0\|_{L^2(\Gamma)}\to 0.
\label{appini}
\end{align}

Taking $\xi=\xi_i$, $\zeta=\zeta_i$, $i=1,\cdots,n$ in \eqref{rw1.a}--\eqref{rw1.e}, respectively, we obtain a system of ordinary differential equations for the time-dependent coefficients $\{a_i(t),\,b_i(t),\,c_i(t),\,d_i(t)\}$, $i=1,\cdots,n$, subject to the initial conditions
\begin{align}
a_i(0)=\int_\Omega u_0\xi_i\,\mathrm{d}x,\quad b_i(0)=\int_\Gamma \varphi_0\zeta_i\,\mathrm{d}S,\quad c_i(0)=\int_\Gamma v_0\zeta_i\,\mathrm{d}S,\qquad i=1,\cdots,n.\label{rwini}
\end{align}
Thanks to the assumptions
\textbf{(H1)}, \textbf{(H3)} and the continuity of $q$, we can apply Peano's existence theorem to conclude:
\begin{proposition}
\label{exe:gar}
Let $\kappa\in (0,r_0)$ be a fixed parameter.
For every positive integer $n$, there exists $T_n\in (0,T]$ such that the discretized problem \eqref{rw1.a}--\eqref{rwini2} admits at least one local solution on $[0, T_n]$ denoted by $(u^n_\kappa,\varphi^n_\kappa,v^n_\kappa,\mu^n_\kappa,\eta^n_\kappa)$, which is determined by the functions $\{a_i(t),\,b_i(t),\,c_i(t),\,d_i(t)\}\in C^1([0,T_n])$, $i=1,\cdots,n$.
\end{proposition}

\subsubsection{Uniform estimates w.r.t $n$}
We shall derive some necessary estimates for the approximate solutions  $(u^n_\kappa,\varphi^n_\kappa,v^n_\kappa,\mu^n_\kappa,\eta^n_\kappa)$ that are uniform with respect to the approximating parameter $n$ (and also $T_n$). At this stage, we fix the parameter $\kappa$.
\medskip

\textbf{First estimate}.
Taking $\xi=1$ in \eqref{rw1.a} and $\zeta=1$ in \eqref{rw1.c}, \eqref{rw1.e}, respectively, integrating on $(0,t)\subset (0,T_n]$, we get
\begin{align}
&\int_\Omega u^n_\kappa(x,t)\,\mathrm{d}x+ \int_\Gamma v^n_\kappa(x,t)\,\mathrm{d}S = \int_\Omega u_0(x)\,\mathrm{d}x+ \int_\Gamma v_0(x)\,\mathrm{d}S,\quad \forall\, t\in[0,T_n],\label{rwmass1}\\
&\int_\Gamma \varphi^n_\kappa(x,t)\,\mathrm{d}S=\int_\Gamma \varphi_0(x)\,\mathrm{d}S,\quad \forall\, t\in[0,T_n].\label{rwmass2}
\end{align}
These identities imply that the approximate solutions keep the mass conservation properties along evolution.\medskip

\textbf{Second estimate}.
By an argument similar to that for \cite[(3.5)]{AK20}, we can obtain the following energy identity:
\begin{align}
& \frac{\mathrm{d}}{\mathrm{d}t} \mathcal{E}_\kappa(u^n_\kappa,\varphi^n_\kappa,v^n_\kappa)
+ D\int_\Omega |\nabla u^n_\kappa|^2\,\mathrm{d}x +\int_\Gamma \big(|\nabla_\Gamma \mu^n_\kappa|^2
+|\nabla_\Gamma \eta^n_\kappa|^2\big)\,\mathrm{d}S \nonumber\\
&\quad  =\int_\Gamma q^n_\kappa (\eta^n_\kappa-u^n_\kappa)\,\mathrm{d}S,\qquad \forall\, t\in(0,T_n),
\label{BEL3}
\end{align}
with
$$
\mathcal{E}_\kappa(u^n_\kappa,\varphi^n_\kappa,v^n_\kappa) =\frac12\int_\Omega |u^n_\kappa|^2\,\mathrm{d}x
+ \frac12\int_\Gamma |\nabla_\Gamma \varphi^n_\kappa|^2\,\mathrm{d}S
+ \int_\Gamma W_\kappa(\varphi_\kappa^n)\,\mathrm{d}S
+\frac{\delta}{8}\int_\Gamma |\eta_\kappa^n|^2\,\mathrm{d}S.
$$

We consider here a general mass exchange term $q$ under the linear growth condition \eqref{ligro}. In this case, the right-hand side of \eqref{BEL3} can be estimated as
\begin{align*}
\left|\int_\Gamma q^n_\kappa (\eta^n_\kappa-u^n_\kappa)\,\mathrm{d}S\right|
&\leq \int_\Gamma |\eta_\kappa^n|^2\,\mathrm{d}S
+\widetilde{C} \Big(1+  \int_\Gamma |u_\kappa^n|^2\,\mathrm{d}S
+ \int_\Gamma |\varphi_\kappa^n|^2\,\mathrm{d}S
+\int_\Gamma |v_\kappa^n|^2\,\mathrm{d}S \Big),
\end{align*}
for some $\widetilde{C}>0$ that may depend on $|\Gamma|$. 
As in \cite{AK20},  we can apply the interpolation inequality \eqref{inter1} and Young's inequality to derive the following estimate
\begin{align}
\int_\Gamma |u_\kappa^n|^2\,\mathrm{d}S
& \leq C\| u_\kappa^n\|_{H^1(\Omega)}\|u_\kappa^n\|_{L^2(\Omega)}\non\\
& \leq \gamma \|\nabla u_\kappa^n\|_{\bm{L}^2(\Omega)}^2+ C(1+\gamma^{-1})\|u_\kappa^n\|_{L^2(\Omega)}^2,
\label{uknL2}
\end{align}
where $\gamma>0$ is an arbitrary small constant. Besides, it holds
\begin{align*}
\int_\Gamma |v_\kappa^n|^2\,\mathrm{d}S
&\leq C\big(1+\|\eta_\kappa^n\|_{L^2(\Gamma)}^2+ \|\varphi_\kappa^n\|_{L^2(\Gamma)}^2\big).
\end{align*}
Taking $\zeta= \varphi_\kappa^n$ in \eqref{rw1.c}, we obtain
\begin{align}
&\frac12\frac{\mathrm{d}}{\mathrm{d}t}\|\varphi_\kappa^n \|_{L^2(\Gamma)}^2 + \|\Delta_\Gamma \varphi_\kappa^n\|_{L^2(\Gamma)}^2\non\\
&\quad = -\int_\Gamma   W_\kappa''(\varphi^n_\kappa) |\nabla_\Gamma \varphi^n_\kappa|^2\,\mathrm{d}S -\frac{1}{2}\int_\Gamma \eta^n_\kappa  \Delta_\Gamma \varphi^n_\kappa\,\mathrm{d}S\non\\
&\quad \leq (\theta_0-\theta) \|\nabla_\Gamma \varphi_\kappa^n\|_{\bm{L}^2(\Gamma)}^2 + \frac12 \|\Delta_\Gamma \varphi_\kappa^n\|_{L^2(\Gamma)} \|\eta_\kappa^n\|_{L^2(\Gamma)}\non\\
&\quad \leq (|\theta_0|+|\theta|) \|\Delta_\Gamma \varphi_\kappa^n\|_{L^2(\Gamma)}\| \varphi_\kappa^n\|_{L^2(\Gamma)}  + \frac12 \|\Delta_\Gamma \varphi_\kappa^n\|_{L^2(\Gamma)} \|\eta_\kappa^n\|_{L^2(\Gamma)}\non\\
&\quad \leq \frac12 \|\Delta_\Gamma \varphi_\kappa^n\|_{L^2(\Gamma)}^2
+C \big(\|\varphi_\kappa^n\|_{L^2(\Gamma)}^2+ \|\eta_\kappa^n\|_{L^2(\Gamma)}^2\big).
\label{esL2v}
\end{align}
From the definition of $W_\kappa$, we also find
\begin{align*}
\int_\Gamma W_\kappa(\varphi_\kappa^n)\,\mathrm{d}S
&\geq -L|\Gamma|-\frac{\theta_0}{2}\|\varphi_\kappa^n\|_{L^2(\Gamma)}^2.
\end{align*}
Hence, taking $0<\gamma\leq D/(2\widetilde{C})$ in \eqref{uknL2}, we can deduce from the above estimates and the differential inequalities \eqref{BEL3}, \eqref{esL2v} that
\begin{align}
& \frac{\mathrm{d}}{\mathrm{d}t} \widetilde{\mathcal{E}}_\kappa(u^n_\kappa,\varphi^n_\kappa,v^n_\kappa)
+ \frac{D}{2}\|\nabla u^n_\kappa\|_{\bm{L}^2(\Omega)}^2
+\|\nabla_\Gamma \mu^n_\kappa\|_{\bm{L}^2(\Gamma)}^2
+\|\nabla_\Gamma \eta^n_\kappa\|_{\bm{L}^2(\Gamma)}^2 \non\\
&\qquad + (|\theta_0|+1) \|\Delta_\Gamma \varphi_\kappa^n\|_{L^2(\Gamma)}^2\nonumber\\
&\quad  \leq C\widetilde{\mathcal{E}}_\kappa(u^n_\kappa,\varphi^n_\kappa,v^n_\kappa),\qquad \forall\, t\in(0,T_n),
\label{BEL3a}
\end{align}
where
$$
\widetilde{\mathcal{E}}_\kappa(u^n_\kappa,\varphi^n_\kappa,v^n_\kappa)
=\mathcal{E}_\kappa(u^n_\kappa,\varphi^n_\kappa,v^n_\kappa) + (|\theta_0|+1)\|\varphi_\kappa^n \|_{L^2(\Gamma)}^2
+ L|\Gamma| \geq 0.
$$
Applying Gronwall's lemma, we can conclude from \eqref{BEL3a} that
\begin{align}
& \widetilde{\mathcal{E}}_\kappa(u^n_\kappa(t), \varphi^n_\kappa(t),v^n_\kappa(t))   \non\\
&\qquad +\int_0^t \Big(D\|\nabla u^n_\kappa(\tau)\|_{\bm{L}^2(\Omega)}^2
+\|\nabla_\Gamma \mu^n_\kappa(\tau)\|_{\bm{L}^2(\Gamma)}^2+\|\nabla_\Gamma \eta^n_\kappa(\tau)\|_{\bm{L}^2(\Gamma)}^2+ \|\Delta_\Gamma \varphi_\kappa^n(\tau)\|_{L^2(\Gamma)}^2\Big)\,\mathrm{d}\tau   \non\\
&\quad \leq C(T)\, \widetilde{\mathcal{E}}_\kappa(u^n_\kappa(0), \varphi^n_\kappa(0),v^n_\kappa(0)),\qquad \forall\, t\in (0,T_n],
\label{esL4}
\end{align}
where $C(T)>0$ may depend on $|\Omega|$, $|\Gamma|$, the structural coefficients of the system and $T$, but not on $n$ and $\kappa$.
From the approximation \eqref{vF}, we see that
$$
F_\kappa(r)\leq C(\kappa)(1+r^2),\qquad \forall\, r\in \mathbb{R},
$$
where $C(\kappa)>0$ may depend on $\kappa$. This upper bound and \eqref{appini} yield
$$
\widetilde{\mathcal{E}}_\kappa(u^n_\kappa(0), \varphi^n_\kappa(0),v^n_\kappa(0))\leq C,
$$
where $C>0$ depends on $\|u_0\|_{L^2(\Omega)}$, $\|\varphi_0\|_{H^1(\Gamma)}$, $\|v_0\|_{L^2(\Gamma)}$, $|\Omega|$, $|\Gamma|$, the structural coefficients and $\kappa$, but not on $n$.
As a consequence, we deduce from \eqref{esL4} that
\begin{align}
&  \|u_\kappa^n(t)\|_{L^2(\Omega)}^2
+ \|\varphi_\kappa^n(t)\|_{H^1(\Gamma)}^2
+ \|v_\kappa^n(t)\|_{L^2(\Gamma)}^2
+ \|\eta_\kappa^n(t)\|_{L^2(\Gamma)}^2 \non\\
&\qquad  +\int_0^t \Big(D\|\nabla u^n_\kappa(\tau)\|_{\bm{L}^2(\Omega)}^2
+\|\nabla_\Gamma \mu^n_\kappa(\tau)\|_{\bm{L}^2(\Gamma)}^2+\|\nabla_\Gamma \eta^n_\kappa(\tau)\|_{\bm{L}^2(\Gamma)}^2+ \|\Delta_\Gamma \varphi_\kappa^n(\tau)\|_{L^2(\Gamma)}^2\Big)\,\mathrm{d}\tau \non\\
&\quad \leq C(T),\qquad \forall\, t\in (0,T_n].
\label{esL5}
\end{align}
Besides, exploiting the elliptic regularity for $\varphi_\kappa^n$, we get
\begin{align}
\int_0^t   \|\varphi_\kappa^n(\tau)\|_{H^2(\Gamma)}^2 \,\mathrm{d}\tau &\leq
\int_0^t  \big( \|\Delta_\Gamma \varphi_\kappa^n(\tau)\|_{L^2(\Gamma)}^2+ \|\varphi_\kappa^n(\tau)\|_{L^2(\Gamma)}^2 \big) \,\mathrm{d}\tau \non\\
& \leq C(T),\qquad \forall\, t\in (0,T_n].
\label{esvaH2}
\end{align}

\textbf{Third estimate}.
Taking $\zeta=1$ in \eqref{rw1.d}, we have
\begin{align}
\int_\Gamma \mu_\kappa^n\,\mathrm{d}S= \int_\Gamma W'_\kappa(\varphi_\kappa^n)\,\mathrm{d}S-\frac12\int_\Gamma \eta_\kappa^n\,\mathrm{d}S.
\label{mu-mean1}
\end{align}
It easily follows from \eqref{esL5} that
$$
\sup_{t\in[0,T_n]}\left|\int_\Gamma \eta_\kappa^n(t)\,\mathrm{d}S\right|\leq |\Gamma|^\frac12\|\eta_\kappa^n(t)\|_{L^2(\Gamma)}\leq C(T).
$$
Next, by \eqref{vF} (which gives the linear growth of $F'_\kappa$) and \eqref{esL5}, we find
\begin{align*}
\sup_{t\in[0,T_n]} \left|\int_\Gamma  F'_\kappa(\vp_\kappa^n(t))\,\mathrm{d}S\right|
 \leq C |\Gamma| + C \sup_{t\in[0,T_n]} \|\vp_\kappa^n(t)\|_{L^1(\Gamma)}   \leq C(T).
\end{align*}
The above estimates together with \eqref{esL5}, \eqref{mu-mean1} and Poincar\'e's inequality yield
\begin{align}
\int_0^t \|\mu^n_\kappa(\tau)\|_{H^1(\Gamma)}^2\,\mathrm{d}\tau
&\leq C\int_0^t \|\nabla_\Gamma \mu^n_\kappa(\tau)\|_{\bm{L}^2(\Gamma)}^2\,\mathrm{d}\tau +C \int_0^t \left|\int_\Gamma \mu_\kappa^n(\tau)\,\mathrm{d}S\right|^2\,\mathrm{d}\tau\non\\
&\leq C(T),\qquad \forall\, t\in (0,T_n].
\label{esmuh1}
\end{align}
We remark that the estimates \eqref{esL5}, \eqref{esvaH2} and \eqref{esmuh1} are independent of $n$ (as well as $T_n$), but they may depend on $\kappa$.\medskip

\textbf{Fourth estimate}.
We now derive estimates for the time derivatives as in \cite{AK20}.
Since for any $\xi\in H^1(\Omega)$,
\begin{align*}
 &\left|\langle \partial_t u_\kappa^n, \xi\rangle_{(H^1(\Omega))',H^1(\Omega)}\right|
 =\left|\int_\Omega \partial_t u^n_\kappa \xi \,\mathrm{d}x\right|\non\\
 &\quad \leq D\|\nabla u^n_\kappa\|_{\bm{L}^2(\Omega)}\|\nabla \xi\|_{\bm{L}^2(\Omega)} + \|q_\kappa^n\|_{L^2(\Gamma)}\|\xi\|_{L^2(\Gamma)}\non\\
 &\quad \leq C\big(D\|\nabla u^n_\kappa\|_{\bm{L}^2(\Omega)}+ \|u_\kappa^n\|_{L^2(\Gamma)} +\|\varphi_\kappa^n\|_{L^2(\Gamma)}+  \|v_\kappa^n\|_{L^2(\Gamma)}+1\big)\|\xi\|_{H^1(\Omega)},
\end{align*}
then we infer from \eqref{uknL2} and \eqref{esL5} that
\begin{align}
\int_0^t \|\partial_t u_\kappa^n(\tau)\|_{(H^1(\Omega))'}^2\,\mathrm{d}\tau\leq  C(T),
\qquad \forall\, t\in (0,T_n].
\label{estu}
\end{align}
In a similar manner, we can obtain
\begin{align}
\int_0^t \|\partial_t \varphi_\kappa^n(\tau)\|_{(H^1(\Gamma))'}^2\,\mathrm{d}\tau
+ \int_0^t \|\partial_t v_\kappa^n(\tau)\|_{(H^1(\Gamma))'}^2\,\mathrm{d}\tau\leq  C(T),
\qquad \forall\, t\in (0,T_n].
\label{estvv}
\end{align}
From the derivation of \eqref{estu} and \eqref{estvv}, we see that the upper bound $C(T)$ is again independent of the index $n$, but it may depend on $\kappa$.\medskip

\subsubsection{Extension of approximate solutions to $[0,T]$}
Using the definition of the approximate solutions as well as the uniform estimate \eqref{esL5}, we find that the solutions $\{a_i(t),\,b_i(t),\,c_i(t),\,d_i(t)\}$, $i=1,\cdots,n$, are bounded on the time interval $[0, T_n]\subset [0,T]$
by a positive constant that neither depends on $T_n$ nor on $n$.
Consequently, thanks to the classical ODE theory, the
solutions $\{a_i(t),\,b_i(t),\,c_i(t),\,d_i(t)\}$, $i=1,\cdots,n$, can be extended to $[0, T]$.
Thus, we are able to construct the approximate solution $(u^n_\kappa,\varphi^n_\kappa,v^n_\kappa,\mu^n_\kappa,\eta^n_\kappa)$ on the whole time interval $[0,T]$, which satisfies the weak formulation \eqref{rw1.a}--\eqref{rw1.e} everywhere in $[0, T]$
with the uniform estimates \eqref{esL5}, \eqref{esvaH2}, \eqref{esmuh1}, \eqref{estu} and \eqref{estvv} with respective to $n$.

\subsubsection{Passage to the limit as $n\to +\infty$}
 In view of the uniform estimates \eqref{esL5}, \eqref{esvaH2}, \eqref{esmuh1}, \eqref{estu} and \eqref{estvv}, we can apply a standard compactness argument involving the Banach-Alaoglu theorem and the Aubin-Lions-Simon lemma to find a convergent subsequence  $\{(u^n_\kappa,\varphi^n_\kappa,v^n_\kappa,\mu^n_\kappa,\eta^n_\kappa)\}$ (not relabelled for simplicity), whose limit under $n\to +\infty$ denoted by $(u_\kappa,\varphi_\kappa,v_\kappa,\mu_\kappa,\eta_\kappa)$, fulfil the following regularity properties:
\begin{align}
& u_\kappa\in L^\infty(0,T;L^2(\Omega))\cap L^2(0,T;H^1(\Omega))\cap H^1(0,T;(H^1(\Omega))'),\notag\\
& \varphi_\kappa \in L^{\infty}(0,T;H^1(\Gamma))\cap L^{2}(0,T;H^2(\Gamma))\cap H^{1}(0,T;(H^1(\Gamma))'),\notag \\
&\mu_\kappa \in L^{2}(0,T;H^1(\Gamma)),\notag \\
& v_\kappa,\,\eta_\kappa  \in L^{\infty}(0,T;L^2(\Gamma))\cap L^{2}(0,T;H^1(\Gamma)) \cap H^1(0,T;(H^1(\Gamma))').\notag
\end{align}
Moreover, by the construction of $F'_\kappa$, we have $F'_\kappa(\vp_\kappa)\in L^\infty(0,T;H^1(\Gamma))$. The functions  $(u_\kappa,\varphi_\kappa,v_\kappa,\mu_\kappa,\eta_\kappa)$ satisfy the weak formulation of the regularized system \eqref{a1.a}--\eqref{a1.f} such that
\begin{subequations}
			\begin{alignat}{3}
& \int_0^T  \langle \partial_t u_\kappa,\xi \rangle_{(H^1(\Omega))',\,H^1(\Omega)}\,\mathrm{d}t
+ D\int_0^T \!\int_\Omega \nabla u_\kappa\cdot \nabla \xi\,\mathrm{d}x\mathrm{d}t=
- \int_0^T\! \int_\Gamma q_\kappa \xi \,\mathrm{d}S \mathrm{d}t,&\label{rraw1.a}\\
& \int_0^T \left \langle \partial_t \varphi_\kappa,\zeta\right \rangle_{(H^1(\Gamma))',\,H^1(\Gamma)}\,\mathrm{d}t
 = - \int_0^T \!\int_\Gamma \nabla_\Gamma \mu_\kappa \cdot\nabla_\Gamma \zeta\,\mathrm{d}S \mathrm{d}t, &\label{rraw1.c} \\
			&\ \, \mu_\kappa=- \Delta_\Gamma \varphi_\kappa + W_\kappa'(\varphi_\kappa)-\frac{1}{2}\eta_\kappa,\qquad \text{a.e. on}\ \Gamma\times (0,T),&\label{rraw1.d}\\
			&\int_0^T\left \langle\partial_t v_\kappa,\zeta\right \rangle_{(H^1(\Gamma))',\,H^1(\Gamma)}\,\mathrm{d}t +  \int_0^T \!\int_\Gamma \nabla_\Gamma \eta_\kappa\cdot \nabla_\Gamma \zeta\,\mathrm{d}S \mathrm{d}t = \int_0^T\! \int_\Gamma q_\kappa \zeta \,\mathrm{d}S \mathrm{d}t,& \label{rraw1.e}\\
            &\ \, \eta_\kappa = \frac{2}{\delta}\left(2v_\kappa-1-\varphi_\kappa\right),\qquad \text{a.e. on}\ \Gamma\times (0,T),& \label{rraw1.f}
\end{alignat}
\end{subequations}
for any test functions $\xi\in L^2(0,T;H^1(\Omega))$ and $\zeta\in L^2(0,T;H^1(\Gamma))$, where
$$q_\kappa=q(u_\kappa,\varphi_\kappa,v_\kappa).$$
Besides, the initial condition are attained
\begin{subequations}
\begin{alignat}{3}
&u_\kappa|_{t=0}=u_{0}(x),&\qquad\qquad \text{in}\ \Omega,\label{rrawini1}\\
&\varphi_\kappa|_{t=0}=\varphi_{0}(x), \quad  v_\kappa|_{t=0}=v_{0}(x), &\qquad \qquad \text{on}\ \Gamma.
\label{rrawini2}
\end{alignat}
\end{subequations}
Since the above procedure is parallel to the proof of \cite[Theorem 2.3]{AK20}, we omit the details here.

\subsubsection{Passage to the limit as $\kappa\to 0$}
Our next aim is to derive uniform estimates for the approximate solutions $(u_\kappa,\varphi_\kappa,v_\kappa,\mu_\kappa,\eta_\kappa)$ that are independent of the parameter $\kappa$.

Applying the argument similar to that in \cite{CFG}, we get
$$
\langle \partial_t \vp_\kappa, F'_\kappa(\vp_\kappa)\rangle_{(H^1(\Omega))',\,H^1(\Omega)} = \frac{\mathrm{d}}{\mathrm{d}t} \int_\Gamma F_\kappa(\vp_\kappa)\,\mathrm{d}S,
$$
for almost all $t\in (0,T)$. Thus, taking $\xi= u_\kappa$ in \eqref{rraw1.a}, $\zeta=\mu_\kappa$ in \eqref{rraw1.c} and $\zeta=\eta_\kappa$ in \eqref{rraw1.e}, adding the resultants together and integrating over $(0,T)$ (using also the chain rule \cite[Lemma 7.3]{RoT05}), we can obtain the following energy equality (cf. \eqref{BEL3}):
\begin{align}
& \mathcal{E}_\kappa(u_\kappa(t),\varphi_\kappa(t),v_\kappa(t))
+ \int_0^t\left[D\int_\Omega |\nabla u_\kappa(\tau)|^2\,\mathrm{d}x +\int_\Gamma \big(|\nabla_\Gamma \mu_\kappa(\tau)|^2+|\nabla_\Gamma \eta_\kappa(\tau)|^2\big)\,\mathrm{d}S\right]\mathrm{d}\tau \nonumber\\
&\quad  = \mathcal{E}_\kappa(u_\kappa(0),\varphi_\kappa(0),v_\kappa(0))+ \int_0^t\!\int_\Gamma q_\kappa(\tau) (\eta_\kappa(\tau)-u_\kappa(\tau))\,\mathrm{d}S\mathrm{d}\tau,\qquad \forall\,t\in (0,T].
\label{BEL4}
\end{align}
The second term on the right-hand side of \eqref{BEL4} can be estimated with the help of the linear growth assumption on $q$ like before.
Thanks to the regularity of $\varphi_\kappa$, it follows that $\nabla_\Gamma W_\kappa'(\varphi_\kappa)=W_\kappa''(\varphi_\kappa)\nabla_\Gamma\varphi_\kappa$ for almost all $(x,t)\in \Gamma\times (0,T)$. Then taking $\zeta=\varphi_\kappa$ in \eqref{rraw1.c}, similar to \eqref{esL2v}, we can derive
\begin{align}
&\|\varphi_\kappa(t) \|_{L^2(\Gamma)}^2 + \int_0^t\|\Delta_\Gamma \varphi_\kappa(\tau)\|_{L^2(\Gamma)}^2\,\mathrm{d}\tau \non\\
& \quad \leq \|\varphi_\kappa(0) \|_{L^2(\Gamma)}^2
+C \int_0^t \big(\|\varphi_\kappa(\tau)\|_{L^2(\Gamma)}^2+ \|\eta_\kappa(\tau)\|_{L^2(\Gamma)}^2\big)\mathrm{d}\tau,\qquad \forall\,t\in (0,T].
\label{esL2vb}
\end{align}
Then it follows from \eqref{BEL4} and \eqref{esL2vb} that
\begin{align}
&  \widetilde{\mathcal{E}}_\kappa(u_\kappa(t),\varphi_\kappa(t),v_\kappa(t))
+ \int_0^t\!\left[ \frac{D}{2}\|\nabla u_\kappa(\tau)\|_{\bm{L}^2(\Omega)}^2
+\|\nabla_\Gamma \mu_\kappa(\tau)\|_{\bm{L}^2(\Gamma)}^2+\|\nabla_\Gamma \eta_\kappa(\tau)\|_{\bm{L}^2(\Gamma)}^2\right]\,\mathrm{d}\tau \non\\
&\qquad + (|\theta_0|+1) \int_0^t\!\|\Delta_\Gamma \varphi_\kappa(\tau)\|_{L^2(\Gamma)}^2\,\mathrm{d}\tau \nonumber\\
&\quad  \leq \widetilde{\mathcal{E}}_\kappa(u_\kappa(0),\varphi_\kappa(0),v_\kappa(0))
+ C\int_0^t \widetilde{\mathcal{E}}_\kappa(u_\kappa(\tau), \varphi_\kappa(\tau),v_\kappa(\tau))\,\mathrm{d}\tau,\qquad \forall\, t\in(0,T],
\label{BEL4a}
\end{align}
where
$$
\widetilde{\mathcal{E}}_\kappa(u_\kappa,\varphi_\kappa,v_\kappa)
=\mathcal{E}_\kappa(u_\kappa,\varphi_\kappa,v_\kappa) + (|\theta_0|+1)\|\varphi_\kappa \|_{L^2(\Gamma)}^2
+ L|\Gamma| \geq 0.
$$

On account of the assumption \textbf{(H1)}, we find
\begin{align}
F_\kappa(r)\leq F(r),\qquad \forall\, r\in [-1,1].
\label{fkk}
\end{align}
Besides, the assumption $\|\varphi_0\|_{L^\infty(\Gamma)}\leq 1$ implies that
$$
\int_\Gamma F_\kappa(\varphi_0)\,\mathrm{d}S\leq \int_\Gamma F(\varphi_0)\,\mathrm{d}S.
$$
As a consequence, we obtain the uniform upper bound
\begin{align}
\mathcal{E}_\kappa(u_\kappa(0),\varphi_\kappa(0),v_\kappa(0))
=\mathcal{E}_\kappa(u_0,\varphi_0,v_0)\leq \mathcal{E}(u_0,\varphi_0,v_0)\leq C,
\label{iniE0}
\end{align}
where $C>0$ depends on $\|u_0\|_{L^2(\Omega)}$, $\|\varphi_0\|_{H^1(\Gamma)}$, $\|v_0\|_{L^2(\Gamma)}$, $\int_\Gamma F(\varphi_0)\,\mathrm{d}S$, $\Omega$, $\Gamma$, the structural coefficients, but not on the parameter $\kappa$.
Applying Gronwall's lemma to \eqref{BEL4a}, we obtain the following $\kappa$-independent estimate:
\begin{align}
&  \|u_\kappa(t)\|_{L^2(\Omega)}^2
+ \|\varphi_\kappa(t)\|_{H^1(\Gamma)}^2
+ \|v_\kappa(t)\|_{L^2(\Gamma)}^2
+ \|\eta_\kappa(t)\|_{L^2(\Gamma)}^2 \non\\
&\qquad  +\int_0^t \Big(D\|\nabla u_\kappa(\tau)\|_{\bm{L}^2(\Omega)}^2
+\|\nabla_\Gamma \mu_\kappa(\tau)\|_{\bm{L}^2(\Gamma)}^2+\|\nabla_\Gamma \eta_\kappa(\tau)\|_{\bm{L}^2(\Gamma)}^2+ \|\Delta_\Gamma \varphi_\kappa(\tau)\|_{L^2(\Gamma)}^2\Big)\,\mathrm{d}\tau \non\\
&\quad \leq C(T),\qquad \forall\, t\in (0,T].
\label{esL5a}
\end{align}

Next, we recall a well-known result for the approximating potential $F_{\kappa}$ (see e.g., \cite{KNP} or \cite[Section 3]{FG12}), that is, there exists a positive constant $C$ that depends on $\langle \vp_\kappa(0)\rangle_\Gamma=\langle \vp_0\rangle_\Gamma\in (-1,1)$, $F$ and $|\Gamma|$, but is independent of $0<\kappa\ll 1$, such that
\begin{equation}
\label{eq:uniformhatS}
\| F'_{\kappa} (\varphi_\kappa)\|_{L^1(\Gamma)} \leq C \int_\Gamma (\varphi_\kappa - \langle\varphi_\kappa\rangle_\Gamma)\big(F'_{\kappa}(\varphi_\kappa) - \langle F'_{\kappa}(\varphi_\kappa)\rangle_\Gamma\big)\,\mathrm{d}x + C.
\end{equation}
Recalling the mass conservation $\langle\varphi_\kappa(t)\rangle_\Gamma=\langle \vp_0\rangle_\Gamma$, we test \eqref{rraw1.d} by $\varphi_\kappa - \langle\varphi_\kappa\rangle_\Gamma=\varphi_\kappa-\langle \vp_0\rangle_\Gamma $ and obtain 
\begin{align*}
  & \| \nabla_\Gamma \varphi_\kappa \|_{\bm{L}^2(\Gamma)}^2 + \int_\Gamma \big(F'_{\kappa}(\varphi_\kappa) - \langle F'_{\kappa}(\varphi_\kappa)\rangle_\Gamma\big) (\varphi_\kappa - \langle\varphi_0\rangle_\Gamma)\,\mathrm{d}S\\
  &\quad = \int_\Gamma(\mu_\kappa - \langle\mu_\kappa \rangle_\Gamma)(\varphi_\kappa - \langle\varphi_0\rangle_\Gamma)\,\mathrm{d}S
   + \theta_0\int_\Gamma \varphi_\kappa(\varphi_\kappa - \langle\varphi_0\rangle_\Gamma)\,\mathrm{d}S
   +\frac{1}{2}\int_\Gamma \eta_\kappa (\varphi_\kappa - \langle\varphi_0\rangle_\Gamma)\,\mathrm{d}S.
\end{align*}
Then it follows from \eqref{eq:uniformhatS}, Poincar\'{e}'s inequality and the Cauchy-Schwarz inequality that
\begin{align}
		 \| F'_{\kappa} (\varphi_\kappa)\|_{L^1(\Gamma)}
& \leq C  \| \nabla_\Gamma \mu_\kappa\|_{\bm{L}^2(\Gamma)} \|\nabla_\Gamma \varphi_\kappa \|_{\bm{L}^2(\Gamma)} + C \|\varphi_\kappa \|_{L^2(\Gamma)}^2+ C \|\eta_\kappa \|_{L^2(\Gamma)}^2\non\\
&\quad +C \|\varphi_\kappa \|_{L^1(\Gamma)}+ C \|\eta_\kappa \|_{L^1(\Gamma)},
\label{SSSe}
\end{align}
which implies
\begin{align}
\left|\int_\Gamma \mu_\kappa(t)\,\mathrm{d}S\right| 
&\leq   \left|\int_\Gamma W_\kappa'(\varphi_\kappa(t))\,\mathrm{d}S\right| 
+ \frac{1}{2}\left| \int_\Gamma\eta_\kappa(t)\,\mathrm{d}S\right| 
\non\\
&\leq C \| \nabla_\Gamma \mu_\kappa(t)\|_{\bm{L}^2(\Gamma)} +C,\qquad \text{for almost all}\ t\in (0,T).
\label{SSSe1}
\end{align}
From \eqref{esL5a}, \eqref{SSSe1}  and Poincar\'e's inequality, we find
\begin{align}
\int_0^t \|\mu_\kappa(\tau)\|_{H^1(\Gamma)}^2\,\mathrm{d}\tau
& \leq C\int_0^t \|\nabla_\Gamma \mu_\kappa(\tau)\|_{\bm{L}^2(\Gamma)}^2\,\mathrm{d}\tau +C \int_0^t \left|\int_\Gamma \mu_\kappa(\tau)\,\mathrm{d}S\right|^2\,\mathrm{d}\tau\non\\
&\leq  C(T),\qquad \forall\, t\in (0,T].
\label{esmuh1a}
\end{align}
Similar to \eqref{estu}--\eqref{estvv}, we also get
\begin{align}
& \int_0^t \|\partial_t u_\kappa(\tau)\|_{(H^1(\Omega))'}^2\,\mathrm{d}\tau \leq (1+D)C(T),\qquad \forall\, t\in (0,T],
\label{estvva}\\
&  \int_0^t \|\partial_t \varphi_\kappa(\tau)\|_{(H^1(\Gamma))'}^2\,\mathrm{d}\tau + \int_0^t \|\partial_t v_\kappa(\tau)\|_{(H^1(\Gamma))'}^2\,\mathrm{d}\tau\leq  C(T),\qquad \forall\, t\in (0,T].
\non
\end{align}
Besides, testing \eqref{rraw1.d} by  $-\Delta_\Gamma \varphi_\kappa$, we obtain
\begin{align}
\|\Delta_\Gamma \varphi_\kappa\|_{L^2(\Gamma)}^2
& = -\int_\Gamma W_\kappa''(\varphi_\kappa)|\nabla_\Gamma \varphi_\kappa|^2\,\mathrm{d}S + \int_\Gamma \nabla_\Gamma \mu_\kappa\cdot\nabla_\Gamma \varphi_\kappa\,\mathrm{d}S
\non\\
&\quad + \frac12\int_\Gamma \nabla_\Gamma \eta_\kappa\cdot\nabla_\Gamma \varphi_\kappa\,\mathrm{d}S\non\\
&\leq C\|\nabla_\Gamma \varphi_\kappa\|_{\bm{L}^2(\Gamma)}^2 + C
\|\nabla_\Gamma \varphi_\kappa\|_{\bm{L}^2(\Gamma)}\big(\|\nabla_\Gamma \mu_\kappa\|_{\bm{L}^2(\Gamma)}+\|\nabla_\Gamma \eta_\kappa\|_{\bm{L}^2(\Gamma)}\big),\label{vp-H2}
\end{align}
which combined with \eqref{esL5a} and the surface elliptic theory yields
\begin{align}
\int_0^t \|\varphi_\kappa(\tau)\|_{H^2(\Gamma)}^4\,\mathrm{d}\tau\leq C(T),\qquad \forall\, t\in (0,T].
\label{vpL4H2}
\end{align}

We note that all the estimates obtained above are independent of $\kappa$. Then by a standard compactness argument like in \cite{AK20}, we can find a convergent subsequence $\{(u_\kappa,\varphi_\kappa,v_\kappa,\mu_\kappa,\eta_\kappa)\}$ (not relabelled for simplicity), whose limit under $\kappa\to 0^+$ denoted by $(u,\varphi,v,\mu,\eta)$, fulfil the following regularity properties:
\begin{align}
& u \in L^\infty(0,T;L^2(\Omega))\cap L^2(0,T;H^1(\Omega))\cap H^1(0,T;(H^1(\Omega))'),\notag\\
& \varphi  \in L^{\infty}(0,T;H^1(\Gamma))\cap L^{4}(0,T;H^2(\Gamma))\cap H^{1}(0,T;(H^1(\Gamma))'),\notag \\
&\mu  \in L^{2}(0,T;H^1(\Gamma)),\notag \\
& v,\,\eta  \in L^{\infty}(0,T;L^2(\Gamma))\cap L^{2}(0,T;H^1(\Gamma)) \cap H^1(0,T;(H^1(\Gamma))').\notag
\end{align}
In view of \cite{AK20} (for the case of a regular potential $W$), in order to show that $(u,\varphi,v,\mu,\eta)$ is indeed a weak solution to the original system
\eqref{1.a}--\eqref{1.f} with a singular potential (or its weak formulation \eqref{w1.a}--\eqref{w1.f}),
the remaining task is to verify the limit of $F'_\kappa(\varphi_\kappa)$ in a suitable sense. This can be done by adapting the argument in \cite{DD95} (see also \cite{FG12,GGW18,Mi19} for further applications). To this end, we  infer from \eqref{esL5a} and \eqref{SSSe} that
\begin{align}
		\int_0^T \| F'_{\kappa} (\varphi_\kappa(t))\|_{L^1(\Gamma)}^2\,\mathrm{d}t\leq C,
\label{FF1}
\end{align}
where $C>0$ is independent of $\kappa$. By the definition of $F_{\kappa}$ and the assumption \textbf{(H1)}, there exists a
small constant $\kappa_0\in (0,r_0)$ such that for all
$\kappa \in(0,\kappa_0]$, $F'_\kappa(r)\geq 1$ for
$r\in[1-\kappa_0,+\infty)$ and $F'_\kappa(r)\leq -1$ for
$r\in(-\infty,-1+\kappa_0]$ and $F'_\kappa(r)$ is monotone
increasing for all $r\in\mathbb{R}$.
Let us introduce the following sets
\begin{align*}
&E^{\kappa}_{\kappa_0}=\left\lbrace (x,t)\in \Gamma \times (0,T):
\ |\vp_\kappa(x,t)|>1-\kappa_0 \right\rbrace, \quad \kappa\in(0,\kappa_0], \\
&E_{\kappa_0}=\left\lbrace (x,t)\in \Gamma \times (0,T):
\ |\vp(x,t)|>1-\kappa_0 \right\rbrace.
\end{align*}
Since $\varphi_\kappa\in L^{2}(0,T;H^2(\Gamma))\cap H^{1}(0,T;(H^1(\Gamma))')$ are uniformly bounded with respective to $\kappa$, the Aubin-Lions lemma (see e.g., \cite{simon}) yields
$\varphi_\kappa \to \varphi$ strongly in $L^2(0,T;H^1(\Gamma))$ (up to a subsequence). Thus, there exists a further subsequence such that $\vp_\kappa$ converges to $\varphi$ almost everywhere on $\Gamma\times(0,T)$.
From the pointwise convergence and Fatou's Lemma, we infer that for any fixed $\kappa_0\in (0,r_0)$, it holds
$$
|E_{\kappa_0}|
\leq \liminf_{\kappa \rightarrow 0^+}
| E^\kappa_{\kappa_0}|.
$$
On the other hand, for $\kappa\in (0,\kappa_0]$,
 we deduce from \eqref{FF1} that
$$
  \min\{ F'(1-\kappa_0), -F'(-1+\kappa_0)\}\,
  |E^\kappa_{\kappa_0}|
  \leq \|F'_\kappa(\vp_\kappa)
  \|_{L^1{(\Gamma \times (0,T)})}\leq C,
$$
where the constant $C>0$ does not depend on $\kappa_0$ and $\kappa$.
The above two estimates imply
$$
|E_{\kappa_0}|\leq \frac{C}
{\min\{ F'(1-\kappa_0), -F'(-1+\kappa_0)\} }.
$$
Passing to the limit as $\kappa_0\rightarrow 0^+$, we find
$$
\mathrm{meas}(\left\lbrace (x,t)\in \Gamma \times (0,T): |\vp(x,t)|\geq 1 \right\rbrace )=0,
$$
which implies
\begin{equation}
-1<\varphi(x,t)<-1,\qquad \text{for almost all}\ (x,t)\in \Gamma \times (0,T).
\label{Linf-1}
\end{equation}
The property \eqref{Linf-1} together with the
pointwise convergence of $\vp_\kappa$ and
the uniform convergence of $F'_\kappa(r)$ to $F'(r)$ on every
compact subset of $(-1,1)$ entails that
$$
F'_\kappa(\vp_\kappa)
\rightarrow F'(\vp), \qquad \text{for almost all } (x,t) \in \Gamma \times (0,T),
$$
as $\kappa\to 0^+$ (up to a subsequence). Besides, from the
estimates \eqref{esL5a}, \eqref{esmuh1a} and \eqref{vpL4H2}, we have
$$
\int_0^T \|F'_\kappa(\vp_\kappa)\|_{L^2(\Gamma)}^2\,\mathrm{d}t
\leq \int_0^T   \|\varphi_\kappa(t)\|_{H^2(\Gamma)}^2 \,\mathrm{d}t +
\int_0^T \big(\|\mu_\kappa(t)\|_{L^2(\Gamma)}^2+\|\eta_\kappa(t)\|_{L^2(\Gamma)}^2\big) \,\mathrm{d}t
\leq C.
$$
Then up to a subsequence, we find that as $\kappa\to 0^+$,
\be
\label{convF}
F'_\kappa(\vp_\kappa)
\rightarrow F'(\vp),\qquad  \text{ weakly in } L^2(0,T;L^2(\Gamma)).
\ee
Hence, the quintuplet $(u,\varphi,v,\mu,\eta)$ satisfies the weak formulation \eqref{w1.a}--\eqref{w1.f}. Furthermore, it is easy to verify that the initial conditions \eqref{ini1}--\eqref{ini2} are satisfied.

\subsubsection{Mass conservation and energy identity}\label{mei}
It is straightforward to check the mass conservation properties
\begin{align}
&\int_\Omega u(x,t)\,\mathrm{d}x+ \int_\Gamma v(x,t)\,\mathrm{d}S= \int_\Omega u_0(x)\,\mathrm{d}x+ \int_\Gamma v_0(x)\,\mathrm{d}S,\qquad \forall\, t\in[0,T],
\label{omass1}\\
&\int_\Gamma \varphi(x,t)\,\mathrm{d}S=\int_\Gamma \varphi_0(x)\,\mathrm{d}S,\qquad \forall\, t\in[0,T].
\label{omass2}
\end{align}
Next, following the arguments in \cite{FG12,GGW18}, which were based on the abstract results \cite[Lemma 7.3]{RoT05} and \cite[Proposition 4.2]{CKRS}, we can show that the total free energy $\mathcal{E}(u(t),\varphi(t),v(t))$ is absolutely continuous on $[0,T]$ and satisfies the energy identity
\begin{align}
& \mathcal{E}(u(t),\varphi(t),v(t))
+ \int_0^t\left[D\int_\Omega |\nabla u|^2\,\mathrm{d}x +\int_\Gamma \big(|\nabla_\Gamma \mu|^2+|\nabla_\Gamma \eta|^2\big)\,\mathrm{d}S\right]\mathrm{d}\tau \nonumber\\
&\quad  = \mathcal{E}(u_0,\varphi_0,v_0)+ \int_0^t\!\int_\Gamma q (\eta-u)\,\mathrm{d}S\mathrm{d}\tau,\qquad \forall\,t\in (0,T].
\label{BEL6}
\end{align}
Thanks to the linear growth condition on $q$, we find $q (\eta-u)\in L^1(0,T;L^1(\Gamma))$. Moreover, from \eqref{BEL6}, some further regularity properties for the weak solution can be established. By interpolation, we find that
$$
u\in C([0,T];L^2(\Omega))\quad \text{and}\quad v\in C([0,T];L^2(\Gamma)).
$$
 Concerning $\vp$, it easily follows $\varphi\in C_w([0,T];H^1(\Gamma))$  and $\varphi\in C([0,T];L^2(\Gamma))$.
From the continuity of the energy $\mathcal{F}(\varphi(t),v(t))$ and the convexity of $F$, we also find that $\|\nabla_\Gamma\vp(t)\|_{\bm{L}^2(\Gamma)}$ is continuous with respect to time on $[0,T]$. As a consequence, we obtain
$$
\varphi\in C([0,T];H^1(\Gamma)).
$$
On the other hand, since $\int_\Gamma F(\varphi(t))\,\mathrm{d}S$ is bounded for all $t\in [0,T]$, then from the assumption \textbf{(H1)} we can conclude the $L^\infty$-estimate \eqref{ues1}.

Finally, let us consider the following elliptic equation on $\Gamma\times (0,T)$:
\begin{align}
-\Delta_\Gamma \varphi+F'(\varphi)=\mu+\theta_0\varphi+\frac12\eta.
\label{ellp1}
\end{align}
We can apply Lemma \ref{lem:sep} and \eqref{esL5a}, \eqref{esmuh1a} to conclude that for any $2\leq p<+\infty$, it holds
\begin{align}
&\int_0^T \big(\|\varphi(t)\|_{W^{2,p}(\Gamma)}^2+ \|F'(\varphi(t))\|_{L^p(\Gamma)}^2\big)\,\mathrm{d}t\non\\
&\quad \leq CT + C\int_0^T \big(\|\mu(t)\|^2_{H^1(\Gamma)} +\|\varphi(t)\|^2_{H^1(\Gamma)}+\|\eta(t)\|^2_{H^1(\Gamma)}\big)\,\mathrm{d}t\non\\
&\quad\leq  C(T). \label{VL2Wq}
\end{align}
We have thus finished the proof for Theorem \ref{thm:weak}-(1).\
\hfill $\square$

\subsubsection{Uniqueness}\label{sec:uniq}
We proceed to establish the uniqueness of weak solutions provided that the mass exchange term $q$ is globally Lipschitz continuous.

Let $(u_i,\varphi_i,v_i,\mu_i,\eta_i)$, $i=1,2$, be two global weak solutions to problem \eqref{1.a}--\eqref{ini2} on $[0, T]$ with initial data $(u_{0i},\varphi_{0i},v_{0i})$, $i=1,2$, respectively. For simplicity, we assume that $\langle \vp_{01}\rangle_\Gamma=\langle \vp_{02}\rangle_\Gamma$. Consider the difference of two solutions:
$$
u=u_1-u_2,\quad \varphi=\varphi_1-\varphi_2,\quad v=v_1-v_2,
$$
and set $q_1=q(u_1,\varphi_1,v_1)$, $q_2=q(u_2,\varphi_2,v_2)$.
Since $q$ is assumed to be globally continuous, then we have the following estimate
\begin{align}
\|q_1-q_2\|_{L^2(\Gamma)} &\leq C\big(\|u\|_{L^2(\Gamma)}+ \|\vp\|_{L^2(\Gamma)}+ \|v\|_{L^2(\Gamma)}\big),\non
\end{align}
where $C>0$ is independent of $u$, $\vp$ and $v$. Thanks to \eqref{omass2}, we find $\langle\varphi(t)\rangle_\Gamma=0$ on $[0,T]$.
Then taking $u$, $(-\Delta_\Gamma)^{-1} \varphi$ and $\delta v$ as test functions for the equations of the difference $u$, $\vp$ and $v$, respectively, we arrive at
\begin{align}
&\frac12\frac{\mathrm{d}}{\mathrm{d}t}\left(\|u\|_{L^2(\Omega)}^2 + \|\varphi\|_{(H^1(\Gamma))'}^2+ \delta\|v\|_{L^2(\Gamma)}^2\right)
+ D\|\nabla u\|_{\bm{L}^2(\Omega)}^2 + \|\nabla_\Gamma \vp\|_{\bm{L}^2(\Gamma)}^2 + 4 \|\nabla_\Gamma v\|_{\bm{L}^2(\Gamma)}^2
\non\\
&\quad = -\int_\Gamma (q_1-q_2)u\,\mathrm{d}S - \int_\Gamma \left(\int_0^1W''(s\vp_1+(1-s)\vp_2)\vp\,\mathrm{d}s\right)\vp\,\mathrm{d}S \non\\
&\qquad+\frac{1}{\delta}\int_\Gamma (2v-\vp)\vp\,\mathrm{d}S +2\int_\Gamma \nabla_\Gamma\varphi\cdot\nabla_\Gamma v\,\mathrm{d}S + \delta \int_\Gamma (q_1-q_2)v\,\mathrm{d}S\non\\
&\quad \leq (1+\delta)\|q_1-q_2\|_{L^2(\Gamma)}(\|u\|_{L^2(\Gamma)}+ \|v\|_{L^2(\Gamma)}\big)+ (\theta_0-\theta)\|\vp\|_{L^2(\Gamma)}^2\non\\
&\qquad  +\frac{2}{\delta}\|v\|_{L^2(\Gamma)}\|\vp\|_{L^2(\Gamma)}
+ 2\|\nabla_\Gamma\varphi\|_{\bm{L}^2(\Gamma)}\|\nabla_\Gamma v\|_{\bm{L}^2(\Gamma)}\non\\
&\quad \leq C\big(\|u\|_{L^2(\Gamma)}^2
+\|\vp\|_{L^2(\Gamma)}^2
+ \|v\|_{L^2(\Gamma)}^2\big)
+ \frac12 \|\nabla_\Gamma\varphi\|_{\bm{L}^2(\Gamma)}^2
+2\|\nabla_\Gamma v\|_{\bm{L}^2(\Gamma)}^2,
\label{uniq1}
\end{align}
where in the above inequality, we have used H\"{o}lder's inequality, the Cauchy-Schwarz inequality and the assumptions \textbf{(H1)}, \textbf{(H3)}. Thanks to the interpolation inequality \eqref{inter1} and Young's inequality, we obtain
$$
\|u\|_{L^2(\Gamma)}^2\leq \gamma \|\nabla u\|_{\bm{L}^2(\Omega)}^2+ C(\gamma)\|u\|_{L^2(\Omega)}^2,
$$
for any $\gamma>0$. Besides, by interpolation and Poincar\'e's inequality, we have
\begin{align*}
\|\vp\|_{L^2(\Gamma)}^2
&\leq \|\vp\|_{H^1(\Gamma)}\|\vp\|_{(H^1(\Gamma))'}\non\\
&\leq \gamma \|\nabla_\Gamma\varphi\|_{\bm{L}^2(\Gamma)}^2
+ C(\gamma) \|\vp\|_{(H^1(\Gamma))'}^2.
\end{align*}
Choosing $\gamma>0$ sufficiently small, we can deduce from \eqref{uniq1} that
\begin{align}
&\frac12\frac{\mathrm{d}}{\mathrm{d}t}\left(\|u\|_{L^2(\Omega)}^2 + \|\varphi\|_{(H^1(\Gamma))'}^2+ \delta\|v\|_{L^2(\Gamma)}^2\right)
+ \frac{D}{2}\|\nabla u\|_{\bm{L}^2(\Omega)}^2 + \frac14\|\nabla_\Gamma \vp\|_{\bm{L}^2(\Gamma)}^2 + 2 \|\nabla_\Gamma v\|_{\bm{L}^2(\Gamma)}^2
\non\\
&\quad \leq C\big(\|u\|_{L^2(\Omega)}^2
+\|\vp\|_{(H^1(\Gamma))'}^2
+ \delta\|v\|_{L^2(\Gamma)}^2\big).
\label{uniq2}
\end{align}
Then by Gronwall's lemma, we can obtain the following continuous dependence estimate:
\begin{align}
&\|u(t)\|_{L^2(\Omega)}^2 + \|\varphi(t)\|_{(H^1(\Gamma))'}^2+ \delta\|v(t)\|_{L^2(\Gamma)}^2 \non\\
&\qquad + \int_0^t\big( \|\nabla u(\tau)\|_{\bm{L}^2(\Omega)}^2 +  \|\nabla_\Gamma \vp(\tau)\|_{\bm{L}^2(\Gamma)}^2 +  \|\nabla_\Gamma v(\tau)\|_{\bm{L}^2(\Gamma)}^2\big)\,\mathrm{d}\tau\non\\
& \quad \leq C(T)\left(\|u(0)\|_{L^2(\Omega)}^2 + \|\varphi(0)\|_{(H^1(\Gamma))'}^2+ \delta\|v(0)\|_{L^2(\Gamma)}^2\right),\qquad \forall\,t\in[0,T],
\label{conti1}
\end{align}
which easily yields the uniqueness result.

This completes the proof for Theorem \ref{thm:weak}-(2).

\subsubsection{Refined results in the equilibrium case}
When $q$ takes the specific form \eqref{q1}, it is linear and thus globally Lipschitz continuous. Hence, the solution is unique. Next, from the energy identity \eqref{BEL6} we immediately arrive at the conclusion \eqref{BELweak}, which holds for all $t\geq 0$. Thanks to the assumption \textbf{(H1)}, the dissipative energy law \eqref{BELweak} easily yields the uniform-in-time estimate \eqref{ues2} for the global weak solution $(u,\varphi,\mu,v,\eta)$ to problem \eqref{1.a}--\eqref{ini2}.

Thus, we finish the proof for Theorem \ref{thm:weak}-(3).

\subsection{Regularity}
In what follows, we prove Theorem \ref{thm:reg} on the regularity of global weak solutions.
The proof is based on the following auxiliary result about the existence of solutions in a more regular class.

\begin{proposition}
\label{pro:str}
Let $T>0$. Suppose that the assumptions \textbf{(H1)}--\textbf{(H3)} are satisfied, the general mass exchange term $q:\,\mathbb{R}^3\to \mathbb{R}$ is continuous and fulfils the linear growth condition \eqref{ligro}.
\begin{itemize}
\item[(1)] In addition to the assumptions in Theorem \ref{thm:weak}-(1), we assume that
\begin{align}
\vp_0\in H^2(\Gamma),\quad \widetilde{\mu}_0:= -\Delta_\Gamma \vp_0+F'(\vp_0)\in H^1(\Gamma), \quad v_0\in H^1(\Gamma).
\label{inireg}
\end{align}
Then the global weak solution to problem \eqref{1.a}--\eqref{ini2} fulfils further regularity properties
\begin{align}
& \varphi \in L^{\infty}(0,T;H^3(\Gamma))\cap L^2(0,T;H^4(\Gamma))\cap H^{1}(0,T; H^1(\Gamma)),\notag \\
&\mu \in L^\infty(0,T;H^1(\Gamma))\cap  L^2(0,T;H^3(\Gamma))\cap H^{1}(0,T; (H^1(\Gamma))'),\notag \\
& v \in L^{\infty}(0,T;H^1(\Gamma))\cap L^{2}(0,T;H^2(\Gamma)) \cap H^1(0,T;L^2(\Gamma)).\notag
\end{align}
Moreover, there exists a constant $\sigma_T\in (0, 1)$ such that the phase function $\vp$ satisfies
\begin{align}
\|\vp(t)\|_{C(\Gamma)}\leq 1-\sigma_T,\qquad \forall\,t\in [0,T],\label{vp-sep1}
\end{align}
where $\sigma_T$ may depend  $\|u_0\|_{L^2(\Omega)}$, $\|\varphi_0\|_{H^1(\Gamma)}$, $\|\widetilde{\mu}_0\|_{H^1(\Gamma)}$, $\|v_0\|_{H^1(\Gamma)}$, $\langle\varphi_0\rangle_\Gamma$, $\int_\Gamma F(\varphi_0)\,\mathrm{d}S$, $\Omega$, $\Gamma$, coefficients of the system and $T$. If the mass exchange term $q$ takes the specific form of \eqref{q2m} under the assumption \textbf{(H4)} and $u_0\in H^1(\Omega)$, then we have
$$
u\in L^\infty(0,T; H^1(\Omega))\cap L^2(0,T;H^2(\Omega))\cap H^1(0,T;L^2(\Omega)).
$$

\item[(2)] In the equilibrium case, i.e., $q$ takes the specific form of \eqref{q1} under the assumption \textbf{(H5)}, if the additional assumption \eqref{inireg} is satisfied, then the unique solution satisfies
\begin{align}
& \|\varphi\|_{ L^{\infty}(0,+\infty;H^3(\Gamma))}
+ \|\mu\|_{ L^\infty(0,+\infty;H^1(\Gamma))} +\|v\|_{L^{\infty}(0,+\infty;H^1(\Gamma))}\leq C,
\label{unit-es2}
\end{align}
for some $C>0$ and there exists a constant $\sigma\in (0,1)$ such that
\begin{align}
\|\vp(t)\|_{C(\Gamma)}\leq 1-\sigma,\qquad \forall\,t\in [0,+\infty).
\label{vp-sep2}
\end{align}
The constants $C$ and $\sigma$ may depend on $\|u_0\|_{L^2(\Omega)}$, $\|\varphi_0\|_{H^1(\Gamma)}$, $\|\widetilde{\mu}_0\|_{H^1(\Gamma)}$, $\|v_0\|_{H^1(\Gamma)}$, $\langle\varphi_0\rangle_\Gamma$, $\int_\Gamma F(\varphi_0)\,\mathrm{d}S$, $\Omega$, $\Gamma$, coefficients of the system, but are independent of the time $t$. If we assume $u_0\in H^1(\Omega)$, then it holds
\begin{align}
\|u\|_{L^\infty(0,+\infty; H^1(\Omega))}\leq C, \label{unit-es3}
\end{align}
where $C>0$ also depends on $\|u_0\|_{H^1(\Omega)}$.
\end{itemize}
\end{proposition}

\begin{remark}\label{rem:sep0}
Similar to \cite[Remark 3.1]{H1}, for the initial datum $\vp_0$ that satisfies $\vp_0\in H^2(\Gamma)$, $\widetilde{\mu}_0\in H^1(\Gamma)$ with
$\|\vp_0\|_{L^\infty(\Gamma)}\leq 1$ and $|\langle \vp_0\rangle_\Gamma|<1$, we can infer from Lemma \ref{lem:sep} that there exists $\sigma_0\in (0, 1)$ such that
$$
\|\vp_0\|_{C(\Gamma)}\leq 1-\sigma_0,
$$
namely, the initial phase function $\vp_0$ is indeed strictly separated from the pure states $\pm 1$. Then by the surface elliptic theory (see e.g., \cite{Aubin82}) and the assumption \textbf{(H1)}, we can further deduce that $\vp_0\in H^3(\Gamma)$. Thanks to this  observation, we do not need the cutoff procedure for the initial datum $\vp_0$ like that in \cite[Section 4]{GMT}.
\end{remark}

\subsubsection{Uniform estimates}

The proof of Proposition \ref{pro:str} essentially relies on some higher-order energy estimates in suitable Sobolev spaces for solutions to the regularized system \eqref{a1.a}--\eqref{aini2}.

In what follows, we consider the system \eqref{a1.a}--\eqref{aini2} with $\kappa$ satisfying
$$
0<\kappa< \overline{\kappa} :=\min\left\{\frac14\sigma_0,\,r_0\right\},
$$
where $\sigma_0\in (0,1)$ is given as in Remark \ref{rem:sep0}. \medskip

\textbf{Estimates for the initial data}. When $\kappa$ is chosen sufficiently small as above, it holds
$$
-\Delta_\Gamma \vp_0+ F'_\kappa(\vp_0)
= -\Delta_\Gamma \vp_0+ F'(\vp_0)=\widetilde{\mu}_0.
$$
Since $\varphi_0\in H^3(\Gamma)$ (see Remark \ref{rem:sep0}) and $v_0\in H^1(\Gamma)$, in the corresponding Galerkin scheme \eqref{rw1.a}--\eqref{rwini2}, we have the strong convergence of initial data
$$
\|\varphi_\kappa^n(0)-\varphi_0\|_{H^3(\Gamma)}\to 0,\quad
\|v_\kappa^n(0)-v_0\|_{H^1(\Gamma)}\to 0,\quad \text{as}\ n\to +\infty.
$$
By the Sobolev embedding theorem $H^2(\Gamma)\hookrightarrow C(\Gamma)$, we find
\begin{align}
\|\varphi_\kappa^n(0)\|_{C(\Gamma)}\leq 1-\frac12\sigma_0,\label{inivk1}
\end{align}
for $n$ sufficiently large (the situation that we are interested in). Thus, in light of \eqref{appini} and \eqref{fkk}, we can obtain the uniform upper bound
\begin{align}
\mathcal{E}_\kappa(u^n_\kappa(0),\varphi^n_\kappa(0),v^n_\kappa(0)) \leq C,
\label{appEa}
\end{align}
where $C>0$ depends on $\|u_0\|_{L^2(\Omega)}$, $\|\varphi_0\|_{H^1(\Gamma)}$, $\|v_0\|_{L^2(\Gamma)}$, $\int_\Gamma F(\vp_0)\,\mathrm{d}S$, $\sigma_0$, $|\Omega|$, $|\Gamma|$ and the structural coefficients, but neither on $n$ nor on $\kappa$.

Besides, thanks to the assumption \textbf{(H1)} and \eqref{inivk1}, we can infer that (cf. \cite[(4.39)]{GMT})
\begin{align}
 \|\widetilde{\mu}_\kappa^n(0)\|_{H^1(\Gamma)}
 &:= \|\mathbb{P}_{V_{n,\Gamma}}( -\Delta_\Gamma \mathbb{P}_{V_{n,\Gamma}} \varphi_0 + F'_\kappa(\mathbb{P}_{V_{n,\Gamma}} \varphi_0))\|_{H^1(\Gamma)}\non\\
&  \leq \| -\Delta_\Gamma \mathbb{P}_{V_{n,\Gamma}} \varphi_0+ F'_\kappa(\mathbb{P}_{V_{n,\Gamma}} \varphi_0)\|_{H^1(\Gamma)}\non\\
&  \leq \| -\Delta_\Gamma (\mathbb{P}_{V_{n,\Gamma}} \varphi_0 -\vp_0)\|_{H^1(\Gamma)}
+ \|F'(\mathbb{P}_{V_{n,\Gamma}} \varphi_0)-F'(\varphi_0)\|_{H^1(\Gamma)} + \|\widetilde{\mu}_0\|_{H^1(\Gamma)}\non\\
&  \leq  \|\widetilde{\mu}_0\|_{H^1(\Gamma)} +\| \mathbb{P}_{V_{n,\Gamma}} \varphi_0 -\vp_0\|_{H^3(\Gamma)} \non\\
&\quad + C \max_{|r|\leq 1-\overline{\kappa}}\big(|F''(r)|+ |F'''(r)|\big)\| \mathbb{P}_{V_{n,\Gamma}} \varphi_0 -\vp_0\|_{H^1(\Gamma)}\non \\
&  \leq \|\widetilde{\mu}_0\|_{H^1(\Gamma)} + 1,
\qquad \text{for}\ n\gg 1.
\label{es-inimu1}
\end{align}
As a result, we have
\begin{align}
\|\mu_\kappa^n(0)\|_{H^1(\Gamma)} &= \left\|\widetilde{\mu}_\kappa^n(0)-\theta_0  \varphi_\kappa^n(0) +\frac12   \eta_\kappa^n(0)\right\|_{H^1(\Gamma)}\non\\
&\leq  \|\widetilde{\mu}_\kappa^n(0)\|_{H^1(\Gamma)} +C \big(\|\mathbb{P}_{V_{n,\Gamma}} \varphi_0\|_{H^1(\Gamma)} +
\|\mathbb{P}_{V_{n,\Gamma}} v_0\|_{H^1(\Gamma)}\big)\non\\
&\leq \|\widetilde{\mu}_0\|_{H^1(\Gamma)} + C\big( \|\varphi_0\|_{H^1(\Gamma)} +  \|v_0\|_{H^1(\Gamma)} \big)+C,
\label{es-inimu2}
\end{align}
for sufficiently large $n$, with $C>0$ being independent of $n$ and $\kappa$.\medskip

Now we can proceed to derive estimates for solutions $(u^n_\kappa,\varphi^n_\kappa,v^n_\kappa,\mu^n_\kappa,\eta^n_\kappa)$ to the Galerkin scheme \eqref{rw1.a}--\eqref{rwini2} (recall Proposition \ref{exe:gar} for its existence) that are uniform with respect to both approximating parameters $n$ and $\kappa$. \medskip

\textbf{Lower-order estimates}.
First, we still have the mass conservation properties \eqref{rwmass1} and \eqref{rwmass2}.
Next, in the energy inequality \eqref{esL4}, using now \eqref{appEa}, we find, for $n\gg 1$,
$$\widetilde{\mathcal{E}}_\kappa(u^n_\kappa(0), \varphi^n_\kappa(0),v^n_\kappa(0))\leq C,$$
with $C>0$ being independent of $n$ and $\kappa$.
This allows us to obtain
\begin{align}
&  \|u_\kappa^n(t)\|_{L^2(\Omega)}^2
+ \|\varphi_\kappa^n(t)\|_{H^1(\Gamma)}^2
+ \|v_\kappa^n(t)\|_{L^2(\Gamma)}^2
+ \|\eta_\kappa^n(t)\|_{L^2(\Gamma)}^2 \non\\
&\qquad  +\int_0^t \Big(\|\nabla u^n_\kappa(\tau)\|_{\bm{L}^2(\Omega)}^2
+\|\nabla_\Gamma \mu^n_\kappa(\tau)\|_{\bm{L}^2(\Gamma)}^2+\|\nabla_\Gamma \eta^n_\kappa(\tau)\|_{\bm{L}^2(\Gamma)}^2+ \|\Delta_\Gamma \varphi_\kappa^n(\tau)\|_{L^2(\Gamma)}^2\Big)\,\mathrm{d}\tau \non\\
&\quad \leq C(T),\qquad \forall\, t\in (0,T].
\label{esL5b}
\end{align}
Using \eqref{esL5b} and an argument similar to that for \eqref{esmuh1a}, we get the estimate \eqref{esmuh1} for $\mu_\kappa^n$, which becomes now independent of $n$ and $\kappa$. In turn, we can also recover the uniform estimates \eqref{estu}--\eqref{estvv} on the time derivatives $\partial_t u_\kappa^n$, $\partial_t \vp_\kappa^n$ and $\partial_t v_\kappa^n$.

In summary, we can define the approximate solutions  $(u^n_\kappa,\varphi^n_\kappa,v^n_\kappa,\mu^n_\kappa,\eta^n_\kappa)$ on the whole time interval $[0,T]$, with corresponding lower-order estimates that are updated to be independent of both approximating parameters $n$ and $\kappa$ provided that $n\gg 1$ and $\kappa\ll 1$. \medskip

\textbf{Higher-order estimates}.
The crucial step is to derive uniform higher order estimates.
To this end, taking $\zeta=\partial_t \mu_\kappa^n$ in \eqref{rw1.c} and using the uniform lower bound of $W''_\kappa$, we obtain
\begin{align}
&\frac12\frac{\mathrm{d}}{\mathrm{d}t}\|\nabla_\Gamma \mu_\kappa^n\|_{L^2(\Gamma)}^2
  = - \int_\Gamma \partial_t \vp_\kappa^n \partial_t \mu_\kappa^n\,\mathrm{d}S\non\\
&\quad  = -\|\nabla_\Gamma  \partial_t \vp_\kappa^n \|_{\bm{L}^2(\Gamma)}^2 -\int_\Gamma W_\kappa''(\vp_\kappa^n)|\partial_t \vp_\kappa^n|^2\,\mathrm{d}S
+ \frac{2}{\delta} \int_\Gamma \partial_t v_\kappa^n \partial_t \vp_\kappa^n \,\mathrm{d}S\non\\
&\qquad -\frac{1}{\delta}\| \partial_t \vp_\kappa^n\|_{L^2(\Gamma)}^2\non\\
&\quad \leq -\|\nabla_\Gamma  \partial_t \vp_\kappa^n \|_{\bm{L}^2(\Gamma)}^2 + C\| \partial_t \vp_\kappa^n\|_{L^2(\Gamma)}^2
+ C \|  \partial_t \vp_\kappa^n\|_{H^1(\Gamma)}\| \partial_t v_\kappa^n \|_{(H^1(\Gamma))'}\non\\
&\quad \leq -\|\nabla_\Gamma  \partial_t \vp_\kappa^n \|_{\bm{L}^2(\Gamma)}^2 + C \|  \partial_t \vp_\kappa^n\|_{H^1(\Gamma)}\big(\| \partial_t \vp_\kappa^n\|_{(H^1(\Gamma))'}+\| \partial_t v_\kappa^n \|_{(H^1(\Gamma))'}\big)\non\\
&\quad \leq - \frac12 \|\nabla_\Gamma  \partial_t \vp_\kappa^n \|_{\bm{L}^2(\Gamma)}^2 +C \big(\| \partial_t \vp_\kappa^n\|_{(H^1(\Gamma))'}^2+\| \partial_t v_\kappa^n \|_{(H^1(\Gamma))'}^2\big),
\label{esmuH1a}
\end{align}
where $C>0$ only depends on $\Gamma$, $\theta$, $\theta_0$ and $\delta$. Here, we have used the fact $\langle  \partial_t \vp_\kappa^n\rangle_\Gamma=0$, the interpolation $\| \partial_t \vp_\kappa^n\|_{L^2(\Gamma)}^2\leq C\| \partial_t \vp_\kappa^n\|_{H^1(\Gamma)}\| \partial_t \vp_\kappa^n\|_{(H^1(\Gamma))'}$, Poincar\'e's inequality and Young's inequality.

Next, taking $\zeta=-\Delta_\Gamma v_\kappa^n$ in \eqref{rw1.e}, we obtain
\begin{align}
&\frac12 \frac{\mathrm{d}}{\mathrm{d}t}\|\nabla_\Gamma v_\kappa^n\|_{\bm{L}^2(\Gamma)}^2 + \frac{4}{\delta}\|\Delta_\Gamma v_\kappa^n\|_{L^2(\Gamma)}^2\non\\
&\quad = \frac{2}{\delta}\int_\Gamma \Delta_\Gamma \vp_\kappa^n\Delta_\Gamma v_\kappa^n\,\mathrm{d}S -\int_\Gamma q_\kappa^n \Delta_\Gamma v_\kappa^n\,\mathrm{d}S\non\\
&\quad \leq \frac{2}{\delta}\|\Delta_\Gamma v_\kappa^n\|_{L^2(\Gamma)}^2 +C \|\Delta_\Gamma \vp_\kappa^n\|_{L^2(\Gamma)}^2 +C \|q_\kappa^n \|_{L^2(\Gamma)}^2\non\\
&\quad \leq \frac{2}{\delta}\|\Delta_\Gamma v_\kappa^n\|_{L^2(\Gamma)}^2 +C \| \vp_\kappa^n\|_{H^2(\Gamma)}^2 +  C \| u_\kappa^n\|_{L^2(\Gamma)}^2  +  C \| v_\kappa^n\|_{L^2(\Gamma)}^2 +C \non\\
& \quad \leq \frac{2}{\delta}\|\Delta_\Gamma v_\kappa^n\|_{L^2(\Gamma)}^2 + C \| \vp_\kappa^n\|_{H^2(\Gamma)}^2 +  C \| u_\kappa^n\|_{H^1(\Omega)}^2  +  C \| v_\kappa^n\|_{L^2(\Gamma)}^2 +C,
\label{esvH1a}
\end{align}
where $C>0$ only depends on $q$, $\Omega$, $\Gamma$ and $\delta$.

It follows from \eqref{esmuH1a} and \eqref{esvH1a} that
\begin{align}
& \frac{\mathrm{d}}{\mathrm{d}t}\Big(\|\nabla_\Gamma \mu_\kappa^n\|_{\bm{L}^2(\Gamma)}^2+ \|\nabla_\Gamma v_\kappa^n\|_{\bm{L}^2(\Gamma)}^2\Big)
+ \|\nabla_\Gamma  \partial_t \vp_\kappa^n \|_{\bm{L}^2(\Gamma)}^2 + \frac{4}{\delta}\|\Delta_\Gamma v_\kappa^n\|_{L^2(\Gamma)}^2\non\\
&\quad \leq C \big(\| \partial_t \vp_\kappa^n\|_{(H^1(\Gamma))'}^2+\| \partial_t v_\kappa^n \|_{(H^1(\Gamma))'}^2 + \| \vp_\kappa^n\|_{H^2(\Gamma)}^2 +   \| u_\kappa^n\|_{H^1(\Omega)}^2  +  \| v_\kappa^n\|_{L^2(\Gamma)}^2 +1\big).
\label{esmvH1b}
\end{align}
The right-hand side of \eqref{esmvH1b} belongs to $L^1(0,T)$ thanks to the uniform lower-order estimates in the previous step.
Hence, in view of \eqref{es-inimu2}, we can integrate \eqref{esmvH1b} with respect to time and obtain
\begin{align}
 & \|\nabla_\Gamma \mu_\kappa^n(t)\|_{\bm{L}^2(\Gamma)}^2+ \|\nabla_\Gamma v_\kappa^n(t)\|_{\bm{L}^2(\Gamma)}^2
+ \int_0^t \left(\|\nabla_\Gamma  \partial_t \vp_\kappa^n (\tau)\|_{\bm{L}^2(\Gamma)}^2 + \frac{4}{\delta}\|\Delta_\Gamma v_\kappa^n(\tau)\|_{L^2(\Gamma)}^2\right)\,\mathrm{d}\tau\non\\
&\quad \leq  C\big(\|\widetilde{\mu}_0\|_{H^1(\Gamma)}^2 + \|v_0\|_{H^1(\Gamma)}^2\big) + C(T),\qquad \forall\, t\in [0,T],
\label{esmvH1c}
\end{align}
where $C(T)>0$ may depend on $\|u_0\|_{L^2(\Omega)}$, $\|\varphi_0\|_{H^1(\Gamma)}$, $\|v_0\|_{L^2(\Gamma)}$, $\int_\Gamma F(\vp_0)\,\mathrm{d}S$, $\sigma_0$, $\Omega$, $\Gamma$, $T$ and the structural coefficients, but is independent of $n$ and $\kappa$. Similar to \eqref{SSSe1} and \eqref{esmuh1a}, we can deduce from \eqref{esmvH1c} that
\begin{align}
\sup_{t\in[0,T]}\|\mu_\kappa^n(t)\|_{H^1(\Gamma)}\leq C(T).\label{esmuH1b}
\end{align}
The above estimate combined with \eqref{esL5b} gives (cf. \eqref{vp-H2})
\begin{align}
\sup_{t\in[0,T]}\|\vp_\kappa^n(t)\|_{H^2(\Gamma)}\leq C(T).\label{vpLiH2}
\end{align}
Moreover, we have
\begin{align}
\int_0^t \|\partial_t v_\kappa^n(\tau)\|_{L^2(\Gamma)}^2\,\mathrm{d}\tau
&\leq  C\int_0^t \big(\|\Delta_\Gamma v _\kappa^n(\tau)\|_{L^2(\Gamma)}^2 + \|\Delta_\Gamma \vp_\kappa^n(\tau)\|_{L^2(\Gamma)}^2 + \|q_\kappa^n\|_{L^2(\Gamma)}^2\big)\,\mathrm{d}\tau \non\\
&\leq C\int_0^t \big(\|v _\kappa^n(\tau)\|_{H^2(\Gamma)}^2 + \|\vp_\kappa^n(\tau)\|_{H^2(\Gamma)}^2 + \|u_\kappa^n\|_{H^1(\Omega)}^2+1\big)\,\mathrm{d}\tau \non\\
& \leq C(T),\qquad \forall\, t\in [0,T].
\label{esvH1L2}
\end{align}

When $q$ takes the specific form \eqref{q2m} under the assumption \textbf{(H4)} and $u_0\in H^1(\Omega)$, we can say more about $u_\kappa^n$. Taking $\zeta = \partial_t u_\kappa^n $ in \eqref{rw1.a}, using integration by parts, we obtain
\begin{align}
\frac{\mathrm{d}}{\mathrm{d}t} \mathcal{Y}_\kappa^n(t)+ \|\partial_t u_\kappa^n\|_{L^2(\Omega)}^2
& = - B_1 \int_\Gamma \widetilde{H}(u_\kappa^n) \partial_t v_\kappa^n\,\mathrm{d}S - B_2 \int_\Gamma u_\kappa^n \partial_t v_\kappa^n \,\mathrm{d}S,
\label{dtu}
\end{align}
where
\begin{align}
 \mathcal{Y}_\kappa^n(t)
 &=\frac{D}{2}\|\nabla u_\kappa^n(t)\|_{\bm{L}^2(\Omega)}^2 + \frac{B_1}{2}\|u_\kappa^n(t)\|_{L^2(\Gamma)}^2 - B_1 \int_\Gamma \widetilde{H}(u_\kappa^n) v_\kappa^n\,\mathrm{d}S- B_2\int_\Gamma u_\kappa^nv_\kappa^n\,\mathrm{d}S,  \non
\end{align}
and
$$
\widetilde{H}(r)=\int_0^r \widetilde{h}(s)\,\mathrm{d}s.
$$
Thanks to the assumption \textbf{(H4)}, the function $\widetilde{H}$ has at most a linear growth on $\mathbb{R}$. Then using the interpolation inequality \eqref{inter1}, H\"{o}lder's and Young's inequalities, we find
$$
\mathcal{Y}_\kappa^n(t)\leq C_1\big(\|u_\kappa^n(t)\|_{H^1(\Omega)}^2 +  \|v_\kappa^n(t)\|_{L^2(\Gamma)}^2 +1\big),
$$
$$
\mathcal{Y}_\kappa^n(t)\geq C_2 \|u_\kappa^n(t)\|_{H^1(\Omega)}^2 -
C_3\big(\|u_\kappa^n(t)\|_{L^2(\Omega)}^2 +  \|v_\kappa^n(t)\|_{L^2(\Gamma)}^2 +1\big),
$$
where the positive constants $C_1,C_2,C_3$ depend on $D$, $B_1$, $B_2$, $\widetilde{h}$, $\Omega$ and $\Gamma$, but are independent of $n$ and $\kappa$.
In view of \eqref{esL5b}, there exists some $C_4>0$ such that
$$
\widetilde{\mathcal{Y}}_\kappa^n(t)=\mathcal{Y}_\kappa^n(t)+C_4\geq 0,\qquad \forall\, t\in [0,T].
$$
Moreover, it follows from \eqref{dtu} that
\begin{align}
 \frac{\mathrm{d}}{\mathrm{d}t} \widetilde{\mathcal{Y}}_\kappa^n(t)+ \|\partial_t u_\kappa^n\|_{L^2(\Omega)}^2
& \leq C\big(1+\|u_\kappa^n(t)\|_{L^2(\Gamma)}\big)\|\partial_t v_\kappa^n(t)\|_{L^2(\Gamma)}\non\\
&\leq C \|u_\kappa^n(t)\|_{H^1(\Omega)}^2 + C\big(\|\partial_t v_\kappa^n(t)\|_{L^2(\Gamma)}^2+1\big),
\label{dtu1}
\end{align}
where the two terms on the right-hand side belongs to $L^1(0,T)$ thanks to \eqref{esL5b} and \eqref{esvH1L2}. Hence, integrating \eqref{dtu1} in time, we obtain
\begin{align}
\|u_\kappa^n(t)\|_{H^1(\Omega)}^2 +\int_0^{t} \|\partial_t u_\kappa^n(\tau)\|_{L^2(\Omega)}^2\,\mathrm{d}\tau \leq C(T),\qquad \forall\,t\in (0,T],\label{es-unitime5a}
\end{align}
where $C>0$ depends on $\|u_0\|_{H^1(\Omega)}$, $\|\varphi_0\|_{H^2(\Gamma)}$, $\|v_0\|_{H^1(\Gamma)}$, $\|\widetilde{\mu}_0\|_{H^1(\Gamma)}$, $\langle\varphi_0\rangle_\Gamma$, $\int_\Gamma F(\varphi_0)\,\mathrm{d}S$, $\Omega$, $\Gamma$ and the structural coefficients, but not on $n$ and $\kappa$.

\subsubsection{Passage to the limit}
Combining the estimates obtained in the previous step, we find
\begin{align*}
&\vp_\kappa^n\ \ \text{is uniform bounded in}\ \ L^\infty(0,T;H^2(\Gamma))\cap H^1(0,T;H^1(\Gamma)),\\
&\mu_\kappa^n\ \ \text{is uniform bounded in}\ \ L^\infty(0,T;H^1(\Gamma)),\\
&v_\kappa^n\ \ \text{is uniform bounded in}\ \ L^\infty(0,T;H^1(\Gamma))\cap L^2(0,T;H^2(\Gamma))\cap H^1(0,T;L^2(\Gamma)).
\end{align*}
Thus, by the standard compactness argument mentioned before, we can first pass to the limit as $n\to +\infty$ and then as $\kappa\to 0^+$ to obtain the existence of a global weak solution $(u,\vp,v,\mu,\eta)$ on $[0,T]$ with further regularity properties indicated above.
By interpolation, we have $v\in C([0,T];H^1(\Gamma))$. Besides, from the surface elliptic theory, we find
\begin{align}
\int_0^T\|\mu(t)\|_{H^3(\Gamma)}^2\,\mathrm{d}t\leq C \int_0^T\big(\|\partial_t \vp(t)\|_{H^1(\Gamma)}^2 + \|\mu(t)\|_{H^1(\Gamma)}^2\big)\,\mathrm{d}t\leq C(T),
\end{align}
so that $\mu\in L^2(0,T;H^3(\Gamma))$. Here, we also remark that the solution is unique provided that $q$ is globally Lipschitz continuous on $\mathbb{R}^3$.

If $q$ takes the specific form \eqref{q2m} and $u_0\in H^1(\Omega)$, then we have
 $$
 u\in L^\infty(0,T;H^1(\Omega))\cap H^1(0,T;L^2(\Omega)).
 $$
 Thanks to the assumption \textbf{(H4)}, $q=q(u,v)\in C^1(\mathbb{R}^2)$ and fulfils the growth condition as in \cite[(2.17)]{AK20}. Thus, we can apply \cite[Lemma 3.1]{AK20} to
 conclude that $q\in L^2(0,T;H^{1/2}(\Gamma))$ (see \cite[Page 325]{AK20} for further details). This property combined with $\partial_t u \in L^2(0,T;L^2(\Omega))$ and the elliptic estimate for $u$ yields $u\in L^2(0,T;H^2(\Omega))$. By interpolation, we also obtain $u\in C([0,T];H^1(\Omega))$.
%

\subsubsection{Separation property and further regularity}
Since $\vp$ satisfies the surface elliptic equation \eqref{ellp1} and the nonlinearity $F$ now fulfils the assumptions \textbf{(H1)}, \textbf{(H2)}, we can apply Lemma \ref{lem:sep}, \eqref{esL5b}, \eqref{esmvH1c} and \eqref{esmuH1b} to conclude that
\begin{align}
& \varphi  \in L^\infty(0,T; W^{2,p}(\Gamma)), \quad F''(\vp) \in L^\infty(0,T; L^p(\Gamma)),\quad  F'(\vp) \in L^\infty(0,T; W^{1,p}(\Gamma)),
\label{LLP}
\end{align}
for any $2\leq p<+\infty$. Moreover, there exists $\sigma_T\in (0,\sigma_0]$ such that
\begin{align}
\|\vp(t)\|_{L^\infty(\Gamma)}\leq 1-\sigma_T\quad \text{for almost all}\ t\in [0,T].
\label{vp-sep3}
\end{align}
The strict separation property \eqref{vp-sep3} enables us to say more about the regularity of $\vp$. Since $F\in C^3([-1+\sigma_T,1-\sigma_T])$, then from \eqref{LLP} and \eqref{vp-sep3}, we find
$$
F'(\vp)\in L^\infty(0,T;H^2(\Gamma)).
$$
Using \eqref{LLP}, the facts
$$\mu\in L^\infty(0,T;H^1(\Gamma))\cap L^2(0,T;H^3(\Gamma)),\quad v\in L^\infty(0,T;H^1(\Gamma))\cap L^2(0,T;H^2(\Gamma)),
$$
and the surface elliptic theory for $\vp$, we further obtain
$$
\vp\in L^\infty(0,T;H^3(\Gamma))\cap L^2(0,T;H^4(\Gamma)).
$$
Concerning the time derivative of $\mu$, following the argument in \cite{GMT}, we introduce the difference quotient of a function $f$:
$$
\partial_t^h f(t)= \frac{1}{h}[f(t+h)-f(t)],\qquad \forall\,h\in (0,1).
$$
Then for any $\zeta\in H^1(\Gamma)$ with $\|\zeta\|_{H^1(\Gamma)}\leq 1$,
from the identity
\begin{align}
\int_\Gamma \partial_t^h \mu \zeta\,\mathrm{d}S = \int_\Gamma \nabla_\Gamma \partial_t^h\vp \cdot\nabla_\Gamma  \zeta\,\mathrm{d}S
+ \int_\Gamma \Big(\partial_t^h W'(\vp) -\frac{1}{2}\partial_t^h \eta\Big)  \zeta\,\mathrm{d}S
\non
\end{align}
we get
\begin{align*}
\left|\int_\Gamma \partial_t^h \mu \zeta\,\mathrm{d}S\right|
& \leq \left|\int_\Gamma \nabla_\Gamma \partial_t^h\vp \cdot\nabla_\Gamma  \zeta\,\mathrm{d}S\right| + \left|\int_\Gamma \partial_t^h W'(\vp) \zeta\,\mathrm{d}S\right| + \frac{1}{2}\left|\int_\Gamma \partial_t^h\eta \zeta \,\mathrm{d}S\right|\non\\
&  \leq C\big(\|\partial_t^h \vp\|_{H^1(\Gamma)}
+ \|\partial_t^h v\|_{L^2(\Gamma)}\big)\|\zeta\|_{H^1(\Gamma)}\\
&\quad  + C \int_0^1 \|W''(s\vp(t+h)+(1-s)\vp(t))\|_{L^\infty(\Gamma)}\,\mathrm{d}s \|\partial_t^h \vp\|_{L^2(\Gamma)}\|\zeta\|_{L^2(\Gamma)}\non\\
& \leq   C\big(\|\partial_t^h \vp\|_{H^1(\Gamma)}
+ \|\partial_t^h v\|_{L^2(\Gamma)}\big)\|\zeta\|_{H^1(\Gamma)} \\
&\qquad  + C \max_{r\in[-1+\sigma_T,1-\sigma_T]}|W''(r)|\|\partial_t^h \vp\|_{L^2(\Gamma)}\|\zeta\|_{L^2(\Gamma)}\non\\
& \leq C\big(\|\partial_t^h \vp\|_{H^1(\Gamma)}
+ \|\partial_t^h v\|_{L^2(\Gamma)}\big),
\end{align*}
where $C>0$ is independent of $h$. Since $\partial_t\vp\in L^2(0,T;H^1(\Gamma))$ and $\partial_t v\in L^2(0,T;L^2(\Gamma))$, we have
$$
\|\partial_t^h \vp(t)\|_{H^1(\Gamma)}
\leq \frac{1}{h}\int_t^{t+h}\| \partial_t  \vp(\tau)\|_{H^1(\Gamma)}\,\mathrm{d}\tau \to_{h\to 0} \| \partial_t  \vp(t)\|_{H^1(\Gamma)},
$$
$$
\|\partial_t^h v(t)\|_{L^2(\Gamma)}
\leq \frac{1}{h}\int_t^{t+h}\| \partial_t  v(\tau)\|_{L^2(\Gamma)}\,\mathrm{d}\tau \to_{h\to 0} \| \partial_t  v(t)\|_{L^2(\Gamma)},
$$
for almost all $t\in[0,T]$, and thus
$$
\|\partial_t^h \vp\|_{L^2(0,T-h;H^1(\Gamma))}
+ \|\partial_t^h v\|_{L^2(0,T-h;L^2(\Gamma))}
\leq \|\partial_t  \vp\|_{L^2(0,T ;H^1(\Gamma))}
+ \|\partial_t  v\|_{L^2(0,T;L^2(\Gamma))}.
$$
As a consequence, $\|\partial_t^h \mu\|_{L^2(0,T;(H^1(\Gamma))')}$ is uniformly bounded with respect to $h$. Taking $h\to 0$, we find that the distributional derivative $\partial_t\mu$ satisfies
$$
\partial_t\mu\in L^2(0,T;(H^1(\Gamma))').
$$
Thanks to the Aubin-Lions-Simon lemma, we can conclude
$$
\mu\in C([0,T];H^1(\Gamma)),\qquad \vp\in C([0,T]; H^2(\Gamma)).
$$
Hence, by the Sobolev embedding theorem $H^2(\Gamma)\hookrightarrow C(\Gamma)$ and \eqref{vp-sep3}, we arrive at the conclusion \eqref{vp-sep1}.

This completes the proof of Proposition \ref{pro:str}-(1).

\subsubsection{Refined results in the equilibrium case}
\label{sec:rec}
When $q$ takes the specific form \eqref{q1}, similar to \eqref{BELweak} and \eqref{ues2}, we have the following uniform estimate for the approximate solutions:
\begin{align}
 & \|u_\kappa^n\|_{L^\infty(0,+\infty;L^2(\Omega))}
 + \|\varphi_\kappa^n\|_{L^\infty(0,+\infty;H^1(\Gamma))}
 + \|v_\kappa^n\|_{L^\infty(0,+\infty;L^2(\Gamma))}\notag\\
 &\quad + \|\nabla u_\kappa^n\|_{L^2(0,+\infty;\bm{L}^2(\Omega))}
 + \|\nabla_\Gamma \mu_\kappa^n\|_{L^2(0,+\infty;\bm{L}^2(\Gamma))}
 + \|\nabla_\Gamma \eta_\kappa^n\|_{L^2(0,+\infty;\bm{L}^2(\Gamma))}\leq C,
 \label{ues2a}
\end{align}
 where the constant $C>0$ depends on $\|u_0\|_{L^2(\Omega)}$, $\|\varphi_0\|_{H^1(\Gamma)}$, $\|v_0\|_{L^2(\Gamma)}$, $\langle\varphi_0\rangle_\Gamma$, $\int_\Gamma F(\varphi_0)\,\mathrm{d}S$, $|\Omega|$, $|\Gamma|$ and the coefficients $D$, $\delta$, but not on $n$ and $\kappa$.
 Then we can conclude
\begin{align}
 & \int_t^{t+1} \big(\| \nabla_\Gamma\mu_\kappa^n(\tau)\|_{\bm{L}^2(\Gamma)}^2
 + \| v_\kappa^n (\tau)\|_{H^1(\Gamma)}^2
 + \| \vp_\kappa^n(\tau)\|_{H^2(\Gamma)}^2
 + \| u_\kappa^n(\tau)\|_{H^1(\Omega)}^2\big)\,\mathrm{d}\tau\non \\
 &\quad \leq C,\qquad \forall\, t\geq 0,
 \label{es-unitime1}
\end{align}
 where $C>0$ is independent of  $n$, $\kappa$ and $t$.

 Next, the differential inequality \eqref{esmvH1b} can be written as
\begin{align}
& \frac{\mathrm{d}}{\mathrm{d}t}\big(\|\nabla_\Gamma \mu_\kappa^n\|_{\bm{L}^2(\Gamma)}^2+ \|\nabla_\Gamma v_\kappa^n\|_{\bm{L}^2(\Gamma)}^2\big)
+ \|\nabla_\Gamma  \partial_t \vp_\kappa^n \|_{\bm{L}^2(\Gamma)}^2
+ \frac{4}{\delta}\|\Delta_\Gamma v_\kappa^n\|_{L^2(\Gamma)}^2\non\\
&\quad \leq C \big(\|\nabla_\Gamma \mu_\kappa^n\|_{\bm{L}^2(\Gamma)}^2
+ \|\nabla_\Gamma v_\kappa^n\|_{\bm{L}^2(\Gamma)}^2\big)
+ C\big(\| \vp_\kappa^n\|_{H^2(\Gamma)}^2
+ \| u_\kappa^n\|_{H^1(\Omega)}^2
+ \| v_\kappa^n\|_{L^2(\Gamma)}^2 +1\big).
\label{esmvH1d}
\end{align}
Taking advantage of \eqref{es-unitime1}, we can apply the uniform Gronwall lemma (see e.g., \cite[Chapter III, Lemma 1.1]{T}) to obtain
\begin{align}
\|\nabla_\Gamma \mu_\kappa^n(t)\|_{\bm{L}^2(\Gamma)}^2
+ \|\nabla_\Gamma v_\kappa^n(t)\|_{\bm{L}^2(\Gamma)}^2\leq C,
\qquad \forall\, t\geq 1,
\label{es-unitime2}
\end{align}
where $C>0$ does not depend on $\|\nabla_\Gamma \mu_\kappa^n(0)\|_{\bm{L}^2(\Gamma)}$, $\|\nabla_\Gamma v_\kappa^n(0)\|_{\bm{L}^2(\Gamma)}$, $t$, $n$ and $\kappa$.
On the other hand, the estimate on the finite time interval $[0,1]$ can be derived by directly integrating \eqref{esmvH1d} in time and using the estimate \eqref{es-unitime1} with $t=0$:
\begin{align}
&\|\nabla_\Gamma \mu_\kappa^n(t)\|_{\bm{L}^2(\Gamma)}^2
+ \|\nabla_\Gamma v_\kappa^n(t)\|_{\bm{L}^2(\Gamma)}^2\non\\
&\quad \leq C\big(\|\nabla_\Gamma \mu_\kappa^n(0)\|_{\bm{L}^2(\Gamma)}^2
+ \|\nabla_\Gamma v_\kappa^n(0)\|_{\bm{L}^2(\Gamma)}^2\big)+C,
\qquad \forall\, t\in [0,1],
\label{es-unitime3}
\end{align}
where $C>0$ does not depend on $n$ and $\kappa$. Similar to \eqref{SSSe}, \eqref{SSSe1}, we can apply Poincar\'e's inequality to obtain
\begin{align}
\|\mu_\kappa^n(t)\|_{H^1(\Gamma)}
\leq \|\nabla_\Gamma \mu_\kappa^n(t)\|_{\bm{L}^2(\Gamma)}
+ \left|\int_\Gamma \mu_\kappa^n(t)\,\mathrm{d}S\right|
\leq C,\qquad \forall\, t\geq 0.\label{es-unitime3b}
\end{align}
Besides, we can deduce from \eqref{esmvH1d}--\eqref{es-unitime3} that
\begin{align}
 & \int_t^{t+1} \big(\|\partial_t \vp_\kappa^n(t)\|_{H^1(\Gamma)}^2+ \|\partial_t v_\kappa^n(t)\|_{L^2(\Gamma)}^2\big)\,\mathrm{d}\tau \leq C,\qquad \forall\, t\geq 0.
 \label{es-unitime4}
 \end{align}

We proceed to derive higher-order estimates for $u_\kappa^n$. Taking $\zeta = \partial_t u_\kappa^n $ in \eqref{rw1.a}, using integration by parts, we obtain
\begin{align}
& \frac{\mathrm{d}}{\mathrm{d}t} \mathcal{Y}_\kappa^n(t)+ \|\partial_t u_\kappa^n\|_{L^2(\Omega)}^2\non\\
&\quad  = -\frac{A'(t)}{2}\int_\Gamma (u_\kappa^n)^2 \,\mathrm{d}S +
\frac{2A(t)}{\delta} \int_\Gamma\partial_t \vp_\kappa^nu_\kappa^n\,\mathrm{d}S -\frac{4A(t)}{\delta} \int_\Gamma\partial_t v_\kappa^n u_\kappa^n\,\mathrm{d}S\non\\
& \qquad + \frac{2A'(t)}{\delta}\int_\Gamma (1+\vp_\kappa^n)u_\kappa^n\,\mathrm{d}S  - \frac{4A'(t)}{\delta}\int_\Gamma v_\kappa^n u_\kappa^n\,\mathrm{d}S \non\\
&\quad
\leq C \big(\|\partial_t \vp_\kappa^n\|_{L^2(\Gamma)} + \| \vp_\kappa^n\|_{L^2(\Gamma)} \big) \|u_\kappa^n\|_{L^2(\Gamma)}
+C \big(\|\partial_t v_\kappa^n\|_{L^2(\Gamma)}  + \|  v_\kappa^n\|_{L^2(\Gamma)}  \big)\|u_\kappa^n\|_{L^2(\Gamma)} \non\\
&\qquad +C \|u_\kappa^n\|_{L^2(\Gamma)}^2+C\non\\
&\quad
\leq  C  \|\partial_t \vp_\kappa^n\|_{L^2(\Gamma)} ^2 +  \|\partial_t v_\kappa^n\|_{L^2(\Gamma)}^2+ C \|u_\kappa^n\|_{H^1(\Omega)}^2+ C\| \vp_\kappa^n\|_{L^2(\Gamma)}^2
+C\|v_\kappa^n\|_{L^2(\Gamma)}^2 +C,
\label{dtue}
\end{align}
where
\begin{align}
 \mathcal{Y}_\kappa^n(t)
 &=\frac{D}{2}\|\nabla u_\kappa^n(t)\|_{\bm{L}^2(\Omega)}^2 + \frac{A(t)}{2}\|u_\kappa^n(t)\|_{L^2(\Gamma)}^2 + \frac{2A(t)}{\delta}\int_\Gamma (1+\vp_\kappa^n(t))u_\kappa^n(t)\,\mathrm{d}S\non\\
 &\quad - \frac{4A(t)}{\delta}\int_\Gamma v_\kappa^n(t) u_\kappa^n(t)\,\mathrm{d}S.\non
\end{align}
Using the assumption \textbf{(H5)}, H\"{o}lder's and Young's inequalities, we find
$$
\mathcal{Y}_\kappa^n(t)\leq C_1\big(\|u_\kappa^n(t)\|_{H^1(\Omega)}^2 + \|\vp_\kappa^n(t)\|_{L^2(\Gamma)}^2 +  \|v_\kappa^n(t)\|_{L^2(\Gamma)}^2 +1\big),
$$
$$
\mathcal{Y}_\kappa^n(t)\geq C_2 \|u_\kappa^n(t)\|_{H^1(\Omega)}^2 -
C_3\big(\|u_\kappa^n(t)\|_{L^2(\Omega)}^2+ \|\vp_\kappa^n(t)\|_{L^2(\Gamma)}^2 +  \|v_\kappa^n(t)\|_{L^2(\Gamma)}^2 +1\big),
$$
where the positive constants $C_1,C_2,C_3$ depend  on $D$, $A$, $\delta$, $\Omega$ and $\Gamma$, but are independent of $n$ and $\kappa$.
In view of \eqref{ues2a}, there again exists some $C_4>0$ such that $$\widetilde{\mathcal{Y}}_\kappa^n(t)=\mathcal{Y}_\kappa^n(t)+C_4\geq 0,\qquad \forall\, t\geq 0.$$
Moreover, it follows from \eqref{dtue} that
\begin{align}
\frac{\mathrm{d}}{\mathrm{d}t} \widetilde{\mathcal{Y}}_\kappa^n(t)+ \|\partial_t u_\kappa^n\|_{L^2(\Omega)}^2
& \leq C  \widetilde{\mathcal{Y}}_\kappa^n(t) + C \big(\|\partial_t \vp_\kappa^n\|_{L^2(\Gamma)} ^2 +   \|\partial_t v_\kappa^n\|_{L^2(\Gamma)}^2 +1\big).
\label{dtue1}
\end{align}
By an argument similar to that for \eqref{es-unitime2}, \eqref{es-unitime3} and \eqref{es-unitime4}, we eventually obtain
\begin{align}
\|u_\kappa^n(t)\|_{H^1(\Omega)}^2 +\int_t^{t+1} \|\partial_t u_\kappa^n(\tau)\|_{L^2(\Omega)}^2\,\mathrm{d}\tau \leq C,\qquad \forall\,t\geq 0,\label{es-unitime5}
\end{align}
where $C>0$ depends on $\|u_0\|_{H^1(\Omega)}$, $\|\varphi_0\|_{H^2(\Gamma)}$, $\|v_0\|_{H^1(\Gamma)}$, $\|\widetilde{\mu}_0\|_{H^1(\Gamma)}$, $\langle\varphi_0\rangle_\Gamma$, $\int_\Gamma F(\varphi_0)\,\mathrm{d}S$, $\Omega$, $\Gamma$ and the structural coefficients, but not on $n$, $\kappa$ and $t$.

Combining the above uniform estimates, we are able to pass to the limit and obtain a unique global weak solution to problem \eqref{1.a}--\eqref{ini2} on the whole interval $[0,+\infty)$ with corresponding regularity properties. In particular, since
$\mu,\, \vp,\, v\in L^\infty(0,+\infty; H^1(\Gamma))$, then by Lemma \ref{lem:sep}, we can establish the uniform strict separation property
\begin{align}
\|\vp(t)\|_{L^\infty(\Gamma)}\leq 1-\sigma\quad \text{for almost all}\ t\geq 0,
\label{vp-sep4}
\end{align}
for some constant $\sigma\in (0,\sigma_0]$ independent of time. This together with the surface elliptic theory for $\vp$ leads to $\vp \in L^\infty(0,+\infty;H^3(\Gamma))$ as well as $\vp \in L^2(0,T; H^4(\Gamma))$ for any $T>0$. By interpolation, we again have  $\vp \in C([0,+\infty);C(\Gamma))$ so that the conclusion \eqref{vp-sep2} holds. Since we can verify that $q\in L^2(0,T; H^{1/2}(\Gamma))$, then by the elliptic estimate for $u$ (see e.g., \cite{GT01,La}), we also find
$$
\|u\|_{L^2(0,T;H^2(\Omega))}\leq C \|\partial_t u\|_{L^2(0,T;L^2(\Omega))}+ C\|q\|_{L^2(0,T;H^\frac12(\Gamma))}
\leq C(T),
$$
for any $T>0$.

The proof of Proposition \ref{pro:str}-(2) is complete.
\hfill $\square$

\subsubsection{Proof of Theorem \ref{thm:reg}}
We first consider the case with a general mass exchange term $q$ that is globally Lipschitz continuous. Let $(u,\vp,v,\mu,\eta)$ be the unique global weak solution to problem  \eqref{1.a}--\eqref{ini2} on $[0,T]$ as determined in Theorem \ref{thm:weak}-(1) and (2), subject to the initial data $(u_0,\vp_0,v_0)$.
Thanks to the energy identity \eqref{BEL6}, for any given $t_0\in (0,T)$,
we have
\begin{align*}
 & \int_0^{t_0} \big(\|u(\tau)\|_{H^1(\Omega)}^2+ \| \mu(\tau)\|_{H^1(\Gamma)}^2
 + \| v (\tau)\|_{H^1(\Gamma)}^2
 + \| \vp(\tau)\|_{H^2(\Gamma)}^2\big)\,\mathrm{d}\tau  \leq C_1,
 \end{align*}
where the constant $C_1>0$ depends on $\|u_0\|_{L^2(\Omega)}$, $\|\varphi_0\|_{H^1(\Gamma)}$, $\|v_0\|_{L^2(\Gamma)}$, $\langle\varphi_0\rangle_\Gamma$, $\int_\Gamma F(\varphi_0)\,\mathrm{d}S$, $\Omega$, $\Gamma$, $T$ and coefficients of the system. Then there must exist $t_1\in (0,t_0)$ such that
$$
\|u(t_1)\|_{H^1(\Omega)}^2 + \| \mu(t_1)\|_{H^1(\Gamma)}^2
 + \| v (t_1)\|_{H^1(\Gamma)}^2
 + \| \vp(t_1)\|_{H^2(\Gamma)}^2 \leq \frac{2C_1}{t_0}.
$$
Hence, taking $(u(t_1),\vp(t_1),v(t_1),\mu(t_1))$ as the initial datum, we can find a unique solution $(\widetilde{u},\widetilde{\vp},\widetilde{v},\widetilde{\mu},\widetilde{\eta})$ on $[t_1,T]$ with better regularity properties as stated in Proposition \ref{pro:str}-(1). Moreover, $\widetilde{\vp}$ satisfies the strict separation property on $[t_1,T]$. This solution indeed coincides with $(u,\vp,v,\mu,\eta)$ on $[t_1,T]$ thanks to the uniqueness of weak solutions.

Next, in the equilibrium case with $q$ given by \eqref{q1}, we use the dissipative energy identity \eqref{BELweak} instead.
For any $t_0\in(0,1]$,
\begin{align*}
 & \int_0^{t_0} \big(\| u(\tau)\|_{H^1(\Omega)}^2+ \| \mu(\tau)\|_{H^1(\Gamma)}^2
 + \| v (\tau)\|_{H^1(\Gamma)}^2
 + \| \vp(\tau)\|_{H^2(\Gamma)}^2\big)\,\mathrm{d}\tau  \leq C_2,
 \end{align*}
where the constant $C_2>0$ depends on $\|u_0\|_{L^2(\Omega)}$, $\|\varphi_0\|_{H^1(\Gamma)}$, $\|v_0\|_{L^2(\Gamma)}$, $\langle\varphi_0\rangle_\Gamma$, $\int_\Gamma F(\varphi_0)\,\mathrm{d}S$, $\Omega$, $\Gamma$ and coefficients of the system. There again exists some $t_1\in (0,t_0)$ such that
$$
\| u(t_1)\|_{H^1(\Omega)}^2+ \| \mu(t_1)\|_{H^1(\Gamma)}^2
 + \| v (t_1)\|_{H^1(\Gamma)}^2
 + \| \vp(t_1)\|_{H^2(\Gamma)}^2 \leq \frac{2C_2}{t_0}.
$$
Thus,
$(u(t_1),\vp(t_1),v(t_1),\mu(t_1))$ satisfies the additional assumptions of Proposition \ref{pro:str}-(2). The conclusion of Theorem \ref{thm:reg}-(2) follows from a similar argument with the help of those uniform-in-time estimates and the uniqueness of weak solution.

The proof of Theorem \ref{thm:reg} is complete.
\hfill $\square$

\section{Long-time Behavior of the Full Bulk-Surface Coupled System: the Equilibrium Case}
\setcounter{equation}{0}
\label{sec:LTBS}

In this section, we investigate the long-time behavior of global weak solutions to problem \eqref{1.a}--\eqref{ini2} in the equilibrium case \eqref{q1} with an asymptotically
autonomous mass exchange coefficient $A(t)$.

\subsection{The $\omega$-limit set}

For any given initial datum $(u_0,\varphi_0,v_0)$ satisfying \eqref{cini}, we denote the  corresponding unique global weak solution to problem \eqref{1.a}--\eqref{ini2}
 by  $(u,\varphi,\mu,v,\eta)$. Under the assumptions of Theorem \ref{thm:conv}, we find that $(u,\varphi,\mu,v,\eta)$ is well defined on the whole interval $[0,+\infty)$ and satisfies the conclusion of Theorem \ref{thm:reg}-(2).

The $\omega$-limit set of $(u_0,\varphi_0,v_0)$ is defined as follows
\begin{align}
& \omega(u_0,\varphi_0,v_0)\label{ome}\\
& \quad =\big\{(u_\infty,\varphi_\infty,v_\infty)\in H^1(\Omega)\times H^3(\Gamma)\times H^1(\Gamma)\ \big|\ \exists\,\{t_n\}\nearrow +\infty \ \text{such that}\ \non \\
& \qquad \quad  (u(t_n),\varphi(t_n),v(t_n))\to (u_\infty,\varphi_\infty,v_\infty)\ \text{in}\ L^2(\Omega)\times H^2(\Gamma)\times L^2(\Gamma)\ \text{as}\ t_n\to +\infty\big\}.\non
\end{align}

\begin{proposition}[Characterization of the $\omega$-limit set]
\label{pro:ome}
Suppose that the assumptions for Theorem \ref{thm:reg}-(2) are satisfied. Besides, we assume $A$ is asymptotically autonomous such that
\begin{align}
\lim_{t\to+\infty}A(t)=A_\infty,\label{conA}
\end{align}
with some constant $A_\infty\geq 0$.
Then for any initial datum $(u_0,\varphi_0,v_0)$ satisfying \eqref{cini}, we have
\begin{itemize}
\item[(1)] the $\omega$-limit set $\omega(u_0,\varphi_0,v_0)$ is  nonempty, bounded in $H^1(\Omega)\times H^3(\Gamma)\times H^1(\Gamma)$ and thus compact in $L^2(\Omega)\times H^2(\Gamma)\times L^2(\Gamma)$; there exists some constant  $\sigma_\infty\in (0,1)$ such that for every $(u_\infty,\varphi_\infty,v_\infty)\in \omega(u_0,\varphi_0,v_0)$ its component $\vp_\infty$ satisfies  $\|\vp_\infty\|_{C(\Gamma)}\leq 1-\sigma_\infty$;
\item[(2)] the total free energy $\mathcal{E}(u,\vp,v)$ is a constant on $\omega(u_0,\varphi_0,v_0)$;
\item[(3)] every
element $(u_\infty,\varphi_\infty,v_\infty)\in \omega(u_0,\varphi_0,v_0)$ satisfies the stationary problem \eqref{sta1}--\eqref{sta5}, where $u_\infty, \mu_\infty,\eta_\infty$ are three constants.
\end{itemize}
\end{proposition}
\begin{proof}
From the uniform estimate \eqref{unit-es1}, the compact embeddings $H^1(\Omega)\hookrightarrow\hookrightarrow L^2(\Omega)$, $H^k(\Gamma)\hookrightarrow\hookrightarrow H^{k-1}(\Gamma)$, $k=1,3$, we find that for any unbounded increasing sequence $\{t_n\}$, there exists a subsequence (not relabelled for simplicity) such that as $t_n\to +\infty$,
\begin{equation}
(u(t_n),\varphi(t_n),v(t_n))\to (u_\infty,\varphi_\infty,v_\infty)\ \ \ \text{strongly in}\ \ \  L^2(\Omega)\times H^2(\Gamma)\times L^2(\Gamma), \label{cc1}
\end{equation}
with its limit $(u_\infty,\varphi_\infty,v_\infty)\in H^1(\Omega)\times H^3(\Gamma)\times H^1(\Gamma)$. This easily yields that $ \omega(u_0,\varphi_0,v_0)$ is a nonempty bounded subset in $H^1(\Omega)\times H^3(\Gamma)\times H^1(\Gamma)$.
 The compactness of $\omega(u_0,\varphi_0,v_0)$ follows from the compact embedding $H^1(\Omega)\times H^3(\Gamma)\times H^1(\Gamma) \hookrightarrow\hookrightarrow L^2(\Omega)\times H^2(\Gamma)\times L^2(\Gamma)$ and the fact that it is indeed closed (cf. \cite{T} for more general results). Thanks to the strict separation property of the solution $\vp$ (for $t\geq 1$) and the Sobolev embedding theorem $H^2(\Gamma)\hookrightarrow C(\Gamma)$, there exists some $\sigma_\infty\in (0,1)$ such that $\|\vp_\infty\|_{C(\Gamma)}\leq 1-\sigma_\infty$ for all limit functions $\vp_\infty$.

Next, from the energy identity \eqref{BELweak}, we infer that the total free energy $\mathcal{E}$ is continuous and non-increasing in time. Since $\mathcal{E}$ is bounded from below (by a constant that only depends on $\Gamma$), we find that there exists some $\mathcal{E}_\infty \in \mathbb{R}$ and
\begin{align}
\lim_{t\to +\infty}\mathcal{E}(u(t),\vp(t),v(t))=\mathcal{E}_\infty.\label{decayE}
\end{align}
Hence, $\mathcal{E}$ must be a constant on $\omega(u_0,\varphi_0,v_0)$.

Given any $(u_\infty,\varphi_\infty,v_\infty)\in \omega(u_0,\varphi_0,v_0)$, we consider the unbounded increasing sequence $\{t_n\} \nearrow +\infty$ such that \eqref{cc1} holds.
Thanks to the assumption \textbf{(H5)} and \eqref{unit-es1}, after extracting a subsequence (not relabelled for simplicity), we have
\begin{equation}
(\mu(t_n),\eta(t_n))\to (\mu_\infty,\eta_\infty)\ \ \ \text{strongly in}\ \ \  L^2(\Gamma)\times L^2(\Gamma),
 \label{cc2}
\end{equation}
with some $(\mu_\infty,\eta_\infty)\in H^1(\Gamma)\times H^1(\Gamma)$.
From \eqref{1.f}, we easily have
\begin{align}
\eta_\infty= \frac{4}{\delta}\left(v_\infty-\frac{1+\varphi_\infty}{2}\right).
\label{inf-eta}
\end{align}
Thanks to the strict separation property of $\vp(t)$ and $\vp_\infty$, we can take the limit as $t_n\to +\infty$ in \eqref{1.d} to get
\begin{align}
& \mu_\infty = -\Delta_\Gamma \vp_\infty + W'(\vp_\infty)-\frac12\eta_\infty.\label{inf-mu}
\end{align}
Without loss of generality, we assume $t_{n+1}\geq t_n+1$ for all $n\in \mathbb{Z}^+$. In analogy to \eqref{BELweak}, we have the energy identity on $[t_n,t_{n+1}]$:
\begin{align}
 & \mathcal{E}(u(t_{n+1}),\varphi(t_{n+1}),v(t_{n+1}))
 + D\int_{t_n}^{t_{n+1}}\!\int_\Omega |\nabla u(\tau)|^2\,\mathrm{d}x \mathrm{d}\tau \notag\\
 &\qquad
 + \int_{t_n}^{t_{n+1}}\! \int_\Gamma \big(|\nabla_\Gamma \mu(\tau)|^2+|\nabla_\Gamma \eta(\tau)|^2\big)\,\mathrm{d}S\mathrm{d}\tau
 +\int_{t_n}^{t_{n+1}}\! A(\tau)\int_\Gamma |(\eta-u)(\tau)|^2 \,\mathrm{d}S\mathrm{d}\tau  \nonumber\\
 &\quad  =\mathcal{E}(u(t_n),\varphi(t_n),v(t_n)).
\label{BELweakn}
 \end{align}
It follows from \eqref{decayE} and \eqref{BELweakn} that
 \begin{align}
 & \int_{0}^{1}\!\int_\Omega |\nabla u(t_n+\tau)|^2\,\mathrm{d}x \mathrm{d}\tau
 + \int_{0}^{1}\! \int_\Gamma \big(|\nabla_\Gamma \mu(t_n+\tau)|^2+|\nabla_\Gamma \eta(t_n+\tau)|^2\big)\,\mathrm{d}S\mathrm{d}\tau\non\\
 &\qquad +\int_{0}^{1}\! A(t_n+\tau)\int_\Gamma |(\eta-u)(t_n+\tau)|^2 \,\mathrm{d}S\mathrm{d}\tau \non\\
 &\quad \leq \int_{t_n}^{t_{n+1}}\!\int_\Omega |\nabla u(\tau)|^2\,\mathrm{d}x \mathrm{d}\tau
 + \int_{t_n}^{t_{n+1}}\! \int_\Gamma \big(|\nabla_\Gamma \mu(\tau)|^2+|\nabla_\Gamma \eta(\tau)|^2\big)\,\mathrm{d}S\mathrm{d}\tau\non\\
&\qquad
 +\int_{t_n}^{t_{n+1}}\! A(\tau)\int_\Gamma |(\eta-u)(\tau)|^2 \,\mathrm{d}S\mathrm{d}\tau\non\\
 &\quad \to 0,\qquad \text{as}\ t_n\to +\infty.
 \label{BEL-de}
 \end{align}
On the other hand,  from \eqref{1.a}, \eqref{1.c} and \eqref{1.e} we can deduce the following estimates on time derivatives
\begin{align}
& \|\partial_t u\|_{(H^1(\Omega))'}\leq C\big(\|\nabla  u\|_{\bm{L}^2(\Omega)} + \|q\|_{L^2(\Gamma)}\big), \label{vu}\\
& \|\partial_t \vp\|_{(H^1(\Gamma))'}\leq \|\nabla_\Gamma \mu\|_{\bm{L}^2(\Gamma)}, \label{vpt}\\
& \|\partial_t v\|_{(H^1(\Gamma))'}\leq \|\nabla_\Gamma \eta\|_{\bm{L}^2(\Gamma)} + \|q\|_{L^2(\Gamma)}, \label{vt}
\end{align}
which combined with \eqref{BEL-de}, \eqref{q1} and the assumption \textbf{(H5)} yield (cf. \cite{HT01})
\begin{align}
& \int_0^1 \big( \|\partial_t u(t_n+\tau)\|_{(H^1(\Omega))'}^2 + \|\partial_t \vp(t_n+\tau)\|_{(H^1(\Gamma))'}^2\big)\,\mathrm{d}\tau \non\\
&\qquad + \int_0^1 \|\partial_t v(t_n+\tau)\|_{(H^1(\Gamma))'}^2  \,\mathrm{d}\tau\to 0, \qquad \text{as}\ t_n\to +\infty.
\label{conv-pvt}
\end{align}
As a consequence, it holds
\begin{align}
& \|u(t_n+\tau)-u(t_n+\tau')\|_{(H^1(\Omega))'}\to 0,
\label{conv-uni1}\\
& \|\vp(t_n+\tau)-\vp(t_n+\tau')\|_{(H^1(\Gamma))'}\to 0,\qquad
\|v(t_n+\tau)-v(t_n+\tau')\|_{(H^1(\Gamma))'} \to 0,
\label{conv-uni}
\end{align}
uniformly for all $\tau, \tau'\in [0,1]$.

From \eqref{conv-uni1}, \eqref{conv-uni}, the uniform boundedness of $(u(t),\vp(t),v(t))$ in $H^1(\Omega)\times H^3(\Gamma)\times H^1(\Gamma)$ (for $t\geq 1$) and the sequential strong convergence of $(u(t_n),\vp(t_n),v(t_n))$ in $L^2(\Omega)\times H^2(\Gamma)\times L^2(\Gamma)$, we obtain for all $\tau \in [0,1]$,
\begin{align*}
& \lim_{t_n\to +\infty} \big(\|u(t_n+\tau)-u_\infty\|_{L^2(\Omega)}+ \|\vp(t_n+\tau)-\vp_\infty\|_{H^2(\Gamma)}+ \|v(t_n+\tau)-v_\infty\|_{L^2(\Gamma)}\big)=0,
\end{align*}
which together with \eqref{1.f} and \eqref{inf-eta} also yields
$$
\lim_{t_n\to +\infty}\|\eta(t_n+\tau)-\eta_\infty\|_{L^2(\Gamma)}=0,\qquad \forall\, \tau \in [0,1].
$$
Besides, thanks to the strict separation property of $\vp(t)$ and $\vp_\infty$, we deduce from \eqref{1.d} and \eqref{inf-mu} that
\begin{align}
& \|\mu(t_n+\tau)-\mu_\infty\|_{L^2(\Gamma)}\non\\
&\quad=  \left\|-\Delta_\Gamma(\vp(t_n+\tau)-\vp_\infty)
+\big[W'(\vp(t_n+\tau))-W'(\vp_\infty)\big] -\frac12(\eta(t_n+\tau)-\eta_\infty)\right\|_{L^2(\Gamma)}\non\\
&\quad \leq C\|\vp(t_n+\tau)-\vp_\infty\|_{H^2(\Gamma)}+\frac12\|\eta(t_n+\tau)-\eta_\infty\|_{L^2(\Gamma)}
\to 0, \qquad \text{as}\ \ t_n\to +\infty,\non
\end{align}
for all $\tau \in [0,1]$. Hence, for any $\zeta\in L^2(\Gamma)$, using Lebesgue's dominated convergence theorem, Poincar\'{e} inequality and \eqref{BEL-de}, we find
\begin{align*}
& \left|\int_\Gamma \Big(-\Delta_\Gamma \vp_\infty + W'(\vp_\infty)-\frac{1}{|\Gamma|}\int_\Gamma W'(\vp_\infty)\,\mathrm{d}S\Big)\zeta\,\mathrm{d}S\right|\\
&\quad = \lim_{t_n\to +\infty}\left| \int_0^1 \int_\Gamma \Big[-\Delta_\Gamma \vp(t_n+\tau) + W'(\vp(t_n+\tau))-\frac{1}{|\Gamma|}\int_\Gamma W'(\vp(t_n+\tau))\,\mathrm{d}S\Big]\zeta\,\mathrm{d}S\mathrm{d}\tau\right|\\
&\quad = \lim_{t_n\to +\infty}\left| \int_0^1 \int_\Gamma \Big[ \big(\mu(t_n+\tau)- \langle \mu(t_n+\tau)\rangle_\Gamma\big) + \frac12\big(\eta(t_n+\tau)- \langle \eta(t_n+\tau)\rangle_\Gamma\big) \Big]\zeta\,\mathrm{d}S\mathrm{d}\tau\right|\\
&\quad \leq C\lim_{t_n\to +\infty}  \left(\int_0^1\|\nabla_\Gamma \mu(t_n+\tau)\|_{\bm{L}^2(\Gamma)}^2\,\mathrm{d}\tau\right)^\frac12 \|\zeta\|_{L^2(\Gamma)}\\
&\qquad + C\lim_{t_n\to +\infty}  \left(\int_0^1\|\nabla_\Gamma \eta(t_n+\tau)\|_{\bm{L}^2(\Gamma)}^2\,\mathrm{d}\tau\right)^\frac12 \|\zeta\|_{L^2(\Gamma)}\\
&\quad =0,
\end{align*}
which implies that  $\vp_\infty\in H^3(\Gamma)$ is a strong solution to the nonlocal elliptic equation \eqref{sta1} (the fact $\langle\vp_\infty\rangle_\Gamma= \langle\vp_0\rangle_\Gamma$ easily follows from the mass conservation).

Next, the pointwise convergence of $v(t_n+\tau)$ in $\Gamma\times(0,1)$ together with its uniform boundedness in $L^2(0,1;H^1(\Gamma))$ implies the weak convergence
$$
\lim_{t_n\to +\infty} \int_0^1\!\int_\Gamma \big(v(t_n+\tau)- v_\infty\big)\zeta\,\mathrm{d}S\mathrm{d}\tau=0,\qquad \forall\, \zeta\in L^2(0,1;(H^1(\Gamma))').
$$
Then
for any $\zeta\in H^1(\Omega)$, we have
\begin{align*}
& \left|\int_\Gamma
\Big(-\frac{4}{\delta}\nabla_\Gamma v_\infty + \frac{2}{\delta}\nabla_\Gamma \vp_\infty \Big)\cdot  \nabla_\Gamma \zeta \,\mathrm{d}S\right|\\
&\quad =
\lim_{t_n\to +\infty}
\left|\int_0^1 \int_\Gamma
 \Big(-\frac{4}{\delta} \nabla_\Gamma v(t_n+\tau) + \frac{2}{\delta}  \nabla_\Gamma \vp(t_n+\tau) \Big) \cdot \nabla_\Gamma \zeta  \,\mathrm{d}S\mathrm{d}\tau\right|\\
&\quad =
\lim_{t_n\to +\infty}
\left|\int_0^1 \int_\Gamma
\nabla_\Gamma\eta(t_n+\tau) \cdot  \nabla_\Gamma \zeta
  \,\mathrm{d}S\mathrm{d}\tau\right|\\
&\quad \leq \lim_{t_n\to +\infty}  \left(\int_0^1\|\nabla_\Gamma \eta(t_n+\tau)\|_{\bm{L}^2(\Gamma)}^2\,\mathrm{d}\tau\right)^\frac12 \|\nabla\zeta\|_{\bm{L}^2(\Gamma)}\\
&\quad =0.
\end{align*}
 Hence, $v_\infty\in H^1(\Gamma)$ is a weak solution to the elliptic equation \eqref{sta4}. Then the surface elliptic theory implies $v_\infty \in H^2(\Gamma)$. In a similar manner, we obtain
 $$
 \int_\Gamma \nabla_\Gamma\eta_\infty\cdot \nabla_\Gamma \zeta\,\mathrm{d}S = \int_\Gamma \nabla_\Gamma\mu_\infty\cdot \nabla_\Gamma\zeta\,\mathrm{d}S = 0,\qquad \forall\,\zeta\in H^1(\Omega),
 $$
 so that both $\eta_\infty$ and $\mu_\infty$ are constants. Integrating \eqref{1.d} and \eqref{1.f} over $\Gamma$ respectively and taking $t_n\to +\infty$ (up to a subsequence), we thus arrive at \eqref{sta3} and \eqref{sta5}. Finally, by \eqref{conA} and \eqref{BEL-de}, we can conclude $A_\infty(u_\infty-\eta_\infty)=0$.

The proof of Proposition \ref{pro:ome} is complete.
\end{proof}
\begin{remark}\label{rem:ccov}
Proposition \ref{pro:ome} gives a \emph{sequential convergence} result for every global weak solution $(u,\vp,v,\mu,\eta)$ to problem \eqref{1.a}--\eqref{ini2} as $t\to+\infty$. However, the structure of $\omega(u_0,\varphi_0,v_0)$ can be rather complicated in general (due to the non-convexity of the free energy), which prevents us to conclude the \emph{uniqueness of asymptotic limit}.

Indeed, according to \eqref{sta2} and \eqref{sta5}, we can determine the constant $u_\infty$ and $\eta_\infty$ by the integral $\int_\Gamma v_\infty \,\mathrm{d}S$. Next, from \eqref{sta4}, the function $v_\infty$ can be determined by $\vp_\infty$ and its integral $\int_\Gamma v_\infty \,\mathrm{d}S$, while the constant $\mu_\infty$ can be determined by $\vp_\infty$ and $\int_\Gamma v_\infty \,\mathrm{d}S$ via \eqref{sta3}. Hence, the key issue is to fix $\vp_\infty$ and  $\int_\Gamma v_\infty \,\mathrm{d}S$. If we assume in addition that $A_\infty>0$, then from the identity $A_\infty(u_\infty-\eta_\infty)=0$ we can obtain an additional constraint $u_\infty=\eta_\infty$. This relation together with \eqref{sta2} and \eqref{sta5} enables us to uniquely determine $\int_\Gamma v_\infty \,\mathrm{d}S$ by the following equation:
$$
\frac{M}{|\Omega|}-\frac{1}{|\Omega|} \int_\Gamma v_\infty \,\mathrm{d}S = \frac{4}{\delta|\Gamma|}\int_\Gamma v_\infty \,\mathrm{d}S -\frac{2}{\delta}-\frac{2m}{\delta|\Gamma|},
$$
where
$$
M:=\int_\Omega u_0 \,\mathrm{d} x+ \int_\Gamma v_0\,\mathrm{d}S,\qquad  m:=  \int_\Gamma \vp_0\,\mathrm{d}S.
$$
Nevertheless, since the double well potential $W$ can be non-convex (recall the assumption \textbf{(H1)}), we cannot expect the uniqueness of solution to the nonlinear nonlocal elliptic equation \eqref{sta1} (see e.g., \cite{AW07,Mi19,RH99} for a similar situation arising in the study of the classical Cahn-Hilliard equation in a Cartesian domain $\Omega\subset \mathbb{R}^d$ with $d\in\{2,3\}$). For further discussions on the connection with stationary points of the Cahn-Hilliard equation, we refer to \cite[Section 3.3]{GKRR16}.
\end{remark}

\subsection{Convergence to equilibrium}

Under additional assumptions that the potential function $W$ is real analytic and the mass exchange coefficient $A$ decays fast enough, we proceed to show that the $\omega$-limit set $\omega(u_0,\varphi_0,v_0)$ consists of a single point, namely, every bounded global weak solution converges to a single equilibrium as $t\to +\infty$.

\subsubsection{An auxiliary energy inequality}

Let $(u,\varphi,\mu,v,\eta)$ be the unique global weak solution to problem \eqref{1.a}--\eqref{ini2} corresponding to the initial datum $(u_0,\varphi_0,v_0)$ satisfying \eqref{cini}.

Integrating \eqref{1.a} over $\Omega$ and \eqref{1.e} over $\Gamma$, respectively, using \eqref{q1}, we obtain
\begin{align}
&\frac{\mathrm{d}}{\mathrm{d}t}\int_\Omega u\,\mathrm{d}x= A(t)\int_\Gamma (\eta-u)\,\mathrm{d}S,\label{mau}\\
&\frac{\mathrm{d}}{\mathrm{d}t}\int_\Gamma v\,\mathrm{d}S= -A(t)\int_\Gamma (\eta-u)\,\mathrm{d}S,\label{mav}
\end{align}
for almost all $t>0$. From \eqref{1.f}, \eqref{mav} and the mass conservation
\begin{align}
\int_\Gamma \vp(t)\,\mathrm{d}S=\int_\Gamma \vp_0\,\mathrm{d}S,\qquad \forall\, t\geq 0,
\label{mavp}
\end{align}
we also have
\begin{align}
\frac{\mathrm{d}}{\mathrm{d}t}\int_\Gamma \eta\,\mathrm{d}S= -\frac{4}{\delta}A(t)\int_\Gamma (\eta-u)\,\mathrm{d}S.\label{mae}
\end{align}

Thanks to the regularity result obtained in Theorem \ref{thm:reg}-(2), we infer from the original system \eqref{1.a}--\eqref{1.f} and \eqref{mau}, \eqref{mae} that
\begin{subequations}
	\begin{alignat}{3}
	& \partial_t (u-\langle u\rangle_\Omega) =D\Delta u - \frac{A(t)}{|\Omega|}\int_\Gamma (\eta-u)\,\mathrm{d}S, \label{m1.a}\\
    & D\partial_{\bm{n}} (u-\langle u\rangle_\Omega)=A(t)(\eta-u), \label{m1.b}\\
    & \frac{\delta}{4}\partial_t (\eta-\langle\eta\rangle_\Gamma)+\frac12\partial_t\vp =\Delta_\Gamma \eta - A(t)(\eta-u) +\frac{A(t)}{|\Gamma|}\int_\Gamma (\eta-u)\,\mathrm{d}S,
     \label{m1.e}
	\end{alignat}
\end{subequations}
where \eqref{m1.a} holds a.e. in $\Omega\times (0,+\infty)$, \eqref{m1.b} and \eqref{m1.e} holds a.e. on $\Gamma\times (0,+\infty)$.

Let us define the auxiliary energy functional
\begin{align}
\widetilde{\mathcal{E}}(u(t),\vp(t),\eta(t))
& = \frac12\|(u-\langle u\rangle_\Omega)(t)\|_{L^2(\Omega)}^2
+  \frac12\|\nabla_\Gamma \vp(t)\|_{L^2(\Gamma)}^2  \non\\
&\quad + \int_\Gamma W(\vp(t))\,\mathrm{d}S + \frac{\delta}{8}\|(\eta-\langle\eta\rangle_\Gamma)(t)\|_{L^2(\Gamma)}^2,\qquad \forall\, t\geq 0.
\label{auxE}
\end{align}
Multiplying \eqref{m1.a} by $u-\langle u\rangle_\Omega$ and integrating over $\Omega$, multiplying \eqref{1.c} by $\mu$ and \eqref{m1.e} by $\eta-\langle\eta\rangle_\Gamma$ respectively, and integrating over $\Gamma$, adding the resultants together, we obtain
\begin{align}
& \frac{\mathrm{d}}{\mathrm{d}t} \widetilde{\mathcal{E}}(u(t),\vp(t),\eta(t))
+ D\|\nabla u\|_{\bm{L}^2(\Omega)}^2 +\|\nabla_\Gamma \mu\|_{\bm{L}^2(\Gamma)}^2 + \|\nabla_\Gamma \eta\|_{\bm{L}^2(\Gamma)}^2 \non\\
&\quad = - \frac{A(t)}{|\Omega|}\int_\Gamma (\eta-u)\,\mathrm{d}S \underbrace{\int_\Omega (u-\langle u\rangle_\Omega)\,\mathrm{d}x}_{=0}\, + A(t) \int_\Gamma (\eta-u)(u-\langle u\rangle_\Omega)\,\mathrm{d}S\non\\
&\qquad \underbrace{+ \frac12\int_\Gamma \eta\partial_t\vp\,\mathrm{d}S -\frac12\int_\Gamma \partial_t\vp (\eta-\langle\eta\rangle_\Gamma) \,\mathrm{d}S}_{=0}\,
- A(t)\int_\Gamma (\eta-u)(\eta-\langle\eta\rangle_\Gamma) \,\mathrm{d}S \non\\
&\qquad + \frac{A(t)}{|\Gamma|}\int_\Gamma (\eta-u)\,\mathrm{d}S\underbrace{\int_\Gamma (\eta-\langle\eta\rangle_\Gamma) \,\mathrm{d}S}_{=0}\non\\
&\quad = A(t) \int_\Gamma (\eta-u)(u-\langle u\rangle_\Omega)\,\mathrm{d}S- A(t)\int_\Gamma (\eta-u)(\eta-\langle\eta\rangle_\Gamma) \,\mathrm{d}S.
\label{auxE1}
\end{align}
Since we are interested in the long-time behavior, so we just consider the case with positive time, without loss of generality, for $t\geq 1$. Using the uniform higher-order estimate \eqref{unit-es1} for $t\geq 1$, H\"{o}lder's inequality, the interpolation inequality \eqref{inter1} and Poincar\'{e}'s inequality, the right-hand side of \eqref{auxE1} can be estimated as follows
\begin{align}
& \left|A(t) \int_\Gamma (\eta-u)(u-\langle u\rangle_\Omega)\,\mathrm{d}S- A(t)\int_\Gamma (\eta-u)(\eta-\langle\eta\rangle_\Gamma) \,\mathrm{d}S\right|\non\\
&\quad \leq A(t)\|\eta-u\|_{L^2(\Gamma)}\big(\|u-\langle u\rangle_\Omega\|_{L^2(\Gamma)}+ \|\eta-\langle\eta\rangle_\Gamma\|_{L^2(\Gamma)}\big)\non\\
&\quad \leq CA(t)\big(\|\eta\|_{L^2(\Gamma)}+ \|u\|_{H^1(\Omega)}\big)\Big(\|\nabla u\|_{\bm{L}^2(\Omega)}^\frac12\|u-\langle u\rangle_\Omega\|_{L^2(\Omega)}^\frac12 + \|u-\langle u\rangle_\Omega\|_{L^2(\Omega)}\Big)\non\\
& \qquad + CA(t)\big(\|\eta\|_{L^2(\Gamma)}+ \|u\|_{H^1(\Omega)}\big)  \|\eta-\langle\eta\rangle_\Gamma\|_{L^2(\Gamma)}\non\\
&\quad \leq \frac{D}{2}\|\nabla u\|_{\bm{L}^2(\Omega)}^2+ \frac12 \|\nabla_\Gamma \eta\|_{\bm{L}^2(\Gamma)}^2+  \widetilde{C} \big[A(t)\big]^2,\qquad \forall\, t\geq 1,\non
\end{align}
where $\widetilde{C}>0$ depends on $\|u_0\|_{L^2(\Omega)}$, $\|\varphi_0\|_{H^1(\Gamma)}$, $\|v_0\|_{L^2(\Gamma)}$,  $\int_\Gamma F(\varphi_0)\,\mathrm{d}S$, $\langle\varphi_0\rangle_\Gamma$, $\Omega$, $\Gamma$, coefficients of the system.
As a consequence, we can deduce from \eqref{auxE1} that
\begin{align}
& \frac{\mathrm{d}}{\mathrm{d}t} \widetilde{\mathcal{E}}(u,\vp,\eta)
+ \frac{D}{2}\|\nabla u\|_{\bm{L}^2(\Omega)}^2 +\|\nabla_\Gamma \mu\|_{\bm{L}^2(\Gamma)}^2 + \frac{1}{2} \|\nabla_\Gamma \eta\|_{\bm{L}^2(\Gamma)}^2 \leq \widetilde{C} \big[A(t)\big]^2,
\label{auxE2}
\end{align}
for almost all $t\geq 1$.

In view of \eqref{auxE} and \eqref{auxE2}, we define
\begin{align}
\widehat{\mathcal{E}}(u(t),\vp(t),\eta(t))
= \widetilde{\mathcal{E}}(u(t),\vp(t),\eta(t)) + \widetilde{C} \int_t^{+\infty} \big[A(\tau)\big]^2\,\mathrm{d}\tau,\qquad \forall\, t\geq 1.
\label{aux4}
\end{align}
Then it follows that
\begin{align}
& \frac{\mathrm{d}}{\mathrm{d}t} \widehat{\mathcal{E}}(u,\vp,\eta)
+ \frac{D}{2}\|\nabla u\|_{\bm{L}^2(\Omega)}^2 +\|\nabla_\Gamma \mu\|_{\bm{L}^2(\Gamma)}^2 + \frac{1}{2} \|\nabla_\Gamma \eta\|_{\bm{L}^2(\Gamma)}^2 \leq 0,
\label{auxE5}
\end{align}
for almost all $t\geq 1$. Since $\widehat{\mathcal{E}}$ is bounded from below by a constant that only depends on $\Gamma$, there exists some $\widehat{\mathcal{E}}_\infty\in \mathbb{R}$ such that
\begin{align}
\lim_{t\to+\infty} \widehat{\mathcal{E}}(u(t),\vp(t),\eta(t))=\widehat{\mathcal{E}}_\infty.
\label{auxE6}
\end{align}
From the assumption \eqref{decayA}, we find
\begin{align}
\int_t^{+\infty} \big[A(\tau)\big]^2\,\mathrm{d}\tau \leq \frac{C}{2\alpha-1} (1+t)^{1-2\alpha},\qquad \forall\, t\geq 0. \label{decayA1}
\end{align}
This combined with \eqref{auxE6} also yields
\begin{align}
\lim_{t\to+\infty} \widetilde{\mathcal{E}}(u(t),\vp(t),\eta(t))=\widehat{\mathcal{E}}_\infty.
\label{auxE7}
\end{align}
Recalling the definition of $\omega(u_0,\vp_0,v_0)$, we can conclude  that the auxiliary energy functional $\widetilde{\mathcal{E}}(u,\vp,\eta)$ equals to the constant $\widehat{\mathcal{E}}_\infty$ on $\omega(u_0,\vp_0,v_0)$.

\subsubsection{Proof of Theorem \ref{thm:conv}}
We are in a position to prove Theorem \ref{thm:conv}. The proof is based on the well-known {\L}ojasiewicz-Simon approach \cite{LS83} for gradient-like evolution systems (see e.g., \cite{A2009,AW07,GGM2017,GGW18,HJ2001,JWZ,RH99} for applications to the Cahn-Hilliard equation and its extended systems). \medskip

\textbf{Step 1. Convergence of $\vp(t)$}. Define
$$
\widetilde{\omega}(u_0,\vp_0,v_0)
=\big\{\vp_\infty\ \big|\ (u_\infty,\vp_\infty,v_\infty)\in \omega(u_0,\vp_0,v_0)\big\}.
$$
According to Proposition \ref{pro:ome}-(3) and Lemma \ref{lem:LS}-(2), for each $\vp_\infty\in \widetilde{\omega}(u_0,\vp_0,v_0)$, there exist constants $\chi_{\vp_\infty}\in (0,1/2)$ and $C_{\vp_\infty},\,\ell_{\vp_\infty}>0$ such that
the {\L}ojasiewicz-Simon inequality \eqref{LSa} holds for
$$
\vp\in \mathcal{B}_{\vp_\infty}:= \Big\{\vp\in H^2(\Gamma)\ \Big|\ \|\vp-\vp_\infty\|_{H^2(\Gamma)}< \ell_{\vp_\infty}, \ \int_\Gamma\vp\,\mathrm{d}S=\int_\Gamma \vp_\infty\,\mathrm{d}S=\int_\Gamma \vp_0\,\mathrm{d}S\Big\}.
$$
Then the union of balls $\{\mathcal{B}_{\vp_\infty}:\,\vp_\infty\in \widetilde{\omega}(u_0,\vp_0,v_0)\}$
forms an open cover of the set $\widetilde{\omega}(u_0,\vp_0,v_0)$.  Since $\widetilde{\omega}(u_0,\vp_0,v_0)$ is compact in $H^2(\Gamma)$, we can find a finite sub-cover
$$
\widetilde{\omega}(u_0,\vp_0,v_0) \subset \mathcal{U}:= \bigcup_{i=1}^k\mathcal{B}_i,
$$
with $\mathcal{B}_i:=\mathcal{B}_{\vp_\infty^{(i)}}$ such that
$\vp_\infty^{(i)}\in \widetilde{\omega}(u_0,\vp_0,v_0), \, i=1,\cdots,k$ and the corresponding constants are denoted by $\chi^{(i)},\,C^{(i)},\,\ell^{(i)}$, respectively.

By the definition of the $\omega$-limit set $\omega(u_0,\vp_0,v_0)$, there exists a sufficient large time $t_*>1$ such that
$$
\vp(t)\in \mathcal{U},\qquad \forall\, t\geq t_*.
$$
Furthermore, recalling \eqref{auxE7}, the definition \eqref{EEE} and  the fact that $u_\infty$, $\eta_\infty$ are constants, we find
$$
E(\vp_\infty)=\widehat{\mathcal{E}}_\infty,\qquad\forall\,\vp_\infty  \in \widetilde{\omega}(u_0,\vp_0,v_0).
$$
Thus, taking
\begin{align}
\chi_*=\min_{i=1,\cdots,k}\big\{\chi^{(i)}\big\}\in \Big(0,\frac12\Big),\qquad C_*=\max_{i=1,\cdots,k}\big\{C^{(i)}\big\},
\label{chi1}
\end{align}
applying Lemma \ref{lem:LS}-(2) and Poincar\'e's inequality, we can deduce that the trajectory $\vp(t)$ satisfies
\begin{align}
 |E(\vp(t))-\widehat{\mathcal{E}}_\infty|^{1-\chi_*}
 &\leq C_*\left\|-\Delta_\Gamma \varphi(t) +W^{\prime}(\varphi(t))-\displaystyle{\frac{1}{|\Gamma|}\int_\Gamma W^{\prime}(\vp(t))\,\mathrm{d}S}\right\|_{(H^1(\Gamma))'}\non\\
& \leq C \left\| \mu(t)- \langle\mu(t)\rangle_\Gamma \right\|_{L^2(\Gamma)}
+C \left\| \eta(t)- \langle\eta(t)\rangle_\Gamma \right\|_{L^2(\Gamma)}\non\\
&\leq C \big(\left\| \nabla_\Gamma \mu(t) \right\|_{\bm{L}^2(\Gamma)}
+ \left\| \nabla_\Gamma \eta(t) \right\|_{\bm{L}^2(\Gamma)}\big),
\qquad \forall\, t\geq t_*.
\label{LSb}
\end{align}
Define
$$
\rho=\min\left\{\chi_*,\ \frac{\alpha-1}{2\alpha-1}\right\}\in \Big(0,\,\frac12\Big).
$$
Then the estimates \eqref{unit-es1}  and \eqref{LSb}  yield
\begin{align}
|E(\vp(t))-\widehat{\mathcal{E}}_\infty|
& \leq C\big(\left\| \nabla_\Gamma \mu(t) \right\|_{\bm{L}^2(\Gamma)}+ \left\| \nabla_\Gamma \eta(t) \right\|_{\bm{L}^2(\Gamma)}\big)^{\frac{1}{1-\chi_*}} \non\\
& \leq C\big(\left\| \nabla_\Gamma \mu(t) \right\|_{\bm{L}^2(\Gamma)}+ \left\| \nabla_\Gamma \eta(t) \right\|_{\bm{L}^2(\Gamma)}\big)^{\frac{1}{1-\rho}},
\qquad \forall\, t\geq t_*.
\label{LSc}
\end{align}

Define
\begin{align}
Y(t)
=\left( \frac{D}{2}\|\nabla u(t)\|_{\bm{L}^2(\Omega)}^2 +\|\nabla_\Gamma \mu(t)\|_{\bm{L}^2(\Gamma)}^2
+ \frac{1}{2} \|\nabla_\Gamma \eta(t)\|_{\bm{L}^2(\Gamma)}^2 \right)^\frac12,
\qquad \forall\, t\geq t_*.
\label{Y}
\end{align}
It follows from \eqref{auxE2} and \eqref{decayA1} that
\begin{align}
\widetilde{\mathcal{E}}(u(t),\vp(t),\eta(t))- \widehat{\mathcal{E}}_\infty \geq \int_t^{+\infty} \big[Y(\tau)\big]^2\,\mathrm{d}\tau -C (1+t)^{1-2\alpha},
\qquad \forall\, t\geq t_*.\non
\end{align}
which together with \eqref{unit-es1}, \eqref{auxE}, \eqref{LSc} and Poincar\'e's inequality yields
\begin{align}
 \int_t^{+\infty} \big[Y(\tau)\big]^2\,\mathrm{d}\tau
&\leq C\|\nabla u(t)\|_{\bm{L}^2(\Omega)}^2 + C\|\nabla_\Gamma \eta(t)\|_{\bm{L}^2(\Gamma)}^2+ |E(\vp(t))-\widehat{\mathcal{E}}_\infty|
+ C (1+t)^{1-2\alpha}\non\\
&\leq C \big[Y(t)\big]^{\frac{1}{1-\rho}}
+  C (1+t)^{1-2\alpha},
\qquad \forall\, t\geq t_*.
\label{Z1}
\end{align}
Let us introduce the auxiliary function
\begin{align}
Z(t)= Y(t)+ (1+t)^{(1-2\alpha)(1-\rho)}\in L^2(t_*,+\infty).
\label{Z}
\end{align}
Noticing that
\begin{align}
\int_t^{+\infty} \big[(1+\tau)^{(1-2\alpha)(1-\rho)}\big]^2\,\mathrm{d}\tau
\leq \int_t^{+\infty} (1+\tau)^{-2\alpha}\,\mathrm{d}\tau \leq (1+t)^{1-2\alpha},\qquad \forall\, t\geq t_*,
\non
\end{align}
we can deduce from \eqref{Z1} the following inequality:
\begin{align}
 \int_t^{+\infty} \big[Z(\tau)\big]^2\,\mathrm{d}\tau \leq C \big[Z(t)\big]^{\frac{1}{1-\rho}},
 \qquad \forall\, t\geq t_*.
 \label{Z3}
\end{align}
Recall the following elementary lemma (see e.g., \cite[Lemma 4.1]{HT01} or \cite[Lemma 7.1]{FS}):
\begin{lemma}\label{f}
Let $\rho\in(0,1/2)$. Assume that $\mathcal{Z}\geq 0$ be a measurable function on $(0,+\infty)$, $\mathcal{Z}\in L^2(0, +\infty)$ and there exist $C>0$ and
$t_*\geq 0$ such that
\begin{align}
 \int_t^{\infty} [\mathcal{Z}(\tau)]^2\,\mathrm{d}\tau\leq C[\mathcal{Z}(t)]^{\frac{1}{1-\rho}},\quad\text{for almost all}\ \ t\geq t_*.
 \nonumber
\end{align}
Then $\mathcal{Z}\in L^1(0,+\infty)$.
\end{lemma}
\noindent
Hence, in view of \eqref{Z} and \eqref{Z3}, we can apply Lemma \ref{f} to conclude that
\begin{align}
\int_{t_*}^{+\infty}Z(\tau)\,\mathrm{d}\tau<+\infty.
\label{L1Z}
\end{align}
Since $\alpha>1$, it holds
\begin{align}
\int_{t_*}^{+\infty} (1+\tau)^{(1-2\alpha)(1-\rho)} \,\mathrm{d}\tau
\leq \int_{t_*}^{+\infty} (1+\tau)^{-\alpha}\,\mathrm{d}\tau<+\infty.
\non
\end{align}
Then from \eqref{vpt}, \eqref{Y} and \eqref{L1Z}, we can conclude
\begin{align}
\int_{t_*}^{+\infty}\|\partial_t\vp(\tau)\|_{(H^1(\Gamma))'}\,\mathrm{d}\tau<+\infty.
\label{L1vpt}
\end{align}
As a consequence, $\vp(t)$ converges in $(H^1(\Gamma))'$ as $t\to+\infty$. Recalling Proposition \ref{pro:ome}, we can find a unique function $\vp_\infty\in H^3(\Gamma)$ satisfying \eqref{sta1} such that
$$
\lim_{t\to+\infty}\|\vp(t)-\vp_\infty\|_{H^2(\Gamma)}=0.
$$

\textbf{Step 2. Convergence of $v(t)$}. From \eqref{unit-es1}, \eqref{decayA}, \eqref{vt}, \eqref{Y} and \eqref{L1Z}, we obtain
\begin{align}
\int_{t_*}^{+\infty}\|\partial_t v(\tau)\|_{(H^1(\Gamma))'}\,\mathrm{d}\tau
& \leq \int_{t_*}^{+\infty}Z(\tau)\,\mathrm{d}\tau +
\int_{t_*}^{+\infty}A(\tau)\|\eta(\tau)-u(\tau)\|_{L^2(\Gamma)}\,\mathrm{d}\tau\non\\
& \leq \int_{t_*}^{+\infty}Z(\tau)\,\mathrm{d}\tau +C\int_{t_*}^{+\infty}(1+\tau)^{-\alpha}\,\mathrm{d}\tau<+\infty.
\end{align}
Similar to the argument for $\vp$, we can find a unique function $v_\infty\in H^1(\Gamma)$ such that
$$
\lim_{t\to+\infty}\|v(t)-v_\infty\|_{L^2(\Gamma)}=0.
$$
Since we have known that $v_\infty$ is a weak solution to the elliptic equation \eqref{sta4}, thanks to $\vp_\infty\in H^3(\Gamma)$, we infer from the surface elliptic theory that $v_\infty\in H^2(\Gamma)$.\medskip

\textbf{Step 3. Convergence of $u(t)$, $\mu(t)$, $\eta(t)$}. Since the asymptotic limits $\vp_\infty$ and $v_\infty$ have been uniquely determined, according to the discussions in Remark \ref{rem:ccov}, we can uniquely determine the constants $u_\infty$, $\mu_\infty$ and $\eta_\infty$ and arrive at the conclusions \eqref{conv1}, \eqref{conv2}.\medskip

The proof of Theorem \ref{thm:conv} is complete. \hfill $\square$

\subsubsection{Convergence rate}
With the decay property of $A(t)$, we are able to derive some estimates on the convergence rate by applying the energy method as in \cite{Ben,HJ2001,JWZ}.

\begin{corollary}[Convergence rate]\label{cor:rate}
Let the assumptions of Theorem \ref{thm:conv} be satisfied. We have
\begin{align}
& \|u(t)-u_\infty\|_{(H^1(\Omega))'}+ \|\vp(t)-\vp_\infty\|_{(H^1(\Gamma))'}+ \|v(t)-v_\infty\|_{(H^1(\Gamma))'}\non\\
&\quad \leq (1+t)^{-\lambda},\qquad \forall\, t\geq 1,
\label{rate0}
\end{align}
where
$$
\lambda=\min\left\{\frac{\chi_*}{1-2\chi_*},\ \alpha-1\right\}>0,
$$
and $\chi_*\in (0,1/2)$ is given as in \eqref{chi1}.
\end{corollary}
\begin{proof}
Recall \eqref{aux4}, we denote by
$$
\mathcal{G}(t)=\widehat{\mathcal{E}}(u(t),\vp(t),\eta(t)) - \widehat{\mathcal{E}}_\infty\geq 0.
$$
It follows from \eqref{auxE5} that
\begin{align}
& \frac{\mathrm{d}}{\mathrm{d}t} \mathcal{G}(t)+ \big[Y(t)\big]^2 \leq 0,
\label{auxE5a}
\end{align}
for almost all $t\geq t_*$.
On the other hand, we infer from \eqref{decayA1}, \eqref{LSb} and Poincar\'e's inequality that
\begin{align}
\big[\mathcal{G}(t)\big]^{2(1-\chi_*)} & \leq C\big[Y(t)\big]^2 + C'(1+t)^{2(1-\chi_*)(1-2\alpha)},
\label{auxE5b}
\end{align}
which together with \eqref{auxE5a} yields
\begin{align}
\frac{\mathrm{d}}{\mathrm{d}t} \mathcal{G}(t)+ C\big[\mathcal{G}(t)\big]^{2(1-\chi_*)}
\leq C'(1+t)^{2(1-\chi_*)(1-2\alpha)},
\label{auxE5c}
\end{align}
for almost all $t\geq t_*$. Applying \cite[Lemma 2.8]{Ben}, we can obtain the following energy decay
\begin{align}
 \mathcal{G}(t)\leq C(1+t)^{-\lambda_1},\qquad \forall\, t\geq t_*,
 \label{rate1}
\end{align}
where
$$
\lambda_1=\min\left\{\frac{1}{1-2\chi_*},\ 2\alpha-1\right\}>1.
$$
Then arguing as in \cite[Section 3]{Ben}, from \eqref{auxE5a} and \eqref{rate1}, we find that for any $t\geq t_*>1$, it holds
\begin{align}
\int_t^{2t} Y(\tau)\,\mathrm{d}\tau \leq t^\frac12\left(\int_t^{2t} \big[Y(\tau)\big]^2\,\mathrm{d}\tau\right)^\frac12\leq Ct^{\frac{1-\lambda_1}{2}},\qquad \forall\, t\geq t_*.\non
\end{align}
This implies
\begin{align}
\int_t^{+\infty} Y(\tau)\,\mathrm{d}\tau \leq \sum_{i=0}^{+\infty}
\int_{2^it}^{2^{i+1}t} Y(\tau)\,\mathrm{d}\tau\leq
C  \sum_{i=0}^{+\infty} (2^i t)^{-\lambda}
\leq C(1+t)^{-\lambda},\qquad \forall\, t\geq t_*,
\label{rate2}
\end{align}
where
$$
\lambda=\frac12(\lambda_1-1)>0.
$$
From \eqref{vpt} and \eqref{rate2}, we can conclude
\begin{align}
\int_{t}^{+\infty}\|\partial_t\vp(\tau)\|_{(H^1(\Gamma))'}\,\mathrm{d}\tau
\leq C(1+t)^{-\lambda},\qquad \forall\, t\geq t_*.
\label{rate3}
\end{align}
Next,  from  \eqref{unit-es1}, \eqref{decayA}, \eqref{vt} and \eqref{rate2}, we obtain
\begin{align}
\int_{t}^{+\infty}\|\partial_t v(\tau)\|_{(H^1(\Gamma))'}\,\mathrm{d}\tau
&\leq \int_{t}^{+\infty}Y(\tau)\,\mathrm{d}\tau +
\int_{t}^{+\infty}A(\tau)\|\eta(\tau)-u(\tau)\|_{L^2(\Gamma)}\,\mathrm{d}\tau\non\\
& \leq \int_{t}^{+\infty}Y(\tau)\,\mathrm{d}\tau +C\int_{t}^{+\infty}(1+\tau)^{-\alpha}\,\mathrm{d}\tau \non\\
& \leq C(1+t)^{-\lambda},\qquad \forall\, t\geq t_*.
\label{rate4}
\end{align}
We infer from \eqref{vu} that
\begin{align}
\|\partial_t u\|_{(H^1(\Omega))'}
 & \leq
C\|\nabla  u\|_{\bm{L}^2(\Omega)} + C A(t)\big(\|u\|_{L^2(\Gamma)}+ \|\eta\|_{L^2(\Gamma)}\big).
\non
\end{align}
By \eqref{unit-es1}, \eqref{decayA} and \eqref{rate2}, we again get
\begin{align}
\int_{t}^{+\infty}\|\partial_t u(\tau)\|_{(H^1(\Omega))'}\,\mathrm{d}\tau
& \leq C(1+t)^{-\lambda},\qquad \forall\, t\geq t_*.
\label{rate5}
\end{align}
Finally, from \eqref{rate3}, \eqref{rate4} and \eqref{rate5} we easily conclude \eqref{rate0}.
\end{proof}
\begin{remark}
Taking advantage of the decay rate \eqref{rate0} and the uniform-in-time estimate \eqref{unit-es1}, we can obtain decay estimate in higher-order norms by interpolation.
\end{remark}

\section{Reduced System in the Non-equilibrium Case}
\setcounter{equation}{0}
\label{sec:RdS}

In this section, we study the reduced problem \eqref{r1.a}--\eqref{rini2} in the non-equilibrium case, that is, the mass exchange term $q$ takes the specific  form of \eqref{q2}.

\subsection{Well-posedness}
In order to prove Theorem \ref{thm:rweak}, we first investigate the large cytosolic diffusion limit as $D\to +\infty$ for the full bulk-surface coupled system \eqref{1.a}--\eqref{1.f} with a cut-off approximation of the mass exchange term \eqref{q2}.

\begin{proposition}[Large cytosolic diffusion limit]\label{prop:LD}
Let $T\in (0,+\infty)$ and $\{D_n\}_{n\in\mathbb{Z}^+}$ be an increasing sequence with $D_n>0$ and $\lim_{n\to+\infty}D_n=+\infty$.
Suppose that the assumptions \textbf{(H1)}, \textbf{(H3)} are satisfied and the mass exchange term $q$ takes the specific form of \eqref{q2m} under the assumption \textbf{(H4)}. For arbitrary but fixed initial data $(u_0,\,\varphi_0,\,v_0)$ satisfying \eqref{cini}, we denote the weak solution to problem \eqref{1.a}--\eqref{ini2} on $[0,T]$ corresponding to $D=D_n$ by $(u^{D_n},\varphi^{D_n},\mu^{D_n},v^{D_n},\eta^{D_n})$. Then there exists a subsequence $\{D_n\}_{n\in\mathbb{Z}^+}$ (not relabelled for simplicity) such that as $D_n\to +\infty$, it holds
\begin{align*}
& u^{D_n} \rightharpoonup u\ \ \text{weakly in}\ \ L^2(0,T;H^1(\Omega)),\ \ \text{weakly star in}\ \ L^\infty(0,T;L^2(\Omega)),\\
& u(t)\in\mathbb{R}\ \ \text{for any}\ \ t\in[0,T],\ \ \text{and}\ \ \langle u^{D_n} \rangle_\Omega\to u\ \ \text{in}\ \  C([0,T]),\\
& \vp^{D_n} \rightharpoonup \vp\ \ \text{weakly in}\ \ L^{4}(0,T;H^2(\Gamma))\cap L^2(0,T;W^{2,p}(\Gamma))\cap H^1(0,T;(H^1(\Gamma))'),\\
& \vp^{D_n} \rightharpoonup \vp\ \ \text{weakly star in}\ \ L^{\infty}(0,T;H^1(\Gamma)),\\
& v^{D_n} \rightharpoonup v\ \ \text{weakly in}\ \ L^2(0,T;H^1(\Omega))\cap H^1(0,T;(H^1(\Gamma))'),\\
& v^{D_n} \rightharpoonup v\ \ \text{weakly star in}\ \ L^{\infty}(0,T;L^2(\Gamma)),\\
& \mu^{D_n} \rightharpoonup \mu\ \ \text{weakly in}\ \ L^2(0,T;H^1(\Omega)),\\
& \eta^{D_n} \rightharpoonup \eta\ \ \text{weakly in}\ \ L^2(0,T;H^1(\Omega)),
\end{align*}
and the limit function $(u,\vp,v,\mu,\eta)$ is a weak solution to the reduced problem \eqref{r1.a}--\eqref{rini2} on $[0,T]$ in the sense of Definition \ref{def:rweak} with the new initial datum $u|_{t=0}=\langle u_0 \rangle_\Omega$.
\end{proposition}
\begin{proof}
The proof follows the argument for \cite[Porposition 2.6]{AK20}. Below we sketch the main steps and point out necessary modifications. All the convergent results stated below should be understood in the sense of subsequences (not relabelled for simplicity).

In view of \eqref{BEL6}, every weak solution $(u^{D_n},\varphi^{D_n},\mu^{D_n},v^{D_n},\eta^{D_n})$ satisfies the corresponding energy identity
\begin{align}
& \mathcal{E}(u^{D_n}(t),\varphi^{D_n}(t),v^{D_n}(t))
+ \int_0^t\left[D_n\int_\Omega |\nabla u^{D_n}|^2\,\mathrm{d}x +\int_\Gamma \big(|\nabla_\Gamma \mu^{D_n}|^2+|\nabla_\Gamma \eta^{D_n}|^2\big)\,\mathrm{d}S\right]\mathrm{d}\tau \nonumber\\
&\quad  = \mathcal{E}(u_0,\varphi_0,v_0)+ \int_0^t\!\int_\Gamma q^{D_n} (\eta^{D_n}-u^{D_n})\,\mathrm{d}S\mathrm{d}\tau,\qquad \forall\,t\in (0,T],
\label{BEL-Dn}
\end{align}
where
$$
q^{D_n}= B_1u^{D_n} -B_1\widetilde{h}(u^{D_n})v^{D_n}-B_2 v^{D_n}.
$$
Using the assumption \textbf{(H4)}, by an argument similar to that for \eqref{esL5a}, we obtain the following uniform estimate
\begin{align}
&  \|u^{D_n}(t)\|_{L^2(\Omega)}^2
+ \|\varphi^{D_n}(t)\|_{H^1(\Gamma)}^2
+ \|v^{D_n}(t)\|_{L^2(\Gamma)}^2
+ \|\eta^{D_n}(t)\|_{L^2(\Gamma)}^2 + \int_0^t  D_n\|\nabla u^{D_n}(\tau)\|_{\bm{L}^2(\Omega)}^2\,\mathrm{d}\tau  \non\\
&\qquad  +\int_0^t \Big( \|\nabla_\Gamma \mu^{D_n}(\tau)\|_{\bm{L}^2(\Gamma)}^2+\|\nabla_\Gamma \eta^{D_n}(\tau)\|_{\bm{L}^2(\Gamma)}^2+ \|\Delta_\Gamma \varphi^{D_n}(\tau)\|_{L^2(\Gamma)}^2\Big)\,\mathrm{d}\tau \non\\
&\quad \leq C(T),\qquad \forall\, t\in (0,T],
\label{esL5Dn}
\end{align}
where $C(T)>0$ depends on $\|u_0\|_{L^2(\Omega)}$, $\|\varphi_0\|_{H^1(\Gamma)}$, $\|v_0\|_{L^2(\Gamma)}$, $\int_\Gamma F(\varphi_0)\,\mathrm{d}S$, $\Omega$, $\Gamma$, $\delta$ and $T$. In particular, in view of \eqref{uknL2}, we find that the constant $C(T)$ can be chosen independent of $D_n$ whenever $D_n\geq 1$.

By a similar argument as in the proof of Theorem \ref{thm:weak}, we obtain
\begin{align}
& u^{D_n}\in L^\infty(0,T;L^2(\Omega))\cap L^2(0,T;H^1(\Omega)),\notag\\
& \varphi^{D_n} \in L^{\infty}(0,T;H^1(\Gamma))\cap L^{4}(0,T;H^2(\Gamma))\cap L^2(0,T;W^{2,p}(\Gamma)) \cap H^{1}(0,T;(H^1(\Gamma))'),\notag \\
&\mu^{D_n} \in L^{2}(0,T;H^1(\Gamma)),\notag \\
& v^{D_n},\,\eta^{D_n}  \in L^{\infty}(0,T;L^2(\Gamma))\cap L^{2}(0,T;H^1(\Gamma)) \cap H^1(0,T;(H^1(\Gamma))'),\notag\\
& F'(\varphi^{D_n})\in L^2(0,T;L^p(\Gamma)),\non
\end{align}
for any $p\geq 2$, with uniform bounds in the corresponding function spaces for $D_n\geq 1$. Besides, it holds that
$\varphi^{D_n}\in L^{\infty}(\Gamma\times (0,T))$ with $|\varphi^{D_n}(x,t)|<1$ almost everywhere on $\Gamma\times(0,T)$ and $\|\varphi^{D_n}\|_{L^\infty(0,T;L^\infty(\Gamma))}\leq 1$. However, according to \eqref{estvva}, here we lose the uniform estimate for $\partial_t u^{D_n}$ in $L^2(0,T;(H^1(\Omega))')$.

Hence, the expected weak and weak star convergence for $(u^{D_n},\varphi^{D_n},\mu^{D_n},v^{D_n},\eta^{D_n})$ easily follows from the uniform estimates mentioned above. Next, using the Aubin-Lions lemma, we obtain the strong convergence as $D_n\to +\infty$
\begin{align*}
& \vp^{D_n} \to \vp\quad \text{strongly in}\ \ L^{2}(0,T;H^1(\Gamma))
\end{align*}
and thus its a.e. convergence on $\Gamma\times(0,T)$. Therefore, $F'(\vp^{D_n})\to F'(\vp)$ a.e. on $\Gamma\times(0,T)$ as well. This combined with the uniform boundedness of $F'(\vp^{D_n})$ in $L^2(0,T;L^2(\Gamma))$  enables us to conclude that
\begin{align*}
F'(\vp^{D_n})\rightharpoonup F'(\vp)\quad \text{weakly in}\ \ L^2(0,T;L^2(\Gamma)).
\end{align*}
Concerning $u^{D_n}$, we infer from \eqref{esL5Dn} that
$$
 \int_0^T  \|\nabla u^{D_n}(t)\|_{\bm{L}^2(\Omega)}^2\,\mathrm{d}t
 \leq \frac{C(T)}{D_n}\to 0, \qquad \text{as} \ D_n\to +\infty.
$$
Furthermore, from \eqref{esL5Dn}, \eqref{q2m} and the assumption \textbf{(H4)}, we find $\langle u^{D_n}\rangle_\Omega \in L^\infty(0,T)$ is uniformly bounded and
\begin{align*}
\int_0^T\left|\frac{\mathrm{d}}{\mathrm{d}t}\langle u^{D_n}\rangle_\Omega\right|^2\,\mathrm{d}t
& \leq \frac{1}{|\Omega|^2}\int_0^T \left|\int_\Gamma q^{D_n}(t)\,\mathrm{d}S \right|^2\,\mathrm{d}t\non\\
&\leq C\int_0^T \big(\|u^{D_n}(t)\|_{L^2(\Gamma)}^2+\|v^{D_n}(t)\|_{L^2(\Gamma)}^2\big) \,\mathrm{d}t\\
&\leq C\int_0^T \big(\|u^{D_n}(t)\|_{H^1(\Omega)}^2+\|v^{D_n}(t)\|_{L^2(\Gamma)}^2\big) \,\mathrm{d}t \\
&\leq C(T),\qquad \forall\,  D_n\geq 1.
\end{align*}
Following the argument in \cite{AK20}, we can deduce that $u(t)\in \mathbb{R}$ for all $t\in[0,T]$ and  $\langle u^{D_n}\rangle_\Omega$ are uniformly bounded in $H^1(0,T)$. Thanks to the compact embedding $H^1(0,T)\hookrightarrow\hookrightarrow C([0,T])$, we get
\begin{align}
 \|\langle u^{D_n} \rangle_\Omega\to u\|_{C([0,T])}\to 0,\qquad \text{as}\ \ D_n\to +\infty.\label{ucmean}
\end{align}
By the trace theorem and Poincar\'e's inequality, we find
\begin{align}
 \int_0^T \|u^{D_n}-u\|_{H^\frac12(\Gamma)}^2 \,\mathrm{d}t
 &\leq C \int_0^T \|u^{D_n}-u\|_{H^1(\Omega)}^2 \,\mathrm{d}t\non\\
 &\leq C\int_0^T\big(\|u^{D_n}-\langle u^{D_n}\rangle\|_{H^1(\Omega)}^2+  |\Omega|| \langle u^{D_n}\rangle-u|^2\big) \,\mathrm{d}t\non\\
 &\leq C\int_0^T\big(\|\nabla u^{D_n} \|_{\bm{L}^2(\Omega)}^2+ |\langle u^{D_n}\rangle -u|^2\big) \,\mathrm{d}t\non\\
 &\to 0,\qquad \text{as}\ \ D_n\to +\infty,\non
 \end{align}
which implies that
$$
u^{D_n}\to u\quad \text{strongly in }\ \ L^2(0,T;H^\frac12(\Gamma)).
$$
For $v^{D_n}$, the Aubin-Lions lemma also yields
\begin{align*}
& v^{D_n} \to v\ \ \text{strongly in}\ \ L^2(0,T;H^\frac12(\Gamma)),\quad \text{as}\ \ D_n\to +\infty,
\end{align*}
Thus, thanks to the assumption \textbf{(H4)} and the continuous property of Nemytskii operators (see e.g., \cite[Theorem 1.27]{RoT05}), we can conclude
\begin{align*}
q^{D_n}\to B_1u -B_1\widetilde{h}(u)v-B_2 v\quad \text{strongly in}\ \ L^2(0,T;L^2(\Gamma)).
\end{align*}

Based on the above results, we are able to pass to the limit as $D_n\to +\infty$ (up to a subsequence) in the equations for $(u^{D_n},\varphi^{D_n},\mu^{D_n},v^{D_n},\eta^{D_n})$
(e.g., in the weak formulations  \eqref{w1.a}--\eqref{w1.f} with $D=D_n$).
In particular, in \eqref{w1.a}, we should take the test function $\xi\in L^2(0,T)$ to be spatially constant. Concerning the initial data, from \cite[Lemma 3.1.7]{Z04}, we can conclude
\begin{align*}
& \vp^{D_n}(0) \rightharpoonup \vp(0)\ \ \text{and}\ \ v^{D_n}(0) \rightharpoonup v(0)\quad \text{weakly in}\ \ (H^1(\Gamma))'.
\end{align*}
Since $\vp,\,v\in C([0,T];L^2(\Gamma))$, then their initial values fulfil $\vp(\cdot,0)=\vp_0(\cdot)$ and $v(\cdot,0)=v_0(\cdot)$ in $L^2(\Gamma)$.
Moreover, it easily follows from \eqref{ucmean} that $u(0)=\langle u_0 \rangle_\Omega$ is indeed a constant.

In summary, the limit function $(u,\vp,v,\mu,\eta)$ is a weak solution to the reduced problem \eqref{r1.a}--\eqref{rini2} on $[0,T]$ in the sense of Definition \ref{def:rweak} with now $u|_{t=0}=\langle u_0 \rangle_\Omega$.

The proof of Proposition \ref{prop:LD} is complete.
\end{proof}

\medskip

\textbf{Proof of Theorem \ref{thm:rweak}}. The proof consists of several steps. \smallskip

\textbf{Step 1. Existence via large cytosolic diffusion limit}.
For any given initial data $(u_0,\,\vp_0,\,v_0)$ satisfying \eqref{rcini1}--\eqref{rcini2}, we consider the full bulk-surface coupled system \eqref{1.a}--\eqref{1.b} subject to the same initial conditions, with $D>0$ and $q$ taking the modified mass exchange term \eqref{q2m} under the assumption \textbf{(H4)}. In particular, in the cut-off function $\widetilde{h}$, we assume that
\begin{align}
h_0=\frac{2M}{|\Omega|}.\label{h0}
\end{align}
Let $T>0$ be fixed. Thanks to Theorem \ref{thm:weak}-(1), for every $D>0$, problem \eqref{1.a}--\eqref{ini2} admits a weak solution on $[0,T]$ in the sense of Definition \ref{def:weak}, which is denoted by $ (u^{D},\varphi^{D},\mu^{D},v^{D},\eta^{D})$. Next, according to Proposition \ref{prop:LD}, we can find an unbounded increasing sequence $\{D_n\}_{n\in\mathbb{Z}^+}$ with $D_n>0$ such that
$\{(u^{D_n},\varphi^{D_n},\mu^{D_n},v^{D_n},\eta^{D_n})\}$ is convergent and its limit $(u,\vp,v,\mu,\eta)$ is a weak solution to the reduced problem \eqref{r1.a}--\eqref{rini2} (with the same mass exchange term \eqref{q2m}) on $[0,T]$ in the sense of Definition \ref{def:rweak} subject to the same initial data (note that here $u_0$ is a constant given by \eqref{rcini2}).

Recall the key observation in \cite[Section 3.1]{GKRR16} (see also \cite[(2.24)]{AK20}), we infer from \eqref{r1.a} and \eqref{q2m}  that
\begin{align}
\frac{\mathrm{d}}{\mathrm{d}t} u(t) =
-\frac{B_1|\Gamma|}{|\Omega|} u(t) +  \big[B_1\widetilde{h}(u(t))+B_2\big]\Big(\frac{M}{|\Omega|}-u(t)\Big).
\label{umODE}
\end{align}
Thanks to the assumption \textbf{(H4)}, the above equation implies that if $0\leq u_0\leq M/|\Omega|$, then
\begin{align}
u(t)\in \Big[0,\ \frac{M}{|\Omega|}\Big],\qquad \forall\, t\in [0,T]. \label{Lifu}
\end{align}
Hence, by our choice \eqref{h0}, we find that $(u,\vp,v,\mu,\eta)$ is indeed a weak solution to the reduced problem \eqref{r1.a}--\eqref{rini2} with the original mass exchange term \eqref{q2} on $[0,T]$. \medskip

\textbf{Step 2. Uniqueness}.
To prove  uniqueness, we adapt the argument in Section \ref{sec:uniq}.
Let $(u_i,\varphi_i,v_i,\mu_i,\eta_i)$, $i=1,2$, be two global weak solutions to problem \eqref{r1.a}--\eqref{rini2} on $[0, T]$ subject to the corresponding initial data $(u_{0i},\varphi_{0i},v_{0i})$, $i=1,2$.
Consider the difference of two solutions:
$$
u=u_1-u_2,\qquad \varphi=\varphi_1-\varphi_2,\qquad v=v_1-v_2.
$$
We note that
\begin{align}
u=-\frac{1}{|\Omega|} \int_\Gamma v\,\mathrm{d}S \qquad \text{yields} \qquad |u|\leq \frac{|\Gamma|^\frac12}{|\Omega|}\|v\|_{L^2(\Gamma)}. \label{duu}
\end{align}
Thus, it holds
\begin{align}
\|q_1-q_2\|_{L^2(\Gamma)}
&\leq C\big(\|u\|_{L^2(\Gamma)}+ \|v\|_{L^2(\Gamma)}
+ |u|\|v_1\|_{L^2(\Gamma)}+ |u_2|\|v\|_{L^2(\Gamma)}\big)
\leq C\|v\|_{L^2(\Gamma)},\non
\end{align}
where $C>0$ depends on $M$, $\Omega$, $\Gamma$ and $\|v_1\|_{L^2(0,T;L^2(\Gamma))}$.
Taking $(-\Delta_\Gamma)^{-1} (\varphi-\langle \vp\rangle_\Gamma)$ and $\delta v$ as test functions for the equations of the differences $\vp$ and $v$, respectively, by a similar argument for \eqref{uniq2}, we get
\begin{align}
&\frac12\frac{\mathrm{d}}{\mathrm{d}t}\left(\|\varphi-\langle \vp\rangle_\Gamma \|_{(H^1(\Gamma))'}^2+ \delta\|v\|_{L^2(\Gamma)}^2\right)
+ \|\nabla_\Gamma \vp\|_{\bm{L}^2(\Gamma)}^2 + 4 \|\nabla_\Gamma v\|_{\bm{L}^2(\Gamma)}^2
\non\\
&\quad = - \int_\Gamma \left(\int_0^1W''(s\vp_1+(1-s)\vp_2) \varphi \,\mathrm{d}s\right)\vp\,\mathrm{d}S
+ \langle \vp\rangle_\Gamma\int_\Gamma (W'(\vp_1)-W'(\vp_2)) \,\mathrm{d}S
 \non\\
&\qquad +\frac{1}{\delta}\int_\Gamma (2v-\vp)(\varphi-\langle \vp\rangle_\Gamma)\,\mathrm{d}S +2\int_\Gamma \nabla_\Gamma\varphi\cdot\nabla_\Gamma v\,\mathrm{d}S + \delta \int_\Gamma (q_1-q_2)v\,\mathrm{d}S\non\\
&\quad \leq  (\theta_0-\theta)\|\vp\|_{L^2(\Gamma)}^2
+ \big(\|W'(\vp_1)\|_{L^1(\Gamma)}+\|W'(\vp_2)\|_{L^1(\Gamma)}\big)|\langle \vp\rangle_\Gamma|
+ \frac{2}{\delta}\|v\|_{L^2(\Gamma)}\|\vp\|_{L^2(\Gamma)}
\non\\
&\qquad + 2\|\nabla_\Gamma\varphi\|_{\bm{L}^2(\Gamma)}\|\nabla_\Gamma v\|_{\bm{L}^2(\Gamma)} + C\|v\|_{L^2(\Gamma)}^2 \non\\
&\quad \leq   \frac34 \|\nabla_\Gamma\varphi\|_{\bm{L}^2(\Gamma)}^2
+ 2\|\nabla_\Gamma v\|_{\bm{L}^2(\Gamma)}^2
+ C\big(\|\varphi-\langle \vp\rangle_\Gamma\|_{(H^1(\Gamma))'}^2+ |\langle \vp\rangle_\Gamma|^2
+ \|v\|_{L^2(\Gamma)}^2\big)
 \non\\
&\qquad
+ \big(\|W'(\vp_1)\|_{L^1(\Gamma)}+\|W'(\vp_2)\|_{L^1(\Gamma)}\big)|\langle \vp\rangle_\Gamma|,
\label{uniq2a}
\end{align}
We recall that
$\langle \vp(t)\rangle_\Gamma= \langle\varphi_{01}-\varphi_{02}\rangle_\Gamma$ and $\|W'(\vp_i)\|_{L^1(\Gamma)}\in L^1(0,T)$ for weak solutions.
Then applying Gronwall's lemma, we find
\begin{align}
&\|\varphi(t) \|_{(H^1(\Gamma))'}^2+ \delta\|v(t)\|_{L^2(\Gamma)}^2 + \int_0^t\big( \|\nabla_\Gamma \vp(\tau)\|_{\bm{L}^2(\Gamma)}^2 +  \|\nabla_\Gamma v(\tau)\|_{\bm{L}^2(\Gamma)}^2\big)\,\mathrm{d}\tau\non\\
& \quad \leq C(T)\left(\|\varphi(0)\|_{(H^1(\Gamma))'}^2+ \delta\|v(0)\|_{L^2(\Gamma)}^2+ |\langle \vp(0)\rangle_\Gamma|\right),\qquad \forall\,t\in[0,T],
\label{conti1a}
\end{align}
which  yields the uniqueness of $(\vp,v)$. Then it follows from \eqref{duu} that $u$ is also unique. \medskip

\textbf{Step 3. Uniform-in-time estimates}. Since in the previous steps the final time $T>0$ is arbitrary, thanks to the uniqueness result, we can uniquely extend the weak solution on the finite interval to the positive half line $[0,+\infty)$.

Let us now derive uniform-in-time estimates for the global weak solution $(u,\vp,v,\mu,\eta)$. First, the estimate \eqref{Lifu} holds for all $t\geq 0$ and this implies
\begin{align}
\int_\Gamma v(t)\,\mathrm{d}S \in \left[0,  M \right],\qquad \forall\, t\geq 0. \label{Lifv}
\end{align}
Besides, we have
\begin{align}
\|\varphi\|_{L^\infty(0,+\infty;L^\infty(\Gamma))}\leq 1.\label{Lifvp}
\end{align}
Testing \eqref{rew1.c} by $\mu$ and \eqref{rew1.e} by $(2/\delta)(2v-\vp)$, adding the resultants together, we get
\begin{align}
& \frac{\mathrm{d}}{\mathrm{d}t} \Big(\frac12\|\nabla_\Gamma \vp\|_{\bm{L}^2(\Gamma)}^2 +\int_\Gamma W(\vp)\,\mathrm{d}S + \frac{2}{\delta} \|v\|_{L^2(\Gamma)}^2 -\frac{2}{\delta} \int_\Gamma \vp v\,\mathrm{d}S +\frac{1}{2\delta}\|\vp\|_{L^2(\Gamma)}^2\Big)
\non\\
& \qquad + \|\nabla_\Gamma \mu\|^2_{\bm{L}^2(\Gamma)} + \frac{4}{\delta^2}\|\nabla_\Gamma(2v-\vp)\|_{\bm{L}^2(\Gamma)}^2\non\\
&\quad = \frac{2}{\delta} \int_\Gamma (2v-\vp)\big[B_1u(1-v)-B_2v\big]\,\mathrm{d}S.
\label{unirE1}
\end{align}
Using the assumption $B_1,B_2>0$, the estimate for $u$, \eqref{Lifv} and \eqref{Lifvp}, Poincar\'e's inequality and Young's inequality, we can handle the right-hand side of \eqref{unirE1} as follows
\begin{align}
& \frac{2}{\delta} \int_\Gamma (2v-\vp)\big[B_1u(1-v)-B_2v\big]\,\mathrm{d}S\non\\
&\quad = \frac{4B_1}{\delta} u \int_\Gamma v\,\mathrm{d}S - \frac{4B_1}{\delta} u \int_\Gamma v^2 \,\mathrm{d}S - \frac{4B_2}{\delta} \int_\Gamma v^2 \,\mathrm{d}S  -\frac{2}{\delta}\int_\Gamma \vp\big[B_1u(1-v)-B_2v\big]\,\mathrm{d}S\non\\
&\quad \leq \frac{4B_1}{\delta|\Omega|} \Big(M- \int_\Gamma v\,\mathrm{d}S\Big) \int_\Gamma v\,\mathrm{d}S +
\frac{2B_1|\Gamma|}{\delta}u + \frac{2}{\delta}(B_1u+B_2)\int_\Gamma |v| \,\mathrm{d}S \non\\
&\quad \leq \frac{B_1M^2}{\delta|\Omega|} +  \frac{2B_1M|\Gamma|}{\delta|\Omega|} + \frac{|\Gamma|}{\delta}\Big(\frac{B_1M}{|\Omega|}+B_2\Big)
+ \frac{1}{\delta}\Big(\frac{B_1M}{|\Omega|}+B_2\Big)\int_\Gamma |2v-\vp| \,\mathrm{d}S\non\\
&\quad \leq C_1+ C_1 \|2v-\vp\|_{L^1(\Gamma)}\non\\
&\quad \leq C_1+ C_2 \Big(\|\nabla_\Gamma(2v-\vp)\|_{\bm{L}^2(\Gamma)}+ \Big|\int_\Gamma (2v-\vp)\,\mathrm{d}S\Big|\Big)\non\\
&\quad \leq  \frac{2}{\delta^2}\|\nabla_\Gamma(2v-\vp)\|_{\bm{L}^2(\Gamma)}^2 + C_3,\label{unirE2}
\end{align}
where the positive constants $C_1,C_2,C_3$ only depend on $M$, $B_1$, $B_2$, $\Omega$, $\Gamma$ and $\delta$.

Next, testing \eqref{rew1.c} by $(-\Delta_\Gamma)^{-1}(\vp-\langle\vp\rangle_\Gamma)$, we get
\begin{align}
& \frac12\frac{\mathrm{d}}{\mathrm{d}t} \|\vp-\langle\vp\rangle_\Gamma\|_{(H^1(\Gamma))'}^2 +\|\nabla_\Gamma \vp\|_{\bm{L}^2(\Gamma)}^2 + \int_\Gamma F'(\vp)(\vp-\langle\vp\rangle_\Gamma)\,\mathrm{d}S- \frac{2}{\delta}\int_\Gamma \vp v\,\mathrm{d}S \non\\
&\quad = \left(\theta_0-\frac{1}{\delta}\right)\int_\Gamma \vp(\vp-\langle\vp\rangle_\Gamma)\,\mathrm{d}S + \frac{2}{\delta} \langle\vp\rangle_\Gamma \int_\Gamma  v\,\mathrm{d}S\non\\
&\quad \leq \left|\theta_0-\frac{1}{\delta}\right|\|\vp-\langle\vp\rangle_\Gamma\|_{L^1(\Gamma)}
+ \frac{2M}{\delta}\non\\
&\quad \leq C_4 \|\nabla_\Gamma \vp\|_{\bm{L}^2(\Gamma)}  + \frac{2M}{\delta}\non\\
&\quad \leq \frac12 \|\nabla_\Gamma \vp\|_{\bm{L}^2(\Gamma)}^2 +C_5,
\label{unirE3}
\end{align}
where the positive constants $C_4,C_5$ only depend on $M$, $\theta_0$, $\Gamma$ and $\delta$.
By virtue of the assumption \textbf{(H1)}, we have
$$
F(r)\leq F(s)+F'(r)(r-s)-\frac{\theta}{2}(r-s)^2,\qquad \forall\,r,s\in (-1,1),
$$
which combined with the mass conservation $\langle\varphi \rangle_\Gamma=\langle\varphi_0 \rangle_\Gamma\in (-1,1)$ implies
\begin{align}
\int_\Gamma F  (\varphi)\,\mathrm{d}S
& \leq |\Gamma|F (\langle\varphi_0 \rangle_\Gamma) + \int_\Gamma  F'(\varphi) (\varphi - \langle\varphi \rangle_\Gamma)\,\mathrm{d}S  - \frac{\theta}{2} \|\varphi - \langle\varphi \rangle_\Gamma\|_{L^2(\Gamma)}^2.
\label{FFb}
\end{align}
Then we infer from \eqref{unirE3}, \eqref{FFb} and \eqref{Lifvp} that
\begin{align}
& \frac12\frac{\mathrm{d}}{\mathrm{d}t} \|\vp-\langle\vp\rangle_\Gamma\|_{(H^1(\Gamma))'}^2 +\frac12 \|\nabla_\Gamma \vp\|_{\bm{L}^2(\Gamma)}^2 + \int_\Gamma F  (\varphi)\,\mathrm{d}S - \frac{2}{\delta}\int_\Gamma \vp v\,\mathrm{d}S + \frac{\theta}{2} \|\varphi - \langle\varphi \rangle_\Gamma\|_{L^2(\Gamma)}^2\non\\
&\quad \leq C_5 + |\Gamma|F (\langle\varphi_0 \rangle_\Gamma) \leq C_6,
\label{unirE4}
\end{align}
where $C_6>0$ only depend on $M$, $\theta_0$, $\Gamma$, $\delta$ and $\langle\varphi_0 \rangle_\Gamma$.

Combining \eqref{unirE1}, \eqref{unirE2} and \eqref{unirE4}, we find the following differential inequality
\begin{align}
& \frac{\mathrm{d}}{\mathrm{d}t} \mathcal{G}(\vp,v) +  C_7\mathcal{G}(\vp,v)+ \|\nabla_\Gamma \mu\|^2_{\bm{L}^2(\Gamma)} + \frac{2}{\delta^2}\|\nabla_\Gamma(2v-\vp)\|_{\bm{L}^2(\Gamma)}^2  \leq C_8,
\label{unirE5}
\end{align}
where
\begin{align}
 \mathcal{G}(\vp,v)
 &= \frac12\|\nabla_\Gamma \vp\|_{\bm{L}^2(\Gamma)}^2
 +\int_\Gamma W(\vp)\,\mathrm{d}S
 + \frac{2}{\delta} \|v\|_{L^2(\Gamma)}^2
 -\frac{2}{\delta} \int_\Gamma \vp v\,\mathrm{d}S\non\\
 &\quad  +\frac{1}{2\delta}\|\vp\|_{L^2(\Gamma)}^2
 + \frac12\|\vp-\langle\vp\rangle_\Gamma\|_{(H^1(\Gamma))'}^2,
\label{G}
\end{align}
$C_7, C_8$ are positive constants that may depend on $M$, $B_1$, $B_2$, $\theta$, $\theta_0$, $\Omega$, $\Gamma$, $\delta$  and $ \langle\varphi_0 \rangle_\Gamma$. Then we obtain the
the dissipative estimate
\begin{align}
 &\mathcal{G}(\vp(t),v(t)) +\int_{t}^{t+1} \big(\|\nabla_\Gamma \mu(\tau)\|^2_{\bm{L}^2(\Gamma)} + \|\nabla_\Gamma(2v(\tau)-\vp(\tau))\|_{\bm{L}^2(\Gamma)}^2\,\mathrm{d}\tau\non\\
 &\quad \leq \mathcal{G}(\vp_0,v_0)e^{-C_7 t} + C_9,\qquad  \forall\, t\geq 0, \label{es-diss1}
\end{align}
where the positive constant $C_9$ only depends on $C_7, C_8$. From the assumption \textbf{(H1)} and \eqref{Lifvp}, we find that
\begin{align}
\mathcal{G}(\vp,v)\geq C_{10}\big(\|\vp\|_{H^1(\Gamma)}^2+ \|v\|_{L^2(\Gamma)}^2\big)-C_{11}, \label{es-below}
\end{align}
where $C_{10}, C_{11}$ are positive constants that may depend on $\theta$, $\Gamma$, $\delta$ and $\min_{r\in[-1,1]}F(r)$. Form \eqref{es-diss1} and \eqref{es-below} we arrive at the conclusion \eqref{es-dissA}.

The proof of Theorem \ref{thm:rweak} is complete.
\hfill $\square$.

\subsection{Existence of the global attractor}

First, we show that the unique global weak solution to the reduced system
\eqref{nr1.c}--\eqref{nr1.f} generates a dynamical system on the phase space $\mathcal{V}_{M,m}$. \medskip

  \textbf{Proof of Corollary \ref{cor:DS}}.  For any initial datum $(\vp_0,v_0)\in \mathcal{V}_{M,m}$, it follows from Theorem \ref{thm:rweak} that the reduced system
\eqref{nr1.c}--\eqref{nr1.f} admits a unique global weak solution $(\vp,v,\mu,\eta)$ in the sense of Definition \ref{def:weak} (eliminating the variable $u$ therein). Thus, the operator
$$
\mathcal{S}(t):\ \mathcal{V}_{M,m}\to \mathcal{V}_{M,m} \quad \text{such that}\quad \mathcal{S}(t)(\vp_0,v_0)=(\vp,v),\quad \forall\, t\geq 0,
$$
is well defined.
Moreover, we find
\begin{itemize}
\item[(1)] $\mathcal{S}(0)=\mathcal{I}$, i.e., the identity operator;
\item[(2)] $\mathcal{S}(t+\tau)=\mathcal{S}(t)\mathcal{S}(\tau)$ for any $t,\tau\geq 0$, thanks to the uniqueness of weak solutions for the autonomous system;
\item[(3)]  $t\to \mathcal{S}(t)(\vp_0,v_0)\in C([0,+\infty);H^1(\Gamma)\times L^2(\Gamma))$, i.e., the strong continuity in time.
\end{itemize}

   Below we verify the property \eqref{conti}. The case $t=0$ is obvious. To show the estimate for $t>0$, we need the following regularity result:
\begin{lemma}\label{lem:comp}
Let the assumptions of Theorem \ref{thm:rweak} be satisfied.
For any $t_0>0$, we have
\begin{align}
& \|\varphi\|_{ L^{\infty}(t_0,+\infty;H^3(\Gamma))}
+ \|\mu\|_{ L^\infty(t_0,+\infty;H^1(\Gamma))} +\|v\|_{L^{\infty}(t_0,+\infty;H^1(\Gamma))}\leq C(t_0),
\label{nunit-es1}
\end{align}
with some $C(t_0)>0$ and the strict separation property
\begin{align}
\|\vp(t)\|_{C(\Gamma)}\leq 1-\sigma_{t_0},\qquad \forall\,t\geq t_0,
\label{nvp-sep0}
\end{align}
with some $\sigma_{t_0}\in (0,1)$. Both constants $C(t_0)$ and $\sigma_{t_0}$ may depend on $\|\varphi_0\|_{H^1(\Gamma)}$, $\|v_0\|_{L^2(\Gamma)}$, $\int_\Gamma F(\varphi_0)\,\mathrm{d}S$, $\langle \varphi_0\rangle_\Gamma$, $\Omega$, $\Gamma$, coefficients of the system and $t_0$, but are independent of time $t$.
\end{lemma}
\begin{proof}
Based on the uniform-in-time lower-order estimate \eqref{es-dissA}, we can prove the conclusions \eqref{nunit-es1} and \eqref{nvp-sep0} by an argument analogous to that for the full bulk-surface coupled system \eqref{1.a}--\eqref{1.f} done in Section \ref{sec:rec} and thus we omit the details here.
\end{proof}

Following the argument in \cite{GGM2017}, for any given initial datum $(\vp_0,v_0)\in \mathcal{V}_{M,m}$, we  consider the sequence $\big\{(\vp_0^{(n)},v_0^{(n)})\big\}\subset  \mathcal{V}_{M,m}$ such that
$$
\lim_{n\to +\infty}\mathrm{dist}_{\mathcal{V}_{M,m}} \big((\vp_0^{(n)},v_0^{(n)}),(\vp_0,v_0)\big)=0.
$$
Fix $\widehat{t}>0$. In light of Lemma \ref{lem:comp}, the corresponding global weak solutions denoted by
$$
(\vp(t),v(t))=\mathcal{S}(t)(\vp_0,v_0)\quad \text{and}\quad (\vp^{(n)}(t),v^{(n)}(t))=\mathcal{S}(t)(\vp_0^{(n)},v_0^{(n)})
$$
satisfy
\begin{align*}
& \|\varphi(\widehat{t})\|_{H^3(\Gamma)} +\|v(\widehat{t})\|_{H^1(\Gamma)}+  \|\varphi^{(n)}(\widehat{t})\|_{ H^3(\Gamma)} +\|v^{(n)}(\widehat{t})\|_{H^1(\Gamma)}\leq C,
\end{align*}
where $C>0$ is independent of $n$. Then by interpolation and the continuous dependence estimate \eqref{conti1a}, we have
\begin{align}
&\mathrm{dist}_{\mathcal{V}_{M,m}} \big(\big(\vp^{(n)}(\widehat{t}),v^{(n)}(\widehat{t})\big),\big(\vp(\widehat{t}),v(\widehat{t})\big)\big)\non\\
&\quad \leq C\big(\|\vp^{(n)}(\widehat{t})-\vp(\widehat{t})\|_{H^3(\Gamma)}^\frac12 \|\vp^{(n)}(\widehat{t})-\vp(\widehat{t})\|_{(H^1(\Gamma))'}^\frac12 +
\|v^{(n)}(\widehat{t})-v(\widehat{t})\|_{L^2(\Gamma)}\big)\non\\
&\quad \leq C \big(\|\vp_0^{(n)}-\vp_0\|_{(H^1(\Gamma))'}^\frac12+ \delta\|v_0^{(n)}-v_0\|_{L^2(\Gamma)}^\frac12+ |\langle \vp_0^{(n)}-\vp_0\rangle_\Gamma|^\frac14\big)\non\\
&\qquad + \big(\|\vp_0^{(n)}-\vp_0\|_{(H^1(\Gamma))'} + \delta\|v_0^{(n)}-v_0\|_{L^2(\Gamma)} + |\langle \vp_0^{(n)}-\vp_0\rangle_\Gamma|^\frac12\big)\non\\
&\quad \to 0,\qquad \text{as}\ n\to +\infty,\non
\end{align}
which yields the conclusion  \eqref{conti}.

The proof of Corollary \ref{cor:DS} is complete.
\hfill $\square$ \medskip

We are now in a position to prove Theorem \ref{thm:att}.\medskip

 \textbf{Proof of Theorem \ref{thm:att}}.
For the dynamical system $(\mathcal{S}(t),\mathcal{V}_{M,m})$, we  infer from the dissipative estimate \eqref{es-dissA} that it
admits a bounded absorbing set.
\begin{lemma}[Absorbing set]\label{lem:abs}
There exists $R_*>0$ such that the closed ball $\mathcal{B}_*\subset \mathcal{V}_{M,m}$ centered at zero with radius $R_*$ is absorbing for the semigroup $\mathcal{S}(t)$. Namely, for every bounded
set $\mathcal{B}\subset \mathcal{V}_{M,m}$, there exists a time
$t_{\mathcal{B}}>0$ such that
$$
\mathcal{S}(t)\mathcal{B}\subset \mathcal{B}_*,\qquad \forall\, t\geq t_{\mathcal{B}}.
$$
\end{lemma}

  Next, thanks to the compact embeddings $H^3(\Gamma)\hookrightarrow\hookrightarrow H^1(\Gamma) \hookrightarrow\hookrightarrow L^2(\Gamma)$, Lemma \ref{lem:comp} implies that for every bounded
set $\mathcal{B}\subset \mathcal{V}_{M,m}$,
$$
\bigcup_{t\geq 1} \mathcal{S}(t)\mathcal{B}\quad\text{is relatively compact in}\ \mathcal{V}_{M,m}.
$$

Hence, from the classical theory of infinite dimensional dynamical systems (see for instance, \cite[Theorem I.1.1]{T}), we can conclude that the dynamical system $(\mathcal{S}(t),\mathcal{V}_{M,m})$ has a global attractor $\mathcal{A}_{M,m}$, which is bounded in $H^3(\Gamma)\times H^1(\Gamma)$. By its definition, the global attractor is unique. Moreover, because of the continuity property  of $\mathcal{S}(t)$,  $\mathcal{A}_{M,m}$ is connected.

The proof of Theorem \ref{thm:att} is complete.
\hfill $\square$

\section*{Acknowledgments}
\noindent The research of Wu was partially supported by National Natural Science Foundation of China under Grant number 12071084.

\section*{Declarations}
\noindent \textbf{Competing interests}. The authors declare that they have no conflict of interest.

\noindent \textbf{Use of AI tools declaration}. The authors declare they have not used Artificial Intelligence (AI) tools in the creation of this article.



\begin{thebibliography}{99}
\itemsep=-1pt
\bibitem{A2009}
H. Abels,
On a diffuse interface model for two-phase flows of viscous, incompressible fluids with matched densities,
Arch. Ration. Mech. Anal., \textbf{194} (2009), 463--506.

\bibitem{AK20}
H. Abels and J. Kampmann,
On a model for phase separation on biological membranes and its relation to the Ohta-Kawasaki equation,
European J. Appl. Math., \textbf{31} (2020), 297--338.

\bibitem{AK21a}
H. Abels and J. Kampmann,
Existence of weak solutions for a sharp interface model for phase separation on biological membranes,
Discrete Contin. Dyn. Syst. Ser. S, \textbf{14} (2021), 331--351.


\bibitem{AK21b}
H. Abels and J. Kampmann,
On the sharp interface limit of a model for phase separation on biological membranes,
Analysis (Berlin), \textbf{41} (2021), 37--60.


\bibitem{AW07}
H. Abels and M. Wilke,
Convergence to equilibrium for the Cahn-Hilliard equation with a
logarithmic free energy,
Nonlinear Anal., \textbf{67} (2007), 3176-3193.

\bibitem{Aubin82}
T. Aubin,
\emph{Nonlinear Analysis on Manifolds, Monge-Amp\'{e}re Equations},
Springer, New York, 1982.

\bibitem{Ben}
I. Ben Hassen,
Decay estimates to equilibrium for some asymptotically autonomous semilinear evolution equations,
Asymptot. Anal., \textbf{69} (2010), 31--44.

\bibitem{BL00}
D. A. Brown and E. London,
Structure and function of sphingolipid- and cholesterol-rich membrane rafts,
J. Biol. Chem., \textbf{275} (2000), 17221--17224.

\bibitem{CE21}
D. Caetano and C. M. Elliott,
Cahn-Hilliard equations on an evolving surface,
European J. Appl. Math., \textbf{32} (2021), 937--1000.

\bibitem{CEGP}
D. Caetano, C. M. Elliott, M. Grasselli and A. Poiatti,
Regularisation and separation for evolving surface Cahn-Hilliard equations,
(2022), arXiv:2205.09822.

\bibitem{CH}
J. Cahn and J. Hilliard,
Free energy of a nonuniform system I. Interfacial free energy,
J. Chem. Phys., \textbf{28} (1958), 258--267.

\bibitem{CMZ}
L. Cherfils, A. Miranville and S. Zelik,
The Cahn-Hilliard equation with logarithmic potentials,
Milan J. Math., \textbf{79} (2011), 561--596.

\bibitem{CFG}
P. Colli, S. Frigeri, M. Grasselli,
Global existence of weak solutions to a nonlocal Cahn-Hilliard-Navier-Stokes system,
J. Math. Anal. Appl., \textbf{386} (2012), 428--444.

\bibitem{CKRS}
P. Colli, P. Krej\v{c}\'{\i}, E. Rocca and J. Sprekels,
Nonlinear evolution inclusions arising from phase change models,
Czechoslovak Math. J., \textbf{57} (2007), 1067--1098.

\bibitem{DD95}
A. Debussche and L. Dettori,
On the Cahn-Hilliard equation with a logarithmic free energy,
Nonlinear Anal., \textbf{24} (1995), 1491--1514.

\bibitem{DE13}
G. Dziuk and C. M. Elliott,
Finite element methods for surface PDEs,
Acta Numer., \textbf{22} (2013), 289--396.

\bibitem{ER15}
C. M. Elliott and T. Ranner,
Evolving surface finite element method for the Cahn-Hilliard equation,
Numer. Math., \textbf{129} (2015), 483--534.

\bibitem{ES10a}
C. M. Elliott and B. Stinner,
Modeling and computation of two phase geometric biomembranes using surface finite elements, J. Comput. Phys., \textbf{229} (2010), 6585--6612.

\bibitem{ES10b}
C. M. Elliott and B. Stinner,
A surface phase field model for two-phase biological membranes,
SIAM J. Appl. Math., \textbf{70} (2010), 2904--2928.

\bibitem{EDF13}
A. Embar, J. Dolbow and E. Fried,
Microdomain evolution on giant unilamellar vesicles,
Biomech. Model. Mechanobiol., \textbf{12} (2013), 597--615.

\bibitem{FSH10}
J. Fan, M. Sammalkorpi and M. Haataja,
Formation and regulation of lipid microdomains in cell membranes: Theory, modeling, and speculation,
FEBS Lett., \textbf{584} (2010), 1678--1684.

\bibitem{FS}
E. Feireisl and F. Simondon,
Convergence for semilinear degenerate parabolic equations in several space dimensions,
J. Dynam. Differential Equations, \textbf{12} (2000), 647--673.

\bibitem{Flo42}
P. J. Flory,
Thermodynamics of high polymer solutions,
J. Chem. Phys., \textbf{10} (1942), 51--61.

\bibitem{Fo05}
L. Foret,
A simple mechanism of raft formation in two-component fluid membranes,
Europhys. Lett., \textbf{71} (2005), 508--514.

\bibitem{FG12}
S. Frigeri and M. Grasselli,
Nonlocal Cahn-Hilliard-Navier-Stokes systems with singular potentials,
Dyn. Partial Differ. Equ., \textbf{9} (2012), 273--304.

\bibitem{GKRR16}
H. Garcke, J. Kampmann, A. R\"{a}tz and M. R\"{o}ger,
A coupled surface-Cahn-Hilliard bulk-diffusion system modeling lipid raft formation in cell membranes,
Math. Models Methods Appl. Sci., \textbf{26} (2016), 1149--1189.

\bibitem{GL17e}
H. Garcke and K.-F. Lam,
Well-posedness of a Cahn-Hilliard system modelling tumour growth with chemotaxis and active transport,
European J. Appl. Math., \textbf{28} (2017), 284--316.

\bibitem{GLSS}
H. Garcke, K.-F. Lam, E. Sitka and V. Styles,
A Cahn-Hilliard-Darcy model for tumour growth with chemotaxis and active transport,
Math. Models Methods Appl. Sci., \textbf{26} (2016), 1095--1148.

\bibitem{GT01}
D. Gilbarg and N.S. Trudinger,
\emph{Elliptic Partial Differential Equations of Second Order},
Classics in Mathematics, Springer-Verlag, Berlin, 2001. Reprint of the 1998 edition.

\bibitem{GGW23}
C. G. Gal, A. Giorgini and M. Grasselli,
The separation property for 2D Cahn-Hilliard equations: local, nonlocal and fractional energy cases,
Discrete Contin. Dyn. Syst., \textbf{43} (2023), 2270--2304.
	
\bibitem{GGM2017}
A. Giorgini, M. Grasselli and A. Miranville,
The Cahn-Hilliard-Oono equation with singular potential, Math. Models Meth. Appl. Sci., \textbf{27} (2017), 2485--2510.

\bibitem{GGW18}
A. Giorgini, M. Grasselli and H. Wu,
The Cahn-Hilliard-Hele-Shaw system with singular potential,
Ann. Inst. H. Poincar\'e Anal. Non Lin\'eaire, \textbf{35} (2018) 1079--1118.

\bibitem{GMT}
A. Giorgini, A. Miranville and R. Temam,
Uniqueness and regularity for the Navier-Stokes-Cahn-Hilliard system,
SIAM J. Math. Anal., \textbf{51} (2019), 2535--2574.

\bibitem{GSR08}
J. G\'{o}mez, F. Sagu\'{e}s and R. Reigada,
Actively maintained lipid nanodomains in biomembranes,
Phys. Rev. E, \textbf{77} (2008), 021907.

\bibitem{HJ2001}
A. Haraux and M. A. Jendoubi,
Decay estimates to equilibrium for some evolution equations with an analytic nonlinearity,
Asymptot. Anal., \textbf{26} (2001), 21--36.

\bibitem{H1}
J.-N. He and H. Wu,
Global well-posedness of a Navier-Stokes-Cahn-Hilliard system with chemotaxis and singular potential in 2D,
J. Differential Equations,  \textbf{297} (2021), 47--80.

\bibitem{HT01}
S.-Z. Huang and P. Tak\'{a}\v{c},
Convergence in gradient-like systems which are asymptotically autonomous and analytic,
Nonlinear Anal., \textbf{46} (2001), 675--698.

\bibitem{Hug41}
M. L. Huggins,
Solutions of long chain compounds,
J. Chem. Phys., \textbf{9} (1941), 440.

\bibitem{JWZ}
J. Jiang, H. Wu and S. Zheng,
Well-posedness and long-time behavior of a non-autonomous Cahn-Hilliard-Darcy system with mass source modeling tumor growth, J. Differential Equations, \textbf{259} (2015), 3032--3077.

\bibitem{KNP}
N. Kenmochi, M. Niezg\'{o}dka and I. Paw{\l}ow,
Subdifferential operator approach to the Cahn-Hilliard equation with constraint,
J. Differential Equations, \textbf{117} (1995), 320--356.

\bibitem{La}
O. A. Ladyzhenskaya and N. N. Ural'tseva,
\emph{Linear and Quasilinear Elliptic Equations}.
Translated from the Russian by Scripta Technica, Inc. Translation editor: Leon Ehrenpreis. Academic Press, New York and London, 1968.


\bibitem{LS10}
D. Lingwood and K. Simons,
Lipid rafts as a membrane-organizing principle,
Science, \textbf{327} (5961) (2010), 46--50.

\bibitem{Mi19}
A. Miranville,
\emph{The Cahn-Hilliard Equation: Recent Advances and Applications},
CBMS-NSF Regional Conference Series in Applied Mathematics, Vol. 95, SIAM, Philadelphia, 2019.

\bibitem{MZ04}
A. Miranville and S. Zelik,
Robust exponential attractors for Cahn-Hilliard type equations with singular potentials,
Math. Methods Appl. Sci., \textbf{27} (2004), 545--582.

\bibitem{MZ08}
A. Miranville and S. Zelik,
Attractors for dissipative partial differential equations
in bounded and unbounded domains, in Handbook of Differential Equations, Evolutionary Partial Differential Equations, Vol. 4, eds. C. M. Dafermos and M. Pokorny (Elsevier, 2008), pp. 103--200.

\bibitem{NO95}
Y. Nishiura and I. Ohnishi,
Some mathematical aspects of the micro-phase separation
in diblock copolymers,
Physica D, \textbf{84} (1995), 31--39.

\bibitem{OK86}
T. Ohta and K. Kawasaki,
Equilibrium morphology of block copolymer melts,
Macromolecules, \textbf{19} (1986), 2621--2632.

\bibitem{Pi06}
L. J. Pike, Rafts defined: a report on the keystone symposium on lipid rafts and cell function,
J. Lipid Res., \textbf{47} (2006), 1597--1598.

\bibitem{RV06}
A. R\"{a}tz and A. Voigt,
PDE's on surfaces - a diffuse interface approach,
Commun. Math. Sci., \textbf{4} (2006), 575--590.

\bibitem{RoT05}
T. Roub\'{\i}\v{c}ek,
\emph{Nonlinear Partial Differential Equations with Applications},
Birkh\"{a}user Verlag, Basel, 2005.


\bibitem{RH99}
P. Rybka and K.-H. Hoffmann,
Convergence of solutions to Cahn-Hilliard equation,
Commun. Partial Differential Equations, \textbf{24} (1999), 1055--1077.

\bibitem{simon}
J. Simon,
Compact sets in the space $L^p(0, T; B)$,
Ann. Mat. Pura Appl. (4), \textbf{146} (1987), 65--96.

\bibitem{LS83}
L. Simon,
Asymptotics for a class of nonlinear evolution equation with applications to geometric problems,
Ann. Math., \textbf{118} (1983), 525--571.

\bibitem{T}
R. Temam,
\emph{Infinite Dimensional Dynamical Systems in Mechanics and Physics},
Second edition, Applied Mathematical Sciences, Vol. 68, Springer-Verlag, New York, 1997.

\bibitem{WBV12}
T. Witkowski, R. Backofen and A. Voigt,
The influence of membrane bound proteins on phase separation and coarsening in cell membranes,
Phys. Chem. Chem. Phys., \textbf{14} (2012), 14509--14515.

\bibitem{YQMO19}
V. Yushutin, A. Quaini, S. Majd and M. Olshanskii,
A computational study of lateral phase separation in biological membranes,
Int. J. Numer. Meth. Bio., \textbf{35} (2019), e3181.

\bibitem{YQO20}
V. Yushutin, A. Quaini  and M. Olshanskii,
Numerical modelling of phase separation on dynamic surfaces,
J. Comput. Phys., \textbf{407} (2020), 109126.

\bibitem{Z04}
S. M. Zheng,
\emph{Nonlinear Evolution Equations},
Pitman Monographs and Surveys in Pure and Applied Mathematics, Vol. 133,
Chapman \& Hall/CRC, Boca Raton, FL, 2004.

\bibitem{ZWQ21}
A. Zhiliakov, Y. Wang, A. Quaini, M. Olshanskii and S. Majd,
Experimental validation of a phase-field model to predict coarsening
dynamics of lipid domains in multicomponent membranes,
BBA - Biomembranes, \textbf{1863} (2021), 183446.

\bibitem{ZT19}
C. Zimmermann, D. Toshniwal, C. M. Landis, T. J. R. Hughes, K. K. Mandadapu and R. A. Sauer,
An isogeometric finite element formulation for phase transitions on deforming surfaces,
Comput. Methods Appl. Mech. Engrg., \textbf{351} (2019), 441--477.

\end{thebibliography}
\end{document}